\documentclass[preprint]{elsarticle}
\usepackage{fullpage}
\usepackage[T1]{fontenc} 
\usepackage{amsmath,amssymb}
\usepackage{amsopn,amsthm}
\usepackage{subcaption}

\usepackage{graphicx} 
\usepackage{color}
\usepackage{psfrag}

\usepackage{tikz}

\usepackage{mathrsfs} 
\usepackage{bm}
\usepackage{cases}
\usepackage{multirow}

\usepackage{algpseudocode} 
\usepackage{algorithm}
\floatstyle{plain} \newfloat{myalgo}{tbhp}{mya}

\usepackage[a]{esvect}
\usepackage{float}

\usepackage{bbm}

{\begin{myalgo}[#1]
    \centering
    \begin{minipage}{0.9\textwidth}
      \begin{algorithm}[H]}%
      {\end{algorithm}
    \end{minipage}
  \end{myalgo}}

\usepackage[colorlinks=true]{hyperref}
\hypersetup{urlcolor=blue, citecolor=blue, linkcolor=blue}

\newtheorem{theorem}{Theorem}[section]
\newtheorem{proposition}[theorem]{Proposition}
\newtheorem{lemma}[theorem]{Lemma}

\theoremstyle{definition}
\newtheorem{definition}[theorem]{Definition}
\newtheorem{example}[theorem]{Example}

\theoremstyle{remark} \newtheorem{remark}[theorem]{Remark}

\makeatletter

\let\c@table\c@figure
\makeatother 

\numberwithin{equation}{section}
\numberwithin{figure}{section}
\numberwithin{table}{section}
\numberwithin{algorithm}{section}

\newcommand{\field}[1]{{\mathbb{#1}}}
\newcommand{\C}{\field{C}}
\newcommand{\N}{\field{N}}
\newcommand{\R}{\field{R}}

\newcommand{\I}{\mathbb{I}}
\newcommand{\M}{\mathbb{M}}

\newcommand{\Ccal}{\mathcal{C}}

\newcommand{\Fcal}{\mathcal{F}}
\newcommand{\Gcal}{\mathcal{G}}
\newcommand{\Hcal}{\mathcal{H}}

\newcommand{\Pcal}{\mathcal{P}}

\newcommand{\Scal}{\mathcal{S}}
\newcommand{\Tcal}{\mathcal{T}}
\newcommand{\Ucal}{\mathcal{U}}

\newcommand{\bs}{\boldsymbol} 

\newcommand{\bfb}{{\bs b}}
 
\newcommand{\bfd}{{\bs d}}
\newcommand{\bfe}{{\bs e}}

\newcommand{\bfh}{{\bs h}}
\newcommand{\bfn}{{\bs n}}
\newcommand{\bfp}{{\bs p}}

\newcommand{\bft}{{\bs t}}

\newcommand{\bfx}{{\bs x}}
\newcommand{\bfy}{{\bs y}}

\newcommand{\bfA}{{\bs A}}
\newcommand{\bfB}{{\bs B}}

\newcommand{\bfE}{{\bs E}}

\newcommand{\bfT}{{\bs T}}
\newcommand{\bfU}{{\bs U}}
\newcommand{\bfV}{{\bs V}}

\newcommand{\bfX}{{\bs X}}

\newcommand{\bfeta}{{\bs\eta}}
\newcommand{\bfnu}{{\bs\nu}}

\newcommand{\bftheta}{{\bs\theta}}

\newcommand{\bfxi}{\bs\xi}

\newcommand{\loc}{{\mathrm{loc}}}

\newcommand{\ol}[1]{\overline{#1}}

\newcommand{\tm}{\subseteq} 
\newcommand{\di}{\partial}
\newcommand{\trans}{{\top}}

\newcommand{\ds}{\, \dif s} 
\newcommand{\dt}{\, \dif t}
 
\newcommand{\dx}{\, \dif \bfx}

\newcommand{\xhat}{\widehat{\bfx}}

\newcommand{\rmi}{\mathrm{i}} 

\newcommand{\rme}{\mathrm{e}}

\newcommand{\eps}{\varepsilon}

\newcommand{\bfs}{{\mathbf{s}}}

\newcommand{\Sd}{{S^2}}
\newcommand{\Sone}{{S^1}}
\newcommand{\Stwo}{{S^2}}
\newcommand{\Rd}{{\R^3}}
\newcommand{\Rtwo}{{\R^2}}
\newcommand{\Cd}{{\C^3}}
\newcommand{\Rdd}{{\R^{3\times3}}}
\newcommand{\Cdd}{{\C^{3\times3}}}

\newcommand{\Umn}{\bfU^m_n}

\newcommand{\Vmn}{\bfV^m_n}

\newcommand{\Amn}{\bfA^m_n}

\newcommand{\Bmn}{\bfB^m_n}

\DeclareMathOperator{\curl}{{\mathbf{curl}}}

\renewcommand{\div}{\mathrm{div}}

\newcommand{\BR}{{B_R(0)}}

\newcommand{\Ei}{\bfE^i}

\newcommand{\Esrho}{\bfE^s_\rho}

\newcommand{\Etrho}{\bfE_\rho}

\newcommand{\Einftyrho}{\bfE^\infty_\rho}

\newcommand{\Drho}{D_\rho}
\newcommand{\Drhoprime}{D_\rho'}

\newcommand{\FcalDrho}{\Fcal_{D_\rho}}
\newcommand{\TcalDrho}{\Tcal_{D_\rho}}

\newcommand{\LttSd}{{L^2_t(S^2,\Cd)}}

\newcommand{\HS}{\mathrm{HS}}

\newcommand{\XK}{\bfX_\bfp}

\newcommand{\tp}{\bft_\bfp}
\newcommand{\np}{\bfn_\bfp}
\newcommand{\bp}{\bfb_\bfp}

\newcommand{\npada}{\bfn_{\bfp,\theta}}
\newcommand{\bpada}{\bfb_{\bfp,\theta}}

\newcommand{\kappamax}{\kappa_{\max}}
\newcommand{\rGamma}{r}

\newcommand{\tri}{\bigtriangleup}
\newcommand{\Stri}{\Scal_\tri}

\newcommand{\xvec}{\vv{\bfx}}
\newcommand{\dvec}{\vv{\bfd}}
\newcommand{\svec}{\vv{\bfs}}
\newcommand{\yvec}{\vv{\bfy}}

\newcommand{\chitwo}{\chi_{2}}
\newcommand{\chiHS}{\chi_\HS}

\newcommand{\Rtheta}{R_\theta}

\newcommand{\Meps}{\M^\eps}

\newcommand{\meps}{\mathbbm{m}^\eps}

\newcommand{\epsr}{\eps_r}

\newcommand{\Vplus}{V^+}
\newcommand{\Vminus}{V^-}

\newcommand{\bfTrho}{\bfT_\rho}
\newcommand{\bfTrhoN}{\bfT_{\rho,N}}
\newcommand{\Uad}{\Ucal_{\mathrm{ad}}}

\newcommand{\Jtwo}{J_2}
\newcommand{\JHS}{J_\HS}

\newcommand{\ph}{\,\cdot\,}

\newcommand{\SOd}{\mathrm{SO}(3)}
\newcommand{\Vptheta}{V_{\bfp,\theta}}

\newcommand{\fres}{f_{\mathrm{res}}}

\newcommand{\fopt}{f_{\mathrm{opt}}}
\newcommand{\kopt}{k_{\mathrm{opt}}}
\newcommand{\lambdaopt}{\lambda_{\mathrm{opt}}}

\DeclareMathAlphabet{\mathbi}{\encodingdefault}{\rmdefault}{\bfdefault}{\itdefault}
\DeclareRobustCommand{\vec}[1]{\ifmmode\mathbi{#1}\else\textbf{\textit{#1}}\fi}

\DeclareMathOperator{\dif}{d\!}

\DeclareMathOperator{\real}{Re}
\DeclareMathOperator{\imag}{Im}

\DeclareMathOperator*{\argmin}{argmin}

\journal{SFB 1173}

\begin{document}

\title{Maximizing the electromagnetic chirality of thin metallic
  nanowires\\ at optical frequencies}

\author[add1]{Ivan Fernandez-Corbaton}
\ead{ivan.fernandez-corbaton@kit.edu}
\author[add2]{Roland Griesmaier}
\ead{roland.griesmaier@kit.edu}
\author[add2]{Marvin Kn\"oller}
\ead{marvin.knoeller@kit.edu}
\author[add1,add3]{Carsten Rockstuhl}
\ead{carsten.rockstuhl@kit.edu}

\address[add1]{Institute of Nanotechnology, Karlsruhe Institute of
  Technology, 76021 Karlsruhe, Germany} 
\address[add2]{Institute for Applied and Numerical Mathematics,
  Karlsruhe Institute of Technology, 76131 Karlsruhe, Germany} 
\address[add3]{Institute of Theoretical Solid
  State Physics, Karlsruhe Institute of Technology, 76131 Karlsruhe,
  Germany}  

\date{\today}

\begin{abstract}
  Electromagnetic waves impinging on three-dimensional helical
  metallic metamaterials have been shown to exhibit chiral effects of 
  large magnitude both theoretically and in experimental realizations.  
  Chirality here describes different responses of scatterers,
  materials, or metamaterials to left and right circularly polarized
  electromagnetic waves.
  These differences can be quantified in terms of electromagnetic
  chirality measures. 
  In this work we consider the optimal design of thin metallic
  free-form nanowires that possess measures of electromagnetic
  chirality as large as fundamentally possible.
  We focus on optical frequencies and use a gradient based optimization
  scheme to determine the optimal shape of highly chiral thin silver
  and gold nanowires. 
  The electromagnetic chirality measures of our optimized nanowires
  exceed that of traditional metallic helices. 
  Therefore, these should be well suited as building
  blocks of novel metamaterials with an increased chiral response.   
  We discuss a series of numerical examples, and we evaluate the
  performance of different optimized designs. 
\end{abstract}

\maketitle

{\small\noindent
  Mathematics subject classifications (MSC2010): 78M50, (49Q10, 78A45)
  \\\noindent
  Short title: Maximizing the em-chirality of metallic nanowires
}

\section{Introduction}
The concept of electromagnetic chirality (em-chirality) has recently
been introduced to quantify differential interactions of scattering
objects, materials, or metamaterials with electromagnetic waves of
positive and negative helicity. 
Broadly speaking, if all scattered fields that are caused by
illuminating an object with either left or right circularly polarized
electromagnetic waves can also be reproduced using circularly
polarized electromagnetic waves of the opposite helicity, then we call
the object em-achiral.  
If this is not the case, then the object is em-chiral.
This notion of em-chirality is consistent with the traditional
geometric concept of chirality, but in contrast to the geometric
chirality of an object, its em-chirality can be quantified directly in
terms of the object's interaction with electromagnetic waves using
em-chirality measures
\cite{AreHagHetKir18,FerFruRoc16,VavFer22}. 
These em-chirality measures are bounded from below by zero and from
above by the total interaction cross-section of the scattering object,
which makes them well-suited objective functionals for a shape
optimization. 
The lower bound zero is attained for em-achiral objects, and a
maximally em-chiral scattering object, material, or metamaterial would
scatter waves of one helicity while at the same time not scattering
waves of the opposite helicity at all.

Maximally em-chiral scattering objects made of isotropic materials are
being considered as possible building blocks of novel chiral
metamaterials that exhibit effective chiral material parameters many
orders of magnitude larger than what is found in natural substances
\cite{FerRocWeg19,HenSchDuaGie17,KadMilHecWeg19}. 
Metamaterials with large electromagnetic chirality have potential
applications in angle-insensitive circular polarizers, which are
materials that transmit one circular polarization of light and that
reflect and/or absorb the opposite handedness nearly
completely~\cite{GanEtAl12,GanEtAl09}.
The goal of this work is to design such building blocks for
novel metamaterials, whose chiral response at optical
frequencies is as large as fundamentally possible.
Previous numerical studies in
\cite{FerFruRoc16,GanWegBurLin10,GarEtAl21} have concentrated on the
optimal design of silver helices with fixed circular cross-sections at
frequencies ranging from the far-infrared to the optical band. 
This amounts to optimizing four parameters describing the geometry of
the helix: the radius of the helix spine, the thickness of the helix
wire, the pitch of the helix, and the number of turns. 
While the obtained optimized silver helices achieve high chirality
measures for wavelengths of $3~\mu$m or more, their performance
decreases significantly towards the optical frequency band
\cite{GarEtAl21}. 
Three-dimensional helical metallic metamaterials have also been
studied experimentally \cite{GanEtAl12,GanEtAl09,GarEtAl21,KasWeg15b}. 
In \cite{AmmHamNed99} the effective Drude-Born-Fedorov constitutive
relations governing the propagation of electromagnetic waves in a
chiral metamaterial, which is obtained by embedding a large number of
regularly spaced, randomly oriented metallic helices in a homogeneous
medium, have been derived from the linear constitutive relations for
homogeneous isotropic media. 

In this work we go beyond helical shapes for the individual
scatterers.  
Instead of optimizing the few shape parameters of a helix with
circular cross-section, we consider a free-form shape optimization for
thin metallic nanowires with fixed but arbitrary cross-sections. 
In addition to the shape of the spine curve of the nanowire, we also
optimize a possible twisting of its cross-section along the spine
curve. 
This extends an earlier study in \cite{AreGriKno21}, where a free-form
shape optimization for thin dielectric nanowires with circular
cross-sections has been discussed. 
It has been observed in \cite{AreGriKno21} that the optimized thin
dielectric nanowires do not possess very high values of em-chirality
at optical frequencies. 
The distinguishing feature of the noble metals considered for
the nanowires in this work, in particular of silver and gold, is the 
negative real part of their electric permittivity at optical
frequencies combined with a relatively small positive imaginary part.
This permits the excitation of plasmonic resonances, which are not
observed for dielectric nanowires, and that turn out to be relevant
in the design of highly chiral nanowires (see also
\cite{HenSchDuaGie17,HofEtAl19,ValEtAl13}). 
In particular silver has a higher plasma frequency and lower damping 
compared to other noble metals across the optical band, and it is thus
well suited for our purpose. 

The chirality measures from \cite{AreHagHetKir18,FerFruRoc16}, which
have to be maximized in the shape optimization, are defined in terms
of the singular values of the electromagnetic far field operator
associated to the thin nanowire. 
This operator maps superpositions of plane wave incident fields to the
far field patterns of the electromagnetic waves that are scattered at
the nanowire. 
Accordingly, each evaluation of the objective functional and of its
shape derivative in the shape optimization requires the
evaluation of the far field operator and of its shape derivative
corresponding to the current iterate in the shape optimization. 
Applying a traditional shape optimization scheme (see,
e.g.,~\cite{EppHar05,HadJiaRia17,HagAreBetHet19,HagHet20,HinLauYou15,LebEtAl19,SemEtAl15})
would require solving a large number of scattering problems using
either integral equations or finite elements in each iteration step of 
the algorithm. 
This would be computationally intensive, and hardly feasible for the
problem under consideration. 
Instead, we utilize our assumption that the thickness of the nanowire
is small relative to the wave length of the electromagnetic field, 
and we apply an asymptotic representation formula to approximate the  
far field operators associated to thin metallic nanowires
without solving any differential equation. 
This asymptotic formula has been developed for scattered fields due to
thin dielectric nanowires in 
\cite{AlbCap18,BerCapdeGFra09,CapGriKno21,Gri11}.
In previous studies similar asymptotic representation formulas have
been successfully applied in algorithms for shape reconstruction in
electrical impedance tomography and inverse
scattering~\cite{BerCapdeGFra09,CapGriKno21,Gri10,GriHyv11}. 
We provide a rigorous theoretical justification of its
generalization to thin metallic nanowires that are characterized by
complex-valued electric permittivities with negative real and
positive imaginary parts as considered in this work. 
Therewith, we  develop a quasi-Newton algorithm that does not require
to solve a single Maxwell system during the optimization procedure. 
This shape optimization scheme is an extension of an algorithm
from~\cite{AreGriKno21}.  
The novel features are that we consider metallic nanowires, that we
allow for arbitrary cross-sections, and that we optimize not only the
shape of the spine curve of the nanowire but also the twisting of the
cross-section of the nanowire along the spine curve.

In our numerical examples we focus on silver and gold nanowires with 
elliptical cross sections, and we work at discrete frequencies in the
optical band. 
We obtain the largest chirality measures when both, the shape of the
spine curve and the twist rate of the cross-section of the nanowire,
are suitably optimized simultaneously, and when the frequency, where
the optimization is being carried out, is chosen to be somewhat below
the plasmonic resonance frequency of the nanowire. 
We discuss several numerical examples and evaluate the performance
of different optimized designs.
For practical applications of these results it is important to 
note that the asymptotic representation formula that we use in the
shape optimization yields an accurate approximation, when the
thickness of the nanowire is in the range of a few percent of the
wave length of the electromagnetic field.  At optical frequencies
the fabrication of the corresponding nanowires might be challenging.

The article is organized as follows. 
In the next section we introduce our mathematical setting and the
geometric description of thin metallic nanowires that we use
throughout this work. 
We also establish the asymptotic representation formula
for far field patterns of scattered electromagnetic waves due to
thin metallic nanowires.  
In Section~\ref{sec:Chirality} we discuss electromagnetic chirality
and give a short synopsis of the concepts and results that have
recently been developed in \cite{AreHagHetKir18,FerFruRoc16}. 
In Section~\ref{sec:ShapeOptimization} we develop the shape
optimization scheme, and we provide numerical results in
Section~\ref{sec:NumericalResults}.  
In the appendix we establish some estimates that are required in the
explicit characterization of the electric polarization tensor
associated to thin metallic nanowires in
Section~\ref{sec:ScatteringProblem}.

\section{Scattering at thin metallic nanowires}
\label{sec:ScatteringProblem}
We consider the scattering of electromagnetic waves at thin
metallic nanowires in three-dimensional free space.
To begin with, we introduce the geometrical description for these
nanowires. 
Let $\Gamma\tm\Rd$ be a simple (i.e., not self-intersecting) curve
with $C^3$-parametrization~$\bfp:[0,L] \to \Rd$ such that~$\bfp'\neq 0$
on $[0,L]$ for some~$L>0$.
Throughout, we assume for simplicity that $L$ coincides with the
length of $\Gamma$, but the parametrization $\bfp$ does not
necessarily have to be by arc-length. 
We use~$\tp := \bfp'/|\bfp'|$ to denote the
\emph{unit tangent vector} along $\Gamma$, and we consider an
associated moving orthogonal frame $(\tp,\np,\bp)$ with
$\tp\times\np = \bp$ that is \emph{rotation minimizing} (or twist
free) in the sense that 
\begin{equation*}
  \np'(s) - \phi(s)\, \tp(s) \,=\, 0
  \quad\text{and}\quad
  \np(s)\cdot\tp(s) \,=\, 0 \,, \qquad
  s\in[0,L] \,,
\end{equation*}
for some continuous function $\phi:[0,L]\to\Rd$ with a
\emph{reference vector}  $\np(0)\in\Sd$ satisfying $\np(0)\perp\tp(0)$
(see, e.g.,~\cite{Bis75,WanJutZheLiu08}).
We denote by
\begin{equation}
  \label{eq:DefCurvature}
  \kappa(s)
  \,:=\, \frac{|\bfp'(s)\times \bfp''(s)|}{|\bfp'(s)|^3}
  \,=\, \frac{1}{|\bfp'(s)|} \biggl| \frac{\bfp''(s)}{|\bfp'(s)|}
  - \frac{\bfp'(s)\cdot \bfp''(s)}{|\bfp'(s)|^3} \bfp'(s) \biggr| \,,
  \qquad s\in [0,L] \,,
\end{equation}
the \emph{curvature} of $\Gamma$, and we write
$\kappamax := \|\kappa\|_{C([0,L],\R)}$ for the maximal curvature of
$\Gamma$.
Accordingly, let~$B_{\rGamma}'(0) \subset \R^2$ be a two-dimensional
disk with radius $0<\rGamma<1/\kappamax$ centered at the origin. 
Then, 
\begin{align}
  \label{eq:Xrho}
  \XK: [0,L] \times B_\rGamma'(0) \to \R^3 \,, \qquad
  \XK(s,\eta,\xi) \,:=\, \bfp(s) + \eta\,\np(s) + \xi\,\bp(s)
\end{align}
is a local coordinate system around $\Gamma$.

\begin{figure}[thp]
  \centering
  \includegraphics[height = 5.5cm]{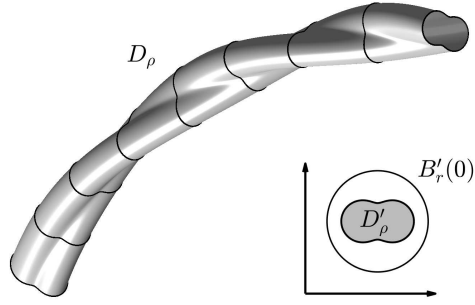}
  \caption{Sketch of the geometry of a thin nanowire $\Drho$.}
  \label{fig:Sketch}
\end{figure}

In the following, $\Gamma$ will be the \emph{spine curve} of the 
nanowire. 
We restrict the discussion to thin nanowires with fixed but possibly
twisting cross-section $\Drhoprime := \rho B'$, where the
\emph{rescaled cross-section} $B'\tm B_1'(0)\tm \R^2$ is a bounded
Lipschitz domain. 
Here, $0<\rho<\rGamma$ is a scaling parameter that will determine the
thickness of the nanowire. 
To describe the twisting of the cross-section of the nanowire along
the spine curve, we let $\theta \in C^2([0,L],\R)$ be a
\emph{twisting function}, and we define a two-dimensional parameter
dependent rotation matrix 
\begin{equation}
  \label{eq:DefRtheta}
  \Rtheta(s)
  \,:=\,
  \begin{bmatrix}
    \cos\theta(s) & -\sin\theta(s) \\
    \sin\theta(s) & \cos\theta(s)
  \end{bmatrix} \,, \qquad s \in [0,L] \,.
\end{equation}
Accordingly, the support of the nanowire shall be given by
\begin{equation}
  \label{eq:DefDrho}
  \Drho
  \,:=\, \Bigl\{\XK(s,\eta,\xi) \;\Big| \;
  s \in [0,L] \,,\;
  \Rtheta^{-1}(s)
  \begin{bmatrix}
    \eta\\\xi
  \end{bmatrix}
  \in \Drhoprime \Bigr\} 
\end{equation}
(see Figure~\ref{fig:Sketch} for a sketch).

\begin{remark}
  In \cite{CapGriKno21} we used the Frenet-Serret frame instead
  of a rotation minimizing frame to describe the geometry of thin
  tubular scattering objects as in \eqref{eq:DefDrho}.
  As a consequence, we had to assume that the spine curve $\Gamma$ 
  has a non-vanishing curvature. 
  Using a rotation minimizing frame, the analysis in~\cite{CapGriKno21}
  extends without further changes to spine curves, which may have
  vanishing curvature at some points. 
  The special case, when $\Gamma$ is a straight line segment, has
  been considered in \cite{BerCapdeGFra09}.~\hfill$\lozenge$
\end{remark}
  
  In contrast to the Frenet-Serret frame, the rotation minimizing frame
  $(\tp,\np,\bp)$ in \eqref{eq:Xrho} is not uniquely determined. 
  It depends on the particular choice of the reference vector
  $\np(0)\perp\tp(0)$.  
  A different reference vector requires a suitably modified twist
  function $\theta$ in \eqref{eq:Xrho} to obtain the same
  support~$\Drho$ for the nanowire in \eqref{eq:DefDrho}. 
  To avoid this ambiguity, we use in the following a
  \emph{geometry adapted frame} $(\tp,\npada,\bpada)$, which is
  defined by 
  \begin{equation}
    \label{eq:AdaptedFrame}
    \bigl[\npada(s) \big| \bpada(s)\bigr]
    \,:=\, \bigl[\np(s) \big| \bp(s)\bigr] \Rtheta(s) \,,
    \qquad s\in [0,L] \,.
  \end{equation}
  Therewith,
  \begin{equation}
    \label{eq:DrhoAdapted}
    \Drho
    \,=\, \Bigl\{ \bfp(s) + \eta\,\npada(s) + \xi\,\bpada(s)
    \;\Big| \; s \in [0,L] \,,\;
    (\eta, \xi) \in \Drhoprime \Bigr\} .
  \end{equation}
  This means that the twisting of the cross-section along the spine
  curve in the description of the support~$\Drho$ of the nanowire is
  now included into the frame.
  The \emph{twist rate} $\beta := \theta'\in C([0,L],\R)$
  of the geometry adapted frame is defined by
  \begin{equation}
    \label{eq:TwistRate}
    \beta(s)
    \,=\, \npada'(s) \cdot \bpada(s)
    \,, \qquad s\in[0,L] \,.
  \end{equation}

We work with electric fields only, but the corresponding magnetic
fields can immediately be obtained from the associated first order
Maxwell systems. 
Let $\eps_0 := 8.85 \times 10^{-12}$~Fm$^{-1}$
and~${\mu_0 := 1.25 \times 10^{-6}}$~Hm$^{-1}$ be the
\emph{electric permittivity} and the \emph{magnetic permeability} of
free space, respectively. 
Denoting by $f>0$ the \emph{frequency}, the angular frequency is
$\omega := 2\pi f$, and the associated \emph{wave number} is given by
$k := \omega\sqrt{\eps_0\mu_0}$. 
An \emph{electric incident field}~$\Ei$ is a solution to Maxwell's
equations 
\begin{equation}
  \label{eq:Ei}
  \curl\curl\Ei - k^2 \Ei
  \,=\, 0 \qquad \text{in } \Rd \,.
\end{equation}

\begin{table}[thp]
  \centering
  \begin{tabular}{c c r r c r r }
    \hline
    & & \multicolumn{2}{c}{Silver} & & \multicolumn{2}{c}{Gold}\\ \hline
    $f$ & & $\real(\epsr)$ & $\imag(\epsr)$ & & $\real(\epsr)$ & $\imag(\epsr)$ \\ \hline
    300 THz & & -50.55 & 0.57 & & -41.78 & 2.94 \\ 
    350 THz & & -36.23 & 0.48 & & -28.84 & 1.77 \\ 
    400 THz & & -26.94 & 0.32 & & -20.11 & 1.24 \\ 
    450 THz & & -20.57 & 0.44 & & -14.10 & 1.04 \\ 
    500 THz & & -16.05 & 0.44 & & -9.36 & 1.53 \\ 
    550 THz & & -12.62 & 0.42 & & -5.59 & 2.19 \\ 
    600 THz & & -9.78 & 0.31 & & -2.54 & 3.65 \\
    650 THz & & -7.64 & 0.25 & &  -1.73 & 5.06 \\ 
    700 THz & & -5.94 & 0.20 & &  -1.69 & 5.66 \\ 
    750 THz & & -4.41 & 0.21 & &  -1.66 & 5.74 \\ 
    800 THz & & -3.10 & 0.21 & &  -1.50 & 5.63 \\ \hline 
  \end{tabular}
  \caption{Relative electric permittivities $\epsr$ of silver and gold
    at optical frequencies (from~\cite{JoChri72}).}
  \label{tab:MaterialParameters}
\end{table}

We suppose that the thin nanowire is made of a noble metal like
silver or gold. 
At optical frequencies these materials are characterized by a
frequency-dependent complex-valued
\emph{relative electric permittivity}~$\epsr\in\C$ with negative real
part $\real(\epsr)<0$ and positive imaginary part $\imag(\epsr)>0$
(i.e., their electric permittivity is $\eps_m:=\eps_0\epsr$), while the
magnetic permeability is the same as for the surrounding free space. 
The real and imaginary parts of the relative electric
permittivities~$\epsr$ of silver and gold at optical frequencies as
considered in this work are retrieved from the data base
\cite{JoChri72}. 
They are shown in Table~\ref{tab:MaterialParameters}.
We prefer these experimental data to the classical Drude model,
because the Drude model is known to be rather inaccurate across the
optical band (see e.g.,~\cite[pp.~117--118]{Greffet12}).

We denote by
\begin{equation}
  \label{eq:epsrho}
  \eps_\rho
  \,:=\, \eps_0 + (\eps_m-\eps_0) \chi_{\Drho}
\end{equation}
the piecewise constant electric permittivity distribution on~$\Rd$ in
the presence of the nanowire.
Here, $\chi_{\Drho}$ is the indicator function for $\Drho$ from
\eqref{eq:DefDrho}.  
The scattering problem consists in finding the \emph{total electric
  field}~$\Etrho$ satisfying
\begin{subequations}
  \label{eq:ScatteringProblem}
  \begin{equation}
    \label{eq:Et}
    \curl \curl \Etrho - \omega^2 \mu_0 \eps_\rho \Etrho
    \,=\, 0 \qquad \text{in } \Rd \,,
  \end{equation}
  such that the \emph{scattered electric field} 
  \begin{equation}
    \label{eq:Es}
    \Esrho
    \,:=\, \Etrho - \Ei
  \end{equation}
  fulfills the Silver-M\"uller radiation condition
  \begin{equation}
    \label{eq:SilverMueller}
    \lim_{|\bfx|\to\infty} \left( \curl\Esrho(\bfx)\times\bfx
      - \rmi k|\bfx| \Esrho(\bfx) \right)
    \,=\, 0 
  \end{equation}
\end{subequations}
uniformly with respect to all directions 
$\xhat := \bfx/|\bfx| \in\Stwo$. 

\begin{lemma}
  \label{lmm:Wellposedness}
  Let $\eps_r\in\C$ with $\real(\eps_r)<0$ and $\imag(\eps_r)>0$, and
  let $\Drho\tm\Rd$ and $\eps_\rho$ be as in \eqref{eq:DefDrho} and
  \eqref{eq:epsrho} for some $0<\rho<\rGamma$, respectively.  
  Suppose that $\Ei \in H_{\loc}(\curl;\R^3)$
  satisfies \eqref{eq:Ei}.
  Then, the scattering problem~\eqref{eq:ScatteringProblem} has a
  unique solution $\Etrho \in H_{\loc}(\curl;\Rd)$. 
  The corresponding scattered electric field $\Esrho$ has the
  asymptotic behavior 
  \begin{equation*}
    \Esrho(\bfx)
    \,=\, \frac{\rme^{\rmi k |\bfx|}}{4\pi |\bfx|}
    \left( \Einftyrho(\xhat) + \mathcal{O}(|\bfx|^{-1}) \right)
    \qquad \text{as }|\bfx|\to\infty
  \end{equation*}
  uniformly in $\xhat\in\Sd$.
  The vector function $\Einftyrho \in \LttSd$ is called the
  \emph{electric far field pattern}.\footnote{As usual $\LttSd$
    denotes the vector space of square integrable tangential vector
    fields on the unit sphere.} 
\end{lemma}

\begin{proof}
  The uniqueness of solutions follows exactly as in the proof of
  \cite[Thm.~10.1]{Mon03}, where complex electric permittivities with
  positive real and imaginary parts have been considered. 
  The existence of solutions can be obtained applying Riesz-Fredholm
  theory, using the same arguments as in \cite[Sect.~10.3]{Mon03}.
  The far field expansion has been shown in \cite[Cor.~9.5]{Mon03}. 
\end{proof}

We focus on thin metallic nanowires, i.e., the scaling parameter
$0<\delta<\rGamma$ in \eqref{eq:DefDrho} is supposed to be small
relative to the \emph{wave length}~$\lambda = 2\pi/k$ and relative to
the total length $L$ of the nanowire. 
In this case, solutions to the scattering problem
\eqref{eq:ScatteringProblem} can be accurately described by
the following asymptotic representation formula.

\begin{theorem}
  \label{thm:GenAsy}
  Let $\eps_r\in\C$ with $\real(\eps_r)<0$ and $\imag(\eps_r)>0$, and
  for any $0<\rho<\rGamma$ let $\Drho\tm\Rd$ and $\eps_\rho$ be as in
  \eqref{eq:DefDrho} and~\eqref{eq:epsrho}, respectively. 
  Denoting by $\Ei$ an electric incident field, the electric far
  field pattern corresponding to the solution of the scattering
  problem \eqref{eq:ScatteringProblem} satisfies, for each
  $\xhat\in\Sd$, 
  \begin{equation}
    \label{eq:GenAsyEinfty}
    \begin{split}
      \Einftyrho(\xhat) 
      &\,=\, |B'| (k\rho)^2 \int_\Gamma
      (\epsr-1) \rme^{-\rmi k \xhat\cdot \bfy}
      \bigl((\xhat\times\I_3)\times\xhat\bigr)
      \Meps(\bfy) \Ei(\bfy) \ds(\bfy)
      + o\bigl((k\rho)^2\bigr)
    \end{split}
  \end{equation}
  as $\rho\to 0$.
  Here, $|B'|$ denotes the area of the rescaled cross-section
  $B'=\Drho'/\rho$, $\I_3\in\Rdd$ is the identity matrix, and the
  matrix valued function 
  $\Meps \in L^2(\Gamma,\Cdd)$ is the \emph{electric polarization tensor}.
  The term $o((k\rho)^2)$ in \eqref{eq:GenAsyEinfty} is such
  that~$\|o((k\rho)^2)\|_{L^\infty(\Sd)}/(k\rho)^2$ converges to zero
  for any fixed $\Ei$ satisfying $\|\Ei\|_{H(\curl;\BR)} \leq C$ for
  some fixed $C>0$. 
\end{theorem}

Theorem~\ref{thm:GenAsy} is a special case of a more general result
that has been established for scatterers of low volume with complex
electric permittivities with positive real and nonnegative imaginary part in
\cite{AlbCap18,CapGriKno21,Gri11} (see also~\cite{CapVog03} for an
earlier version of this asymptotic perturbation formula in
electrostatics). 
Using the well-posedness from Lemma~\ref{lmm:Wellposedness}, this
result including its proof carries over to the case of metallic
nanowires with complex electric permittivities with negative real
and positive imaginary part as considered in this work. 

The electric polarization tensor $\Meps \in L^2(\Gamma,\Cdd)$ in
\eqref{eq:GenAsyEinfty} is defined as follows (see
\cite{CapGriKno21,Gri11}). 
Let~$R>0$ be sufficiently large such that ${\Drho \subset\subset \BR}$
for all $0<\rho<\rGamma$. 
For any $\bfxi\in\Sd$, we denote by ${W_\rho^{(\bfxi)}\in H^1(\BR)}$
the \emph{corrector potentials} satisfying
\begin{equation}
  \label{eq:CorrectorPotential}
  \div\bigl(\eps_{\rho} \nabla W_\rho^{(\bfxi)}\bigr) 
  \,=\, -\div\bigl((\eps_{\rho}-\eps_0) \bfxi\bigr) \quad 
  \text{in $\BR$} \,, \qquad 
  W_\rho^{(\bfxi)} \,=\, 0 \quad \text{on $\di\BR$} \,.
\end{equation}
Then, the electric polarization tensor is uniquely determined by 
\begin{equation}
  \label{eq:DefPolTen}
  \frac{1}{|\Gamma|} \int_\Gamma \bfxi\cdot  \Meps \bfxi \psi \ds
  \,=\, \frac{1}{|\Drho|} \int_{\Drho} |\bfxi|^2 \psi \dx
  + \frac{1}{|\Drho|} \int_{\Drho} 
  \bigl(\bfxi\cdot\nabla W_\rho^{(\bfxi)} \bigr)\psi \dx
  + o(1) \qquad \text{as } \rho \to 0 
\end{equation}
for all $\psi\in C(\ol{\BR})$ and any $\bfxi\in\Stwo$.
In contrast to the very general situation considered in
\cite[Thm.~1]{CapGriKno21}, no extraction of a subsequence is
required in \eqref{eq:DefPolTen} under our assumptions on the geometry
of the cross-sections (see \cite[Rem.~4.5]{CapGriKno21}). 

\begin{remark}[Dielectric nanowires]
  For dielectric nanowires, i.e., when $\epsr>0$ is positive, the
  electric polarization tensor $\Meps \in L^2(\Gamma,\Rdd)$ is a
  real-valued, symmetric, and positive definite matrix pointwise a.e.\
  on $\Gamma$ (see \cite[Sec.~4]{CapVog03}).
  In particular, it can be diagonalized. 
  A pointwise characterization of the eigenvalues and eigenvectors of
  $\Meps$ has been established in~\cite[Thm.~4.1]{CapGriKno21} (see
  \cite{BerCapdeGFra09} for the special case when~$\Gamma$ is a straight
  line segment). 
  It has been shown that, for a.e.\ $s\in[0,L]$, the tangent vector
  $\tp(s)$ is an eigenvector of the electric polarization tensor
  $\Meps(\bfp(s))$ corresponding to the eigenvalue~$1$. 
  Furthermore, using the vectors~$(\npada(s),\bpada(s))$ from
  \eqref{eq:AdaptedFrame} as a basis of the plane orthogonal
  to~$\tp(s)$, the electric polarization tensor reduces to the
  corresponding two-dimensional electric polarization tensor $\meps$
  associated to the rescaled cross-sections~$B' = \Drhoprime/\rho$ in
  this plane. 
  Here, $\meps = (\meps_{ij}) \in \C^{2\times 2}$ is given by
  \begin{equation}
    \label{eq:mepsij}
    \meps_{ij} 
    \,=\, \delta_{ij} + \frac1{|B'|} \int_{B'} 
    \frac{\di w_j}{\di x'_i}(\bfx') \dx' \,,
    \qquad 1\leq i,j\leq 2 \,,
  \end{equation}
  where $\delta_{ij}$ denotes the Kronecker delta, and
  $w_j\in H^1_\loc(\R^2)$ is the unique solution to the transmission
  problem 
  \begin{subequations}
    \label{eq:wj}
    \begin{align}
      \Delta w_j
      &\,=\, 0 \qquad \quad\text{in } \R^2\setminus\di B' \,,\\
      w_j\big|_{\di B'}^+ - w_j\big|_{\di B'}^- 
      &\,=\, 0 \,,\\
      \frac{\di w_j}{\di\bfnu}\Big|_{\di B'}^+ 
      - \eps_r \frac{\di w_j}{\di\bfnu}\Big|_{\di B'}^-
      &\,=\, (\eps_r-1) \nu_j \,,\\
      w_j(\bfx') &\to 0 \qquad \quad \text{as $|\bfx'|\to\infty$} \,.
    \end{align}
  \end{subequations}
  (see \cite[Rem~4.5]{CapGriKno21}).
  Given any specific geometry for $B'$, the functions $w_j$,
  $j=1,2$, in \eqref{eq:mepsij} can be approximated by solving the
  two-dimensional transmission problem \eqref{eq:wj} numerically.
  Then, the two-dimensional electric polarization tensor $\meps$ can be
  evaluated by applying a quadrature rule to the integral
  in~\eqref{eq:mepsij}. 
  In the special case when $B'$ is an ellipse with semi axes of
  lengths $a$ and $b$ for some $0< a \leq b< 1$ that are aligned with
  the coordinate axes in $\Rtwo$, the transmission problem
  \eqref{eq:wj} can be solved by separation of variables.
  Then, the integral in \eqref{eq:mepsij} can be evaluated explicitly
  to obtain that 
  \begin{equation}
    \label{eq:meps}
    \meps
    \,=\, \begin{bmatrix}
      \frac{a+b}{a+\epsr b}& 0 \\
      0 & \frac{a+b}{b+\epsr a}
    \end{bmatrix}
  \end{equation}
  (see, e.g., \cite{AmmKan07,BruHanVog03}).
  Here, the semi axis of length $a$ is aligned with the $x_1$-axis and
  the semi axis of length $b$ is aligned with the
  $x_2$-axis.~\hfill$\lozenge$ 
\end{remark}

When $\epsr$ is complex-valued with negative real and positive
imaginary part, the electric polarization tensor
$\Meps\in L^2(\Gamma,\Cdd)$ is still symmetric pointwise a.e.\ on
$\Gamma$ (see Lemma~\ref{lmm:PolTenBounds} in the appendix), but this
no longer implies diagonalizability.  
However, the following statements of \cite[Thm.~4.1]{CapGriKno21}
including their proofs remain valid for such complex-valued relative
electric permittivities.

\begin{proposition}
  \label{pro:CharacterizationPolTen}
  Suppose $\eps_r\in\C$ with $\real(\eps_r)<0$ and $\imag(\eps_r)>0$,
  and for any $0<\rho<\rGamma$ let $\Drho\tm\Rd$ and $\eps_\rho$ be as
  in \eqref{eq:DefDrho} and \eqref{eq:epsrho}, respectively. 
  Let $\Meps\in L^2(\Gamma,\Cdd)$ be the electric polarization tensor
  defined in~\eqref{eq:CorrectorPotential}--\eqref{eq:DefPolTen} and
  denote by $\meps\in\C^{2\times 2}$ the polarization tensor 
  corresponding to the rescaled cross-section~$B'=\Drhoprime/\rho$ defined
  in~\eqref{eq:mepsij}--\eqref{eq:wj}. 
  Then, the following pointwise characterization of $\Meps$ holds.
  \begin{enumerate}[(a)]
  \item Let $\tp(s)$ be the unit tangent vector at $\bfp(s)\in\Gamma$,
    then 
    \begin{equation}
      \label{eq:CharacterizationPolTen(a)}
      \tp(s) \cdot \Meps(\bfp(s)) \tp(s) \,=\, 1 \qquad
      \text{for a.e.\ $s\in [0,L]$}\,.
    \end{equation}
  \item Let $(\npada, \bpada)$ be the normal components of the
    geometry adapted frame $(\tp, \npada, \bpada)$ from
    \eqref{eq:AdaptedFrame}, let~$\bfxi'\in \Sone$, and let
    $\bfxi\in C^1([0,L],\Stwo)$ be given by
    $\bfxi(s) := \bigl[ \npada(s) , \bpada(s) \bigr] \bfxi' \in \Stwo$
    for all~${s\in[0,L]}$.  
    Then,
    \begin{equation*}
      \bfxi(s) \cdot \Meps(\bfp(s)) \bfxi(s) 
      \,=\, \bfxi' \cdot \meps \bfxi' \qquad 
      \text{for a.e.\ $s\in [0,L]$}\,.
    \end{equation*}
  \end{enumerate}
\end{proposition}

As in the dielectric case, we use pointwise polarization tensor
bounds to show that the unit tangent vector~$\tp(s)$ is in fact an
eigenvector of $\Meps(\bfp(s))$ corresponding to the eigenvalue $1$ for
a.e.~$s\in [0,L]$.
These bounds, which are different from the corresponding polarization
tensor bounds in the dielectric case from~\cite{CapVog03}, are shown
in Lemma~\ref{lmm:PolTenBounds} in the appendix. 

\begin{lemma}
  \label{lmm:TangentEigenvector}
  Suppose $\eps_r\in\C$ with $\real(\eps_r)<0$ and $\imag(\eps_r)>0$,
  and for any $0<\rho<\rGamma$ let $\Drho\tm\Rd$ and~$\eps_\rho$ be as
  in \eqref{eq:DefDrho} and \eqref{eq:epsrho}, respectively. 
  Let $\Meps\in L^2(\Gamma,\Cdd)$ be the electric polarization tensor
  defined in~\eqref{eq:CorrectorPotential}--\eqref{eq:DefPolTen}. 
  Then, the unit tangent vector $\tp(s)$ is an eigenvector of the
  matrix $\Meps(\bfp(s))$ corresponding to the eigenvalue $1$ for a.e.\
  $s\in[0,L]$. 
\end{lemma}

\begin{proof}
  From Lemma~\ref{lmm:PolTenBounds} in the appendix we obtain that
  there exists $\gamma \in (0, \pi/2)$ such that
  $\real(\rme^{\rmi \gamma}\ol{\eps_0}) >0$,
  $\real(\rme^{\rmi \gamma}\ol{\eps_m}) >0 $, and
  $\real(\rme^{\rmi \gamma}(\ol{\eps_m-\eps_0}))<0$.   
  Observing that, for any $\bfxi\in\Sd$,
  \begin{subequations}
    \label{eq:IdentityReIm}
    \begin{align}
      \bfxi \cdot  \real  \left( \rme^{\rmi \gamma}(\ol{\eps_m-\eps_0})
      \Meps \right) \bfxi
      &\,=\, \bfxi \cdot
        \left( \real \left( \rme^{\rmi \gamma}(\ol{\eps_m-\eps_0}) \right)
        \real(\Meps)
        -  \imag \left( \rme^{\rmi \gamma}(\ol{\eps_m-\eps_0}) \right)
        \imag(\Meps) \right) \bfxi \,, \\
      \bfxi \cdot  \imag  \left((\ol{\eps_m-\eps_0}) \Meps \right) \bfxi
      &\,=\, \bfxi \cdot
        \left( \real \left( (\ol{\eps_m-\eps_0}) \right) \imag(\Meps)
        +  \imag \left( (\ol{\eps_m-\eps_0}) \right)
        \real(\Meps) \right) \bfxi \,,
    \end{align}
  \end{subequations}
  the inequality \eqref{eq:InequReal} can be rewritten as
  \begin{equation}
    \label{eq:InequRealRewritten}
    |\bfxi|^2 
    \,\leq\, \bfxi \cdot \left( \real(\Meps(\bfx))
      - \frac{\imag \left( \rme^{\rmi \gamma}(\ol{\eps_m-\eps_0}) \right)}
      {\real \left( \rme^{\rmi \gamma}(\ol{\eps_m-\eps_0}) \right)}
      \imag(\Meps(\bfx)) \right) \bfxi
    \qquad \text{for every } \bfxi\in\Sd
    \text{ and a.e.\ } \bfx\in\Gamma \,.
  \end{equation}
  Moreover, subtracting \eqref{eq:InequReal} from \eqref{eq:InequImag}
  and using \eqref{eq:IdentityReIm} gives that
  \begin{equation}
    \label{eq:ProofComplexEigenvector1}
    \bfxi \cdot
    \left( \frac{\real(\ol{\eps_m-\eps_0})}{\imag(\ol{\eps_m-\eps_0})}
      + \frac{\imag(\rme^{\rmi \gamma}(\ol{\eps_m-\eps_0}))}
      {\real(\rme^{\rmi \gamma}(\ol{\eps_m-\eps_0}))} \right)
    \imag(\Meps(\bfx)) \bfxi
    \,\leq\, 0
    \qquad \text{for every } \bfxi\in\Sd \text{ and a.e.\ }
    \bfx\in\Gamma \,.
  \end{equation}
  Since the real factor on the left hand side of
  \eqref{eq:ProofComplexEigenvector1} is strictly positive by our
  assumptions on $\epsr$ and $\gamma$, 
  this implies that 
  \begin{equation}
    \label{eq:ProofComplexEigenvector2}
    \bfxi \cdot \imag(\Meps(\bfx)) \bfxi
    \,\leq\, 0 \qquad \text{for every } \bfxi\in\Sd \text{ and a.e.\ }
    \bfx\in\Gamma \,.
  \end{equation}
  Now, considering the imaginary part of
  \eqref{eq:CharacterizationPolTen(a)}, applying the min-max
  principle and using \eqref{eq:ProofComplexEigenvector2}, we find
  that the tangent vector $\tp(s)$ is an eigenvector of the
  self-adjoint matrix $\imag(\Meps(\bfp(s)))$ corresponding to the
  eigenvalue~$0$ for a.e.\ $s\in[0,L]$. 
  Similarly, considering the real part of 
  \eqref{eq:CharacterizationPolTen(a)}, applying the min-max
  principle and using \eqref{eq:InequRealRewritten}, shows that
  $\tp(s)$ is an eigenvector of the self-adjoint matrix
  $\real(\Meps(\bfp(s)))$ corresponding to the eigenvalue~$1$ for
  a.e.~$s\in[0,L]$.
  Since $\Meps = \real(\Meps) + \rmi \imag(\Meps)$, this ends the
  proof. 
\end{proof}

Combining Proposition~\ref{pro:CharacterizationPolTen} and
Lemma~\ref{lmm:TangentEigenvector} shows that the electric
polarization tensor $\Meps\in L^2(\Gamma,\Cdd)$ in the asymptotic
representation formula \eqref{eq:GenAsyEinfty} satisfies
\begin{subequations}
  \label{eq:PolTenExplicit}
  \begin{equation}
    \label{eq:PolTenExplicit1}
    \Meps(\bfp(s))
    \,=\, \Vptheta(s) M^\eps \Vptheta^\trans(s)
    \quad \text{for a.e.\ } s\in[0,L] \,,  
  \end{equation}
  where the matrix valued function $\Vptheta \in C^1([0,L],\SOd)$ is
  given by $\Vptheta := [ \tp | \npada | \bpada ]$,
  and
  \begin{equation}
    \label{eq:PolTenExplicit2}
    M^\eps
    \,:=\, \begin{bmatrix}
      1 & \hspace*{-\arraycolsep}\vline\hspace*{-\arraycolsep} &
      \begin{matrix}
        0 & 0
      \end{matrix} \\
      \hline
      \begin{matrix}
        0 \\ 0
      \end{matrix}
      & \hspace*{-\arraycolsep}\vline\hspace*{-\arraycolsep}
      & \meps 
    \end{bmatrix} \in \Cdd \,.
  \end{equation}
\end{subequations}
Therewith, the asymptotic representation formula
\eqref{eq:GenAsyEinfty} yields an efficient tool to evaluate the  
electric far field pattern due to a thin metallic nanowire.

\begin{remark}[Plasmonic resonances] \label{rem:spr}
  In the special case, when the reference cross-section $B'$ is an
  ellipse with semi axes of length $0<a\leq b<1$, we observe
  from~\eqref{eq:PolTenExplicit} and \eqref{eq:meps} that the norm of
  the polarization tensor is large, when either $\epsr\approx -a/b$ or
  $\epsr \approx -b/a$.
  For noble metals like silver or gold, $\imag(\epsr)>0$ is
  strictly positive, and therefore the polarization tensor does not
  become singular for any choice of $a/b$ (see
  Table~\ref{tab:MaterialParameters} for the relative electric
  permittivity $\eps_r$ of silver and gold at optical frequencies).
  However, according to~\eqref{eq:GenAsyEinfty} strong scattering
  occurs, whenever $\imag(\epsr)$ is small and $\real(\epsr)=-a/b$ or
  $\real(\epsr)=-b/a$.
  Following~\cite{AmmDenMil16} (see also~\cite{Gri14,NovHec12}) we call
  the corresponding frequencies $\fres$ \emph{plasmonic resonances}
  for $\Drho$. 
  
  Since $\real(\epsr)<-1$ for silver and gold across the optical band
  (see Table~\ref{tab:MaterialParameters}), and $0<a\leq b<1$ by
  assumption, i.e.,~$-\infty< -b/a\leq -1\leq -a/b\leq 0$, only the
  plasmonic resonance $\fres$ with $\real(\epsr(\fres))=-b/a$ can be
  excited for silver and gold nanowires with elliptical cross-sections
  at optical frequencies. 

  For general cross-sections $\Drho'$, plasmonic resonances can be
  determined by solving an eigenvalue problem for the
  Neumann-Poinc\'are operator on the boundary of the rescaled
  cross-section $B'=\Drho'/\rho$ (see~\cite{AmmDenMil16}). 
  The optimal design of two-dimensional shapes that resonate at
  particular frequencies has been considered
  in~\cite{AmmChoLiuZou15}.~\hfill$\lozenge$ 
\end{remark}

\section{Electromagnetic chirality}
\label{sec:Chirality}
In this section we briefly summarize the concept of electromagnetic
chirality.
For further details we refer to
\cite{AreHagHetKir18,FerFruRoc16,VavFer22}. 
An electric plane wave with \emph{direction of propagation}
$\bftheta\in\Sd$ and \emph{polarization}~$\bfA\in\C^3$,
which must satisfy $\bfA\cdot\bftheta=0$, is described by
the matrix $\Ei(\ph;\bftheta)\in\Cdd$ defined by 
\begin{equation*}
  \Ei(\bfx;\bftheta) \bfA
  \,=\, \bfA \,
  \rme^{\rmi k \bftheta \cdot \bfx} \,, \qquad \bfx\in\Rd \,,
\end{equation*}
We consider scattering of an electric plane wave at a thin
nanowire~$\Drho$ with relative electric permittivity~$\epsr$. 
Because of the linearity of \eqref{eq:ScatteringProblem} with respect
to the incident field, we can also express the scattered electric field
and the electric far field pattern by matrices $\Esrho(\ph;\bftheta)$
and $\Einftyrho(\ph;\bftheta)$, respectively.
Accordingly, the scattered field associated to an
\emph{electric Herglotz incident wave} 
\begin{equation}
  \label{eq:DefHerglotz}
  \Ei[\bfA](\bfx)
  \,:=\, \int_{\Sd} \bfA(\bftheta)
  \rme^{\rmi k \bftheta \cdot \bfx} \ds(\bftheta) \,,
  \qquad \bfx\in\Rd \,,
\end{equation}
with density $\bfA\in\LttSd$ is given by
\begin{equation*}
  \Esrho[\bfA](\bfx)
  \,=\, \int_{\Sd} \Esrho(\bfx;\bftheta) \bfA(\bfd) \ds(\bfd) \,,
  \qquad \bfx\in\Rd \,.
\end{equation*}
The \emph{electric far field operator}
$\FcalDrho: \LttSd \to \LttSd$, 
\begin{equation*}
  (\FcalDrho \bfA)(\xhat)
  \,:=\, \Einftyrho[\bfA](\xhat)
  \,=\, \int_{\Sd} \Einftyrho(\xhat;\bftheta) \bfA(\bftheta)
  \ds(\bftheta) \,,
  \qquad \xhat\in\Sd \,,
\end{equation*}
maps densities of electric Herglotz incident waves to the far field
patterns of the corresponding scattered electric fields. 
Since the kernel of this integral operator is smooth, it is compact
and of Hilbert-Schmidt class, i.e., $\FcalDrho\in\HS(\LttSd)$. 

Electromagnetic chirality describes different responses of scattering
objects to electromagnetic fields of positive and negative helicities.
An electric field $\Ei$, $\Etrho$, or $\Esrho$ in
$\Rd\setminus\ol{\Drho}$ is said to have \emph{helicity}~$\pm 1$ if it
is an eigenfunction of the operator $k^{-1}\curl$ associated to the
eigenvalue $\pm 1$.
Using the Riemann-Silberstein combinations~$\bfU\pm k^{-1}\curl\bfU$,
any solution to $\curl\curl\bfU - k^2\bfU=0$ in some subdomain
of~$\Rd$ can be decomposed into  a sum of two fields of helicity $+1$
and $-1$. 
The helicity of an electric Herglotz wave~$\Ei[\bfA]$ and of the
corresponding scattered electric field 
$\Esrho[\bfA]|_{\Rd\setminus\ol{\Drho}}$ is uniquely determined by the
density $\bfA$ and by the associated far field pattern
$\Einftyrho[\bfA]$, respectively.  
Decomposing $\LttSd = V^+\oplus V^-$ into the two orthogonal eigenspaces
\begin{equation}
  \label{eq:DefVpm}
  V^{\pm}
  \,:=\, \{ \bfA \pm \Ccal \bfA
  \;|\; \bfA \in \LttSd \} 
\end{equation}
of the self-adjoint linear operator $\Ccal: \LttSd\to\LttSd$,
\begin{equation*}
  (\Ccal \bfA)(\bftheta)
  \,:=\, \rmi \bftheta \times \bfA(\bftheta)
  \,, \qquad \bftheta \in \Sd \,,
\end{equation*}
it has been shown in \cite{AreHagHetKir18,FerFruRoc16} that
  \begin{align*}
    \Ei[\bfA] \,\text{ has helicity } \pm 1 
    & \qquad \text{if and only if} \qquad 
      \bfA \in V^\pm \,,\\
    \Esrho[\bfA]|_{\Rd\setminus\ol{\Drho}} \,\text{ has helicity } \pm 1 
    & \qquad \text{if and only if} \qquad 
      \FcalDrho \bfA \in V^\pm \,.
  \end{align*}
Using the orthogonal projections $\Pcal^\pm : \LttSd \to \LttSd$ onto 
$V^\pm$, we define the projected far field operators
$\FcalDrho^{cd} := \Pcal^c \FcalDrho \Pcal^d$ for $c,d \in \{+,-\}$.
Then, $\FcalDrho^{cd}$ describes the components with helicity $c$ of
all possible electric far field patterns corresponding to electric
Herglotz incident waves with helicity $d$. 
Therewith, the far field operator $\FcalDrho$ can be decomposed into
four blocks 
\begin{equation}
  \label{eq:decomposition}
  \FcalDrho
  \,=\, \FcalDrho^{++} + \FcalDrho^{+-}
  + \FcalDrho^{-+} + \FcalDrho^{--} \,.
\end{equation}
The following notion of electromagnetic chirality has been
introduced in \cite{FerFruRoc16} (see also \cite{AreHagHetKir18}).
Roughly speaking, the idea is to decide whether the information
content of the electric far field patterns corresponding to all
possible electric incident fields with positive helicity can be
reproduced using all possible electric incident fields of negative
helicity and vice versa, or not. 

\begin{definition}
  \label{def:emChirality}
  A scatterer $\Drho$ is 
  \emph{electromagnetically achiral} (or \emph{em-achiral}) if there 
  exist unitary operators~$\Ucal^{(j)}:\LttSd\to\LttSd$ satisfying
  $\Ucal^{(j)} \Ccal = -\Ccal\Ucal^{(j)}$, $j=1,\ldots,4$, such that 
  \begin{equation*}
    \FcalDrho^{++} \,=\, \Ucal^{(1)} \FcalDrho^{--} \Ucal^{(2)} \,, \qquad 
    \FcalDrho^{-+} \,=\, \Ucal^{(3)} \FcalDrho^{+-} \Ucal^{(4)} \,.
  \end{equation*}
  If this is not the case, then the scatterer $\Drho$ is
  \emph{electromagnetically chiral} (or \emph{em-chiral}). 
\end{definition}

For $c,d \in \{+,-\}$ let $(\sigma_j^{cd})_{j\in\N}$ denote the
singular values of~$\FcalDrho^{cd}$ in decreasing order and repeated 
with multiplicity. 
Then, Definition~\ref{def:emChirality} says that a scattering
object $\Drho$ is em-achiral if and only if 
\begin{equation*}
  (\sigma_j^{++})_{j\in\N} \,=\, (\sigma_j^{--})_{j\in\N}
  \qquad\text{and}\qquad
  (\sigma_j^{+-})_{j\in\N} \,=\, (\sigma_j^{-+})_{j\in\N} \,.
\end{equation*}
Accordingly, it has been proposed in \cite{FerFruRoc16} to quantify the
degree of em-chirality of a scattering object by means of the distance
of the corresponding sequences of singular values.
In this work, we discuss two possible choices for such an
\emph{em-chirality measure} of a scatterer~$\Drho$ with associated far
field operator~$\FcalDrho$.  
The measure~$\chitwo$ from \cite{FerFruRoc16}, which is defined by 
\begin{subequations}
  \label{eq:ChiralityMeasures}
  \begin{equation}
    \label{eq:Chi2}
    \chitwo(\FcalDrho) 
    \,:=\, \Bigl( 
      \bigl\| (\sigma_j^{++})_{j\in\N}-(\sigma_j^{--})_{j\in\N} \bigr\|_{\ell^2}^2 
      +\bigl\|(\sigma_j^{+-})_{j\in\N}-(\sigma_j^{-+})_{j\in\N} \bigr\|_{\ell^2}^2
      \Bigr)^{\frac{1}{2}} \,,
    \end{equation}
    and a smooth relaxation $\chiHS$ of $\chitwo$ from \cite{Hag19},
    which is given by 
    \begin{equation}
      \label{eq:ChiHS}
      \chiHS(\FcalDrho) 
      \,:=\, \bigl( 
      \bigl( \|\FcalDrho^{++}\|_\HS
      - \|\FcalDrho^{--}\|_\HS \bigr)^2 
      + \bigl(\|\FcalDrho^{-+}\|_\HS
      - \|\FcalDrho^{+-}\|_\HS \bigr)^2 
      \bigr)^{\frac12} \,.
  \end{equation}
\end{subequations}
Here, $\|\ph\|_\HS$ denotes the Hilbert-Schmidt norm.
The Hilbert-Schmidt norm $\|\FcalDrho\|_\HS$ of the far field
operator~$\FcalDrho$ is sometimes called the
\emph{total interaction cross section} of the scattering object
$\Drho$. 
It has been shown in \cite{AreHagHetKir18,FerFruRoc16,Hag19} that
\begin{equation}
  \label{eq:ChiBounds}
  0
  \,\leq\, \chiHS(\FcalDrho) 
  \,\leq\, \chitwo(\FcalDrho)
  \,\leq\, \|\FcalDrho\|_\HS \,.
\end{equation}
For the lower and upper bounds in \eqref{eq:ChiBounds} we have that
  \begin{align*}
    \chiHS(\FcalDrho) 
    &\,=\, 0 
      \phantom{\| \FcalDrho \|_\HS}\;\quad\text{if and only if}\qquad 
      \chitwo(\FcalDrho) 
      \,=\, 0 \,,\\
    \chiHS(\FcalDrho) 
    &\,=\, \| \FcalDrho \|_\HS 
      \phantom{0}\;\quad\text{if and only if}\qquad 
      \chitwo(\FcalDrho) 
      \,=\, \| \FcalDrho \|_\HS \,.
  \end{align*}
The scatterer $\Drho$ is em-achiral if and only if
$\chitwo(\FcalDrho)=\chiHS(\FcalDrho)=0$, and it is said to be
\emph{maximally em-chiral} if
$\chitwo(\FcalDrho)=\chiHS(\FcalDrho)=\|\FcalDrho\|_\HS$. 
The latter is equivalent to the fact that the scatterer $\Drho$ does
not scatter incident fields of either positive or negative helicity at
all.

\section{Optimal shape design}
\label{sec:ShapeOptimization}
We consider the shape optimization problem to design thin metallic
nanowires that are as close to being maximally em-chiral at optical
frequencies as possible. 
To this end, we maximize the normalized em-chirality measures
$\chitwo(\FcalDrho)/\|\FcalDrho\|_\HS$ and
$\chiHS(\FcalDrho)/\|\FcalDrho\|_\HS$, which are bounded between $0$
and $1$ by~\eqref{eq:ChiBounds}, with respect to the shape of $\Drho$.
We run the shape optimization at a fixed frequency $\fopt$, which in
turn determines the relative electric permittivity $\epsr$ of the
nanowire (see Table~\ref{tab:MaterialParameters} for silver and gold),
prior to the optimization process. 
We also choose the rescaled cross-section $B'=\Drho'/\rho$ as well as
the length $L$ of the nanowire in advance, and then we optimize the
shape of the spine curve $\Gamma$ and the geometry adapted
frame~$\Vptheta := [ \tp | \npada | \bpada ]$ that describes the
twisting of the cross-section of the nanowire along the spine curve. 
 
\subsection{Objective functionals and shape derivatives}
The computational complexity of the iterative shape optimization
scheme that we develop below is dominated by the evaluation of
electric far field operators and of shape derivatives of electric
far field operators corresponding to thin metallic nanowires in each
iteration step. 
Using the asymptotic representation formula~\eqref{eq:GenAsyEinfty},
the electric far field operator~$\FcalDrho$ associated to a thin
metallic nanowire~$\Drho$ can be approximated by the
operator~$\TcalDrho: \LttSd \to \LttSd$, 
\begin{equation}
  \label{eq:DefTcal}
  (\TcalDrho\bfA)(\xhat)
  \,:=\, |B'| (k\rho)^2 \int_\Gamma 
  (\epsr-1) \rme^{-\rmi k \xhat\cdot \bfy}
  \bigl((\xhat\times\I_3)\times\xhat\bigr)
  \Meps(\bfy) \Ei[\bfA](\bfy) \ds(\bfy) \,,
  \qquad \xhat\in\Sd \,.
\end{equation}
Theorem~\ref{thm:GenAsy} and the definition of the electric Herglotz
wave $\Ei[\bfA]$ in \eqref{eq:DefHerglotz} show that 
\begin{equation}
  \label{eq:FFExp}
  \FcalDrho
  \,=\, \TcalDrho + o\bigl((k\rho)^2\bigr)
  \qquad \text{as } \rho \to 0 \,,
\end{equation}
where the term $o((k\rho)^2)$ in \eqref{eq:FFExp} is such that
$\|o((k\rho)^2)\|_{\HS}/(k\rho)^2$ converges to zero.

Since we have fixed the rescaled cross-section $B' := \Drho'/\rho$,
the support $\Drho$ of a nanowire is uniquely determined by a
parametrization $\bfp$ of the spine curve $\Gamma$ and an associated 
geometry adapted frame $(\tp, \npada, \bpada)$ as in
\eqref{eq:DrhoAdapted}.  
Accordingly, we introduce a set of admissible parametrizations for
supports of thin nanowires by
\begin{multline}
  \label{eq:DefPcal}
  \Uad
  \,:=\, \bigl\{ \bigl( \bfp, \Vptheta \bigr)
  \in C^3([0,L],\Rd) \times C^2([0,L],\SOd) \;\big|\; \\
  \bfp([0,L]) \text{ is simple, } \bfp'(s) \neq 0 , 
  \text{ and } \Vptheta(s) \bfe_1 = \bfp'(s)/|\bfp'(s)|
  \text{ for all } s\in [0,L] \bigr\} \,.
\end{multline}
Here, $\bfe_1 := (1,0,0)^\trans$ denotes the first standard basis
vector in $\Rd$. 
Therewith, we define a nonlinear
operator~$\bfTrho : \Uad \to \text{HS}(\LttSd)$ that maps admissible
parametrizations of supports of thin nanowires~$\Drho$ to the leading
order term $\TcalDrho$ of the associated far field operator
$\FcalDrho$ in \eqref{eq:FFExp} by 
\begin{equation}
  \label{eq:TrhoPTheta}
  \bfTrho(\bfp,\Vptheta)
  \,:=\, \TcalDrho \,, \qquad \bigl( \bfp, \Vptheta \bigr) \in\Uad \,.
\end{equation}
The projected operators
$(\bfTrho(\bfp,\Vptheta))^{cd} := \Pcal^c\bfTrho(\bfp,\Vptheta)\Pcal^d$
for any combination of $c,d \in \{+,-\}$ yield a decomposition of
$\bfTrho(\bfp,\Vptheta)$ as in \eqref{eq:decomposition}.
Accordingly, we can apply the em-chirality measures from
\eqref{eq:ChiralityMeasures} directly to~$\bfTrho(\bfp,\Vptheta)$. 

We consider the two objective functionals
\begin{equation}
  \label{eq:J2}
  \Jtwo: \Uad \to [0,1] \,, \qquad 
  \Jtwo(\bfp,\Vptheta)
  \,:=\, \frac{\chitwo(\bfTrho(\bfp,\Vptheta))}
  {\| \bfTrho(\bfp,\Vptheta) \|_\HS} \,,
\end{equation}
and 
\begin{equation}
  \label{eq:JHS}
  \JHS: \Uad \to [0,1] \,, \qquad 
  \JHS(\bfp,\Vptheta)
  \,:=\, \frac{\chiHS(\bfTrho(\bfp,\Vptheta))}
  {\| \bfTrho(\bfp,\Vptheta) \|_\HS} \,.
\end{equation}
Both functionals $\Jtwo$ and $\JHS$ attain their maximum value $1$ for
a maximally em-chiral metallic nanowire, and their minimum
value $0$ for an em-achiral nanowire. 
Since~$\Jtwo$ is not differentiable and hence less suitable for a
gradient based optimization, we focus in the following on maximizing
the nonlinear functional~$\JHS$.
However, we will also specify the corresponding values of~$\Jtwo$
for comparison (e.g., with the results from \cite{GarEtAl21}) in the
numerical examples in Section~\ref{sec:NumericalResults} below. 

\begin{remark} \label{rem:rho}
  It is important to observe that for any $\bfA\in\LttSd$ the electric
  far field pattern $\TcalDrho\bfA$ from~\eqref{eq:DefTcal} is
  homogeneous with respect to the squared radius $\rho^2$ of the
  cross-section $\Drho'=\rho B'$ of the support of the thin nanowire.
  Thus, the same is true
  for~$\chitwo (\bfTrho(\bfp,\Vptheta))$, 
  $\chiHS (\bfTrho(\bfp,\Vptheta))$, 
  and~$\|\bfTrho(\bfp,\Vptheta)\|_\HS$ with~$(\bfp,\Vptheta)\in\Uad$.
  In particular, the rescaled em-chirality measures $\Jtwo(\bfp,\Vptheta)$
  and $\JHS(\bfp,\Vptheta)$ are independent of $\rho$.
  This means that the specific value of~$\rho$ does not
  affect the result of the optimization procedure considered below.
  However, in order for the leading order term $\TcalDrho$
  in~\eqref{eq:FFExp} to be an acceptable approximation of
  $\FcalDrho$, the radius $\rho$ of the cross section $\Drho'$ of the
  thin nanowire has to be sufficiently small, which we assume
  throughout this work.~\hfill$\lozenge$ 
\end{remark}

We discuss the optimization problem
\begin{equation}
  \label{eq:OptimizationProblem}
  \text{find} \qquad
  \argmin\limits_{(\bfp,\Vptheta) \in \Uad}
  \left(-\JHS(\bfp,\Vptheta)\right) \qquad
  \text{subject to } \quad |\Gamma| = L 
\end{equation}
at some prescribed frequency $\fopt$, for some fixed rescaled
cross-section $B'=\Drho'/\rho\tm B_1'(0)\tm\Rtwo$, and for some
prescribed length~$L>0$ of the nanowire. 
The frequency $\fopt$ determines the wave number $k$ and the
relative electric permittivity $\eps_r$. 
Therewith, the rescaled cross-section $B'$ defines the two-dimensional
electric polarization tensor $\meps\in\C^{2\times 2}$ according
to~\eqref{eq:mepsij}--\eqref{eq:wj}. 

It has been shown already in \cite{Hag19} that the em-chirality
measure $\chiHS$ from \eqref{eq:ChiralityMeasures}
is differentiable on 
\begin{equation*}
  X 
  \,:=\, \bigl\{ \Gcal \in \HS(\LttSd)
  \;\big|\; \chiHS(\Gcal)\not=0 \text{ and }
  \|\Gcal^{cd}\|_\HS>0 \,,\,\, c,d \in \{+,-\} \bigr\} \,.
\end{equation*}
For any $\Gcal\in X$ and $\Hcal\in\HS(\LttSd)$ the derivative is given 
by 
\begin{equation*}
  (\chiHS)'[\Gcal]\Hcal 
  \,=\, 
  \frac{\real\langle\Gcal,\Hcal\rangle_\HS 
    - \sum_{c,d\in\{+,-\}} \real\langle\Gcal^{cd},\Hcal^{cd}\rangle_\HS
    \frac{\|\Gcal^{\ol{c}\ol{d}}\|_\HS}
    {\|\Gcal^{cd}\|_\HS}}{\chiHS(\Gcal)} \,,
\end{equation*}
where $\ol{c}:=-c$ and $\ol{d}:=-d$.
Accordingly, the Fr\'echet derivative of the objective functional
$\JHS$ satisfies 
\begin{equation*}
  \begin{split}
    \JHS'[\bfp,\Vptheta](\bfh,\phi)
    &\,=\,
    \frac{(\chiHS)'\bigl[\bfTrho(\bfp,\Vptheta)\bigr]
      (\bfTrho'[\bfp,\Vptheta](\bfh,\phi))}
    {\|\bfTrho(\bfp,\Vptheta)\|_\HS} \\
    &\phantom{\,=\,}
    - \frac{\chiHS\bigl(\bfTrho(\bfp,\Vptheta)\bigr) 
      \real \bigl\langle \bfTrho(\bfp,\Vptheta), 
      \bfTrho'[\bfp,\Vptheta](\bfh,\phi) \bigr\rangle_\HS}
    {\|\bfTrho(\bfp,\Vptheta)\|_\HS^3} \,,
  \end{split}
\end{equation*}
where $\bfTrho'[\bfp,\Vptheta]$ denotes the Fr\'echet derivative of the
operator $\bfTrho$ at $(\bfp,\Vptheta)\in\Uad$.
It remains to show that $\bfTrho$ is indeed Fr\'echet differentiable at
$(\bfp,\Vptheta)\in\Uad$ and to determine its Fr\'echet derivative. 

The admissible set $\Uad$ in \eqref{eq:DefPcal} is not a vector space,
but $\Uad$ can be parametrized locally around any admissible
$(\bfp,\Vptheta)\in\Uad$ with
$\Vptheta = \bigl[\tp\big|\npada\big|\bpada\big]$ as follows.
Suppose that~$\delta_\bfp,\delta_\theta>0$ are sufficiently small.
Then, an open neighborhood of $(\bfp,\Vptheta)$ in $\Uad$ is given by 
\begin{equation}
  \label{eq:UadLocPar}
  \bigl\{ \bigl( \bfp+\bfh, V_{\bfp+\bfh,\theta+\phi} \bigr)
  \;\big|\; \bfh \in C^3([0,L],\Rd) \,,\; \|\bfh\|_{C^3} < \delta_\bfp \,,\;
  \phi\in C^2([0,L],\R) \,,\; \|\phi\|_{C^2} < \delta_\theta \bigr\} \,,
\end{equation}
where $V_{\bfp+\bfh,\theta+\phi}
:= [ \bft_{\bfp+\bfh} | \bfn_{\bfp+\bfh,\theta+\phi}
| \bfb_{\bfp+\bfh,\theta+\phi} ] \in C^2([0,L],\SOd)$ with
\begin{subequations}
  \label{eq:FrameTranslated1}
  \begin{align}
    \bft_{\bfp+\bfh}
    &\,:=\, \frac{\bfp'+\bfh'}{|\bfp'+\bfh'|} \,,\\
    \bfn_{\bfp+\bfh,\theta+\phi}
    &\,:=\, (\tp\cdot\bft_{\bfp+\bfh}) \bfn_{\bfp,\theta+\phi}
      - \frac{ \bfb_{\bfp,\theta+\phi}\cdot\bft_{\bfp+\bfh}}
      {1+\bft_p\cdot\bft_{\bfp+\bfh}}
      (\bft_p\times\bft_{\bfp+\bfh})
      - ( \bfn_{\bfp,\theta+\phi}\cdot\bft_{\bfp+h}) \tp \,,\\
    \bfb_{\bfp+\bfh,\theta+\phi}
    &\,:=\, (\tp\cdot\bft_{\bfp+\bfh}) \bfb_{\bfp,\theta+\phi}
      + \frac{ \bfn_{\bfp,\theta+\phi}\cdot\bft_{\bfp+\bfh}}
      {1+\bft_p\cdot\bft_{\bfp+\bfh}}
      (\bft_p\times\bft_{\bfp+\bfh})
      - ( \bfb_{\bfp,\theta+\phi}\cdot\bft_{\bfp+h}) \tp \,,
  \end{align}
\end{subequations}
and
\begin{equation}
  \label{eq:FrameTranslated2}
  \bigl[\bfn_{\bfp,\theta+\phi} \big| \bfb_{\bfp,\theta+\phi}\bigr]
  \,:=\, \bigl[\cos(\phi) \bfn_{\bfp,\theta}
  + \sin(\phi) \bfb_{\bfp,\theta} \big|
  -\sin(\phi) \bfn_{\bfp,\theta}
  + \cos(\phi) \bfb_{\bfp,\theta} \bigr] \,.
\end{equation}
We note that 
$(\bft_{\bfp+\bfh},\bfn_{\bfp+\bfh,\theta+\phi},\bfb_{\bfp+\bfh,\theta+\phi})$
is an orthogonal frame along the perturbed curve $\bfp+\bfh$ by
construction. 

Before we establish the Fr\'echet derivative of $\bfTrho$ in
Theorem~\ref{thm:FrechetDerivativeTrho} below, we discuss the
shape derivative of the polarization tensor $\Meps$ using the explicit
representation in \eqref{eq:PolTenExplicit}. 

\begin{lemma}
  \label{lmm:FrechetDerivativeM}
  The mapping $\Meps: \Uad \to C([0,L],\Cdd)$ defined by
  $\Meps(\bfp,\Vptheta) 
  := \Meps_{\bfp,\Vptheta}
  := \Vptheta \Meps\Vptheta^\trans$
  is Fr\'echet differentiable.
  Its Fr\'echet derivative at $(\bfp,\Vptheta) \in \Uad$ with
  respect to the local parametrization of~$\Uad$ in
  \eqref{eq:UadLocPar}--\eqref{eq:FrameTranslated2} is given
  by~$(\Meps_{\bfp,\Vptheta})': C^3([0,L],\Rd)\times C^2([0,L],\R)
  \to C([0,L],\Cdd)$ with
  \begin{equation}
    \label{eq:ShapeDerivativePolTen}
    (\Meps_{\bfp,\Vptheta})'(\bfh,\phi)
    \,=\, V'_{\bfp,\theta}(\bfh,\phi) \Meps \Vptheta^\top
    + \Vptheta \Meps (V'_{\bfp,\theta}(\bfh,\phi))^\top \,,
  \end{equation}
  where the matrix-valued function $V'_{\bfp,\theta}(\bfh,\phi)$
  satisfies 
  \begin{equation*}
      V'_{\bfp,\theta}(\bfh,\phi)
      \,=\, \biggl[
      \frac{\bfh'\cdot\npada}{|\bfp'|} \npada
      + \frac{\bfh'\cdot \bpada}{|\bfp'|} \bpada \,\bigg|
      - \frac{\bfh' \cdot \npada}{|\bfp'|} \tp
      + \phi\, \bpada \,\bigg|
      - \frac{\bfh'\cdot \bpada}{|\bfp'|} \tp
      - \phi\, \npada \biggr] \,.
  \end{equation*}
\end{lemma}

\begin{proof}
  Using Taylor's theorem we find that
  \begin{equation}
    \label{eq:Taylortph}
    \bft_{\bfp+\bfh,\theta}
    \,=\, \bft_{\bfp,\theta} + \frac{1}{|\bfp'|} \bigl(
    (\bfh'\cdot\bfn_{\bfp,\theta})\bfn_{\bfp,\theta}
    + (\bfh'\cdot\bfb_{\bfp,\theta})\bfb_{\bfp,\theta}
    \bigr) + O(\|\bfh\|^2_{C^2([0,L],\Rd)}) 
  \end{equation}
  and
    \begin{align*}
      \bfn_{\bfp,\theta+\phi}
      \,=\, \bfn_{\bfp,\theta} + \phi \bfp_{\bfp,\theta}
      + O(\|\phi\|^2_{C([0,L],\Rd)}) \,,\\
      \bfb_{\bfp,\theta+\phi}
      \,=\, \bfb_{\bfp,\theta} - \phi \bfn_{\bfp,\theta}
      + O(\|\phi\|^2_{C([0,L],\Rd)}) \,.
    \end{align*}
  Furthermore, substituting \eqref{eq:Taylortph} into
  \eqref{eq:FrameTranslated1} gives 
  \begin{align*}
    \bfn_{\bfp+\bfh,\theta}
    &\,=\, \bfn_{\bfp,\theta} - \frac{1}{|\bfp'|}
      (\bfh'\cdot\bfn_{\bfp,\theta})\bft_{\bfp,\theta}
      + O(\|\bfh\|^2_{C^2([0,L],\Rd)}) \,,\\
    \bfb_{\bfp+\bfh,\theta}
    &\,=\, \bfb_{\bfp,\theta} - \frac{1}{|\bfp'|}
      (\bfh'\cdot\bfb_{\bfp,\theta})\bft_{\bfp,\theta}
      + O(\|\bfh\|^2_{C^2([0,L],\Rd)}) \,.
  \end{align*}
  Accordingly, the partial derivatives
    \begin{align*}
      \di_\bfp \Vptheta(\bfh)
      &\,=\, \frac{1}{|\bfp'|} \bigl[
        (\bfh'\cdot\npada)\npada + (\bfh'\cdot \bpada)\bpada \,\big|
        -(\bfh' \cdot \npada)\tp \,\big|
        -(\bfh'\cdot \bpada)\tp \bigr] \,, \\
      \di_\theta \Vptheta(\phi)
      &\,=\, \phi\, \bigl[ 0 \,\big|\, \bpada \,\big| - \npada \bigr]
    \end{align*}
  satisfy
    \begin{align*}
      \| V_{\bfp+\bfh,\theta}-\Vptheta
      -\di_\bfp \Vptheta(\bfh)\|_{C([0,L],\Rdd)}
      &\,\leq\, C \|\bfh\|^2_{C^2([0,L],\Rd)} \,, \\
      \| V_{\bfp,\theta+\phi}-\Vptheta
      -\di_\theta \Vptheta(\phi)\|_{C([0,L],\Rdd)}
      &\,\leq\, C \|\phi\|^2_{C([0,L],\Rd)} \,. 
    \end{align*}
  Thus, using Taylor's theorem once more, we obtain that there exists
  $\delta\in[0,1]$ such that 
  \begin{equation*}
    \begin{split}
      &\| V_{\bfp+\bfh,\theta+\phi}-\Vptheta
      - V'_{\bfp,\theta}(\bfh,\phi)\|_{C([0,L],\Rdd)}\\
      &\,\leq\, \| V_{\bfp+\bfh,\theta+\phi}-V_{\bfp+\bfh,\theta}
      - \di_\theta \Vptheta(\phi)\|_{C([0,L],\Rdd)}
      + \| V_{\bfp+\bfh,\theta}-\Vptheta
      -\di_\bfp \Vptheta(\bfh)\|_{C([0,L],\Rdd)}\\
      &\,=\, \| \di_\theta V_{\bfp+\bfh,\theta+\delta\phi}(\phi)
      - \di_\theta \Vptheta(\phi)\|_{C([0,L],\Rdd)}
      + \| V_{\bfp+\bfh,\theta}-\Vptheta
      - \di_\bfp \Vptheta(\bfh)\|_{C([0,L],\Rdd)} \,.
    \end{split}
  \end{equation*}
  The continuity of $\di_\theta \Vptheta(\phi)$ with respect
  to $\bfp$ and $\theta$ shows that $\Vptheta$ is Fr\'echet
  differentiable with derivative~$V'_{\bfp,\theta}$.
  The Fr\'echet differentiability of the polarization tensor and
  \eqref{eq:ShapeDerivativePolTen} are now a consequence of the
  product rule. 
\end{proof}

In the next theorem we establish the Fr\'echet derivative of
$\bfTrho$ at $(\bfp,\Vptheta) \in \Uad$.

\begin{theorem}
  \label{thm:FrechetDerivativeTrho}
  The nonlinear operator $\bfTrho$ from \eqref{eq:TrhoPTheta} is Fr\'echet
  differentiable from $\Uad$ to $\HS(\LttSd)$.
  Its Fr\'echet derivative at
  $(\bfp,\Vptheta) \in \Uad$ with respect to the local parametrization
  of $\Uad$ in \eqref{eq:UadLocPar}--\eqref{eq:FrameTranslated2} is
  given by 
  $\bfTrho'[\bfp,\Vptheta]:\, C^3([0,L],\Rd) \times C^2([0,L],\R)
  \to \HS(\LttSd)$ with
  \begin{equation*}
    \bfTrho'[\bfp,\Vptheta] (\bfh,\phi)
    \,=\, |B'| (k\rho)^2 
    (\epsr-1) \sum_{j=1}^4 \bfT_{\rho,j}'[\bfp,\Vptheta] (\bfh,\phi) \,, 
  \end{equation*}
  where, for any $\bfA\in\LttSd$, 
    \begin{align*}
      \bigl(\bigl(\bfT_{\rho,1}'
      [\bfp,\Vptheta] (\bfh,\phi)\bigr)\bfA\bigr) (\xhat)
      &\,=\, - \int_0^L 
        \rmi k (\xhat\cdot\bfh) \rme^{-\rmi k \xhat\cdot\bfp}
        \left((\xhat\times\I_3)\times\xhat\right)
        \M_{\bfp,\theta}^\eps
        \Ei[\bfA](\bfp)
        |\bfp'| \dt \,,\\ 
      \bigl(\bigl(\bfT_{\rho,2}'
      [\bfp,\Vptheta] (\bfh,\phi)\bigr)\bfA\bigr) (\xhat)
      &\,=\, \int_0^L
        \rme^{-\rmi k \xhat\cdot\bfp}
        ((\xhat\times\I_3)\times\xhat)
        (\Meps_{\bfp,\Vptheta})'(\bfh,\phi)
        \Ei[\bfA](\bfp) 
        |\bfp'| \dt \,,\\ 
      \bigl(\bigl(\bfT_{\rho,3}'
      [\bfp,\Vptheta] (\bfh,\phi)\bigr)\bfA\bigr) (\xhat)
      &\,=\, \int_0^L 
        \rme^{-\rmi k \xhat\cdot\bfp}
        \left((\xhat\times\I_3)\times\xhat\right) \,
        \M_{\bfp,\theta}^\eps \,
        \bigl(\Ei[\bfA]\bigr)'[\bfp,\Vptheta] (\bfh,\phi) \,
        |\bfp'| \dt \,,\\ 
      \bigl(\bigl(\bfT_{\rho,4}'
      [\bfp,\Vptheta] (\bfh,\phi)\bigr)\bfA\bigr) (\xhat)
      &\,=\, \int_0^L 
        \rme^{-\rmi k \xhat\cdot\bfp}
        \left((\xhat\times\I_3)\times\xhat\right) \,
        \M_{\bfp,\theta}^\eps \,
        \Ei[\bfA](\bfp) \,
        \frac{\bfp'\cdot \bfh'}{|\bfp'|} \dt \,.
    \end{align*}
\end{theorem}

\begin{proof}
  Using Lemma~\ref{lmm:FrechetDerivativeM}, this result can be shown as
  in \cite[Thm.~4.2]{AreGriKno21}, where the Fr\'echet
  differentiability of~$\bfTrho$ for thin dielectric nanowires with
  circular cross-sections has been established. 
\end{proof}

\begin{remark}
  \label{rem:DiscretisationT}
  For the numerical implementation of the operator
  $\bfTrho(\bfp,\Vptheta)$ and of its Fr\'echet derivative
  $\bfTrho'[\bfp,\Vptheta]$, we use truncated series representations of 
  these operators with respect to a suitable complete orthonormal
  system in $\LttSd$. 
  Denoting by $\Umn$, $\Vmn$, $m=-n,\ldots,n$, vector spherical
  harmonics of order $n>0$ (see, e.g., \cite[p.~248]{ColKre19}), 
  we define the \emph{circularly polarized vector spherical harmonics}
  of order $n$ by 
  \begin{equation*}
    \Amn
    \,:=\, \frac{1}{\sqrt2} (\Umn+\rmi\Vmn) 
    \qquad\text{and}\qquad
    \Bmn
    \,:=\, \frac{1}{\sqrt2} (\Umn-\rmi\Vmn) \,.
  \end{equation*}
  Then $\Amn$ and $\Bmn$, $m=-n,\dots, n$, $n=1,2,\dots$, form a
  complete orthonormal system of the subspaces~$\Vplus$ and $\Vminus$
  from~\eqref{eq:DefVpm}, respectively.
  The main reason for using this orthonormal basis is that the
  orthogonal projections $\Pcal^\pm$ onto $V^\pm$, required
  to evaluate the em-chirality measure from \eqref{eq:ChiralityMeasures} in
  the shape optimization scheme, can be obtained without any further
  computation directly from corresponding series expansions. 
  
  In Lemma~3.2 and Remark~4.3 of \cite{AreGriKno21}, formulas for the
  coefficients in the series representations
  of~$\bfTrho(\bfp,\Vptheta)$ and of its Fr\'echet derivative
  $\bfTrho'[\bfp,\Vptheta](\bfh,\phi)$ in terms of these circularly
  polarized vector spherical harmonics have been developed for the
  corresponding operators associated to thin dielectric nanowires with
  circular cross-sections. 
  Observing that the different material properties and the more
  general twisting cross-sections considered in this work only affect
  the electric polarization tensor and its Fr\'echet derivative in
  $\bfTrho(\bfp,\Vptheta)$ and $\bfTrho'[\bfp,\Vptheta](\bfh,\phi)$, 
  the results from \cite{AreGriKno21} can immediately be applied to
  obtain formulas for the coefficients in the series representations
  of the operators considered in the present work as well.
  
  In the numerical implementation these series expansions have to be
  truncated at some maximal degree~${n=N\in\N}$. 
  The analysis from \cite{GriSyl18} suggests to choose the truncation
  index $N$ such that $N\gtrsim kR$, where~$\BR$ denotes the smallest
  ball centered at the origin that circumscribes the
  nanowire~$\Drho$. 
  Accordingly, we consider discrete approximations
  $\bfTrhoN(\bfp,\Vptheta) \in \C^{Q\times Q}$ and 
  $\bfTrhoN'[\bfp,\Vptheta] (\bfh,\phi)\in \C^{Q\times Q}$ with~$Q =
  2N(N+2)$.~\hfill$\lozenge$ 
\end{remark}

\subsection{Regularization and numerical implementation}
In the numerical optimization scheme we will consider spine curves
$\Gamma$ that are parametrized by three-dimensional cubic not-a-knot
splines ${\bfp:[0,L]\to \Rd}$ with respect to a partition 
\begin{equation*}
  \tri
  \,=\, \{ 0 = t_1 < t_2 < \dots < t_n = L \} \subset [0,L] \,.
\end{equation*}
Throughout, we denote by $\Stri$ and $(\Stri)^3$ the space of
one-dimensional and three-dimensional cubic not-a-knot splines with
respect to this partition, respectively. 

To evaluate local minimizers of the constrained optimization problem
\eqref{eq:OptimizationProblem}, we approximate the latter by an
unconstrained optimization problem, where we include the length
constraint $|\Gamma|=L$ via the penalty term
\begin{equation}
  \label{eq:DefPsi1}
  \Psi_1: \Uad \to \R \,, \qquad
  \Psi_1(\bfp,\Vptheta)
  \,:=\, \sum_{j=1}^{n-1} \left| \frac{1}{n-1}
    - \frac{1}{L} \int_{t_j}^{t_{j+1}}|\bfp'(t)| \dt \right|^2 \,.
\end{equation}
Besides enforcing the length constraint, this term will also
promote uniformly distributed nodes along the spline representing the
spine curve during the optimization process. 

We use two further regularization terms to stabilize the optimization.
The functional
\begin{equation}
  \label{eq:DefPsi2}
  \Psi_2: \Uad \to \R \,, \qquad
  \Psi_2(\bfp,\Vptheta)
  \,:=\, \frac{1}{L} \int_0^L \kappa^2(t) |\bfp'(t)| \dt \,,
\end{equation}
where $\kappa$ is the curvature of $\Gamma$ parametrized by $\bfp$
from \eqref{eq:DefCurvature}, prevents the optimal nanowire from
being too strongly entangled.
Similarly, the term
\begin{equation}
  \label{eq:DefPsi3}
  \Psi_3: \Uad \to \R \,, \qquad
  \Psi_3(\bfp,\Vptheta) 
  \,:=\, \frac{1}{L} \int_0^L \beta^2(t) |\bfp'(t)| \dt \,,
\end{equation}
where $\beta$ is the twist rate of the geometry adapted frame
$(\tp,\npada,\bpada)$ parametrized by $(\bfp,\Vptheta)$ from
\eqref{eq:TwistRate}, penalizes strong twisting of the cross-section
of the optimal nanowire along its spine curve.

Adding $\alpha_1\Psi_1$, $\alpha_2\Psi_2$, and $\alpha_3\Psi_3$ with
some suitable regularization parameters $\alpha_1,\alpha_2,\alpha_3>0$
to the objective functional in~\eqref{eq:OptimizationProblem}, we
obtain the regularized objective functional $\Phi: \Uad \to \R$ given by
\begin{equation}
  \label{eq:DefPhi}
  \Phi(\bfp,\Vptheta)
  \,:=\, -\JHS(\bfp,\Vptheta)
  + \alpha_1 \Psi_1(\bfp,\Vptheta)
  + \alpha_2 \Psi_2(\bfp,\Vptheta)
  + \alpha_3 \Psi_3(\bfp,\Vptheta) \,.
\end{equation}
Accordingly, we consider the unconstrained optimization problem
\begin{equation}
  \label{eq:UnconstrainedOptimizationProblem}
  \text{find} \qquad
  \argmin\limits_{(\bfp,\Vptheta) \in \Uad}
  \Phi(\bfp,\Vptheta) .
\end{equation}

Below we apply a quasi-Newton scheme to solve a finite dimensional
approximation of \eqref{eq:UnconstrainedOptimizationProblem}
numerically.
In the next lemma we collect the Fr\'echet derivatives of the penalty
terms $\Psi_1$, $\Psi_2$, and $\Psi_3$, which are required by this
algorithm. 

\begin{lemma} 
  \label{lmm:DerivativePsi1Psi2}
  The penalty terms $\Psi_1$, $\Psi_2$, and $\Psi_3$ from
  \eqref{eq:DefPsi1}--\eqref{eq:DefPsi3} are Fr\'echet differentiable. 
  Their Fr\'echet derivatives at~$(\bfp,\Vptheta)\in\Uad$ with respect
  to the local parametrization of $\Uad$ in 
  \eqref{eq:UadLocPar}--\eqref{eq:FrameTranslated2} are given
  by~$\Psi_1'[\bfp,\Vptheta]: C^3([0,L],\Rd)\times C^2([0,L],\R)\to\R$
  with 
  \begin{equation*}
    \Psi_1'[\bfp,\Vptheta](\bfh,\phi)
    \,=\, \di_\bfp \Psi_1[\bfp,\Vptheta](\bfh)
    \,=\, -2 \sum_{j=1}^{n-1} \biggl(
    \int_{t_j}^{t_{j+1}} \frac{\bfp' \cdot \bfh'}{|\bfp'|}\dt
    \biggr) 
    \biggl( \frac{1}{n-1}
    - \frac{1}{L} \int_{t_j}^{t_{j+1}}|\bfp'| \dt \biggr) \,,
  \end{equation*}
  by $\Psi_2'[\bfp,\Vptheta]: C^3([0,L],\Rd)\times C^2([0,L],\R)\to\R$
  with
  \begin{equation*}
    \begin{split}
      \Psi_2'[\bfp,\Vptheta](\bfh,\phi)
      &\,=\, \di_\bfp \Psi_2[\bfp,\Vptheta](\bfh) \\ 
      &\,=\, \frac{1}{L} \int_0^L \bigg( 
      2 \frac{\bfp''\cdot \bfh''}{|\bfp'|^3} 
      -3 \frac{|\bfp''|^2 (\bfp'\cdot\bfh')}{|\bfp'|^5} \\
      &\phantom{\,=\,\int_{0}^{1} \! \bigg(}
      -2 \frac{\bigl(\bfp'\cdot \bfh'' 
        + \bfp''\cdot \bfh' \bigr) (\bfp'\cdot\bfp'')}
      {|\bfp'|^5}
      + 5 \frac{(\bfp'\cdot \bfh')(\bfp'\cdot \bfp'')^2}
      {|\bfp'|^7} \bigg)
      \dt \,,
    \end{split}
  \end{equation*}
  and by
  $\Psi_3'[\bfp,\Vptheta]: C^3([0,L],\Rd)\times C^2([0,L],\R)\to\R$
  with
    \begin{equation*}
      \Psi_3'[\bfp,\Vptheta](\bfh,\phi)
      \,=\, \di_\bfp \Psi_3[\bfp,\Vptheta](\bfh)
      + \di_\theta \Psi_3[\bfp,\Vptheta](\phi) \,,
    \end{equation*}
    where
    \begin{align*}
      \di_\bfp \Psi_3[\bfp,\Vptheta](\bfh)
      &\,=\, \frac{1}{L} \int_0^L \biggl(
        2(\npada'\cdot\bpada)
        \bigl( -(\bfh'\cdot\npada)(\tp'\cdot\bpada)
        + (\bfh'\cdot\bpada)(\tp'\cdot\npada) \bigr) |\bfp'| \\
      &\phantom{\,=\, \frac{1}{L} \int_0^L \biggl(}
        + (\npada'\cdot\bpada)^2 \, \frac{\bfp'\cdot\bfh'}{|\bfp'|}
        \biggr) \dt \,,\\
      \di_\theta \Psi_3[\bfp,\Vptheta](\phi)
      &\,=\, \frac{1}{L} \int_0^L 
        2(\npada'\cdot\bpada) \phi' |\bfp'| \dt \,.
    \end{align*}
\end{lemma}

\begin{proof}
  This follows by direct calculations using the definitions
  \eqref{eq:DefPsi1}--\eqref{eq:DefPsi3} together with
  \eqref{eq:DefCurvature}, \eqref{eq:TwistRate}, and
  Lemma~\ref{lmm:FrechetDerivativeM}. 
\end{proof}

We apply the BFGS-scheme from \cite{LiFuk01} to approximate a local
solution of \eqref{eq:UnconstrainedOptimizationProblem}. 
We start with an initial approximation for the spine curve of the
nanowire and for the geometry adapted frame. 
Let $\bfp^{(0)}\in(\Stri)^3$ be a three-dimensional cubic not-a-knot spline
describing the initial guess for the spine curve.  
Accordingly, we compute a rotation minimizing frame
$(\bft_{\bfp^{(0)}},\bfn_{\bfp^{(0)}},\bfb_{\bfp^{(0)}})$ along this
spline using the double reflection method from~\cite{WanJutZheLiu08}.
Then we choose a one-dimensional cubic not-a-knot spline
$\theta^{(0)}\in\Stri$ that describes the rotation function of the
initial guess for the geometry adapted
frame~$(\bft_{\bfp^{(0)}},\bfn_{\bfp^{(0)},\theta^{(0)}},\bfb_{\bfp^{(0)},\theta^{(0)}})$
as in~\eqref{eq:AdaptedFrame}.
As before, we write $V_{\bfp^{(0)},\theta^{(0)}}
:= [\bft_{\bfp^{(0)}} | \bfn_{\bfp^{(0)},\theta^{(0)}}
| \bfb_{\bfp^{(0)},\theta^{(0)}}]$. 
We store the coordinates of the knots of $\bfp^{(0)}$
and~$\theta^{(0)}$ in a vector $\xvec_0\in\R^{4n}$, where the first
$3n$ components are associated to $\bfp^{(0)}=:\bfp\{\xvec_0\}$ and
the last $n$ components correspond
to~$\theta^{(0)}=:\theta\{\xvec_0\}$. 
The vector $\xvec_0$ is the initial guess for the BFGS-scheme. 

Let $\xvec_\ell\in\R^{4n}$ denote the $\ell$-th iterate of the
BFGS-scheme. 
The $\ell$-th spine
curve~${\bfp^{(\ell)}=\bfp\{\xvec_\ell\}}\in(\Stri)^3$ is the
three-dimensional cubic not-a-knot spline determined by the knots
stored in the first $3n$ components of~$\xvec_\ell$. 
Denoting by $\theta^{(\ell)}=\theta\{\xvec_\ell\}\in\Stri$ the
one-dimensional spline described by the knots stored in the last~$n$
components of $\xvec_\ell$, the $\ell$-th geometry adapted frame 
$(\bft_{\bfp^{(\ell)}},\bfn_{\bfp^{(\ell)},\theta^{(\ell)}},\bfb_{\bfp^{(\ell)},\theta^{(\ell)}})$ 
is obtained from the $(\ell-1)$-th geometry adapted frame
$(\bft_{\bfp^{(\ell-1)}},\bfn_{\bfp^{(\ell-1)},\theta^{(\ell-1)}},\bfb_{\bfp^{(\ell-1)},\theta^{(\ell-1)}})$
using the formulas \eqref{eq:FrameTranslated1}--\eqref{eq:FrameTranslated2}
with~${\bfh=\bfp^{(\ell)}-\bfp^{(\ell-1)}}$ and
$\phi=\theta^{(\ell)}-\theta^{(\ell-1)}$.
We write $V_{\bfp^{(\ell)},\theta^{(\ell)}} = \Vptheta\{\xvec_\ell\}
:= [\bft_{\bfp^{(\ell)}} | \bfn_{\bfp^{(\ell)},\theta^{(\ell)}} | \bfb_{\bfp^{(\ell)},\theta^{(\ell)}}]$. 

Given $\xvec_\ell$, the $(\ell+1)$-th iterate of the BFGS-scheme is
defined by 
\begin{equation}
  \label{eq:updateIt}
  \xvec_{\ell+1}
  \,=\, \xvec_\ell + \lambda_\ell \dvec_\ell \,,
\end{equation}
where $\dvec_\ell$ is a solution to the linear system
\begin{equation*}
  H_\ell \dvec_\ell
  \,=\, - \nabla \Phi\bigl[\bfp\{\xvec_\ell\}, 
  \Vptheta\{\xvec_\ell\} \bigr] \,,
\end{equation*}
and $\lambda_\ell\in(0,1)$ determines the stepsize.
The gradient $\nabla\Phi[\bfp\{\xvec_\ell\},\Vptheta\{\xvec_\ell\}]$ of
the regularized objective functional $\Phi$ from \eqref{eq:DefPhi}
with respect to $\xvec_\ell$ is obtained by evaluating the Fr\'echet
derivative of $\bfTrho$ from Theorem~\ref{thm:FrechetDerivativeTrho}
and the Fr\'echet derivatives of the penalty terms $\Psi_1$, $\Psi_2$,
and $\Psi_3$ from \eqref{lmm:DerivativePsi1Psi2} in
$(\bfp\{\xvec_\ell\},\Vptheta\{\xvec_\ell\})$ in the directions
corresponding to the components of $\xvec_\ell$. 
The matrix $H_\ell$ is an approximation to the Hessian matrix
$\nabla^2 \Phi[\bfp\{\xvec_\ell\},\Vptheta\{\xvec_\ell\}]$ with
respect to $\xvec_\ell$.
Starting with the initial guess $H_0 = \I_{4n}$, we use the cautious
update rule  
\begin{equation}
  \label{eq:HUpdate}
  H_{\ell+1} \,=\,
  \begin{cases}
    H_\ell 
    - \frac{H_{\ell}\svec_\ell \svec_\ell^\top H_\ell}
    {\svec_\ell^\top H_\ell \svec_\ell} 
    + \frac{\yvec_\ell\yvec_\ell^\top}{\yvec_\ell^\top \svec_\ell} \qquad
    &\text{if } \frac{\yvec_\ell^\top \svec_\ell}{|\svec_\ell|^2} 
    \,>\, \eps 
    \bigl|\nabla \Phi\bigl[\bfp\{\xvec_\ell\}, 
    \Vptheta\{\xvec_\ell\} \bigr]\bigr| \,,\\
    H_\ell &\text{otherwise} \,,
  \end{cases}
\end{equation}
from \cite{LiFuk01}.
Here, 
\begin{equation*}
  \svec_\ell 
  \,:=\, \xvec_{\ell+1} - \xvec_{\ell} \,,
  \qquad
  \yvec_\ell
  \,:=\, \nabla \Phi\bigl[\bfp\{\xvec_{\ell+1}\}, 
  \Vptheta\{\xvec_{\ell+1}\} \bigr] 
  - \nabla \Phi\bigl[\bfp\{\xvec_\ell\}, 
  \Vptheta\{\xvec_\ell\} \bigr] \,,
\end{equation*}
and $\eps>0$ is a parameter. 
It has been shown in \cite{LiFuk01} that this update rule ensures
positive definiteness of~$H_\ell$ throughout the BFGS-iteration. 

We use an inexact  Armijo-type line search to determine the stepsize
$\lambda_\ell$ in \eqref{eq:updateIt}.
Choosing parameters~${\sigma\in (0,1)}$ and $\delta\in (0,1)$, we
identify the smallest integer $j=0,1,\dots$ such that
$\delta^j$ satisfies 
\begin{equation}
  \label{eq:Backtracking}
  \Phi\bigl(\bfp\{\xvec_\ell+\delta^j\dvec_\ell\}, 
  \Vptheta\{\xvec_\ell+\delta^j\dvec_\ell\} \bigr)
  \,\leq\, \Phi \bigl(\bfp\{\xvec_\ell\}, 
  \Vptheta\{\xvec_\ell\} \bigr)
  + \sigma \delta^j \nabla \Phi\bigl[\bfp\{\xvec_\ell\}, 
  \Vptheta\{\xvec_\ell\} \bigr]
  \cdot \dvec_\ell \,.
\end{equation}
Then, we set $\lambda_\ell := \delta^j$.

In our numerical examples below, we use the parameters
${\eps=10^{-5}}$, ${\sigma=10^{-4}}$, and ${\delta=0.9}$ 
in~\eqref{eq:HUpdate} and~\eqref{eq:Backtracking}.
We approximate all line integrals over $\Gamma$ using a composite
Simpson rule. 
We stop the BFGS iteration when
${|\xvec_{\ell+1}-\xvec_\ell|/|\xvec_\ell|<10^{-4}}$.

\section{Numerical examples}
\label{sec:NumericalResults}
We discuss three numerical examples, where we use the shape
optimization scheme developed in the previous section to design highly
chiral thin silver and gold nanowires at four different frequencies in
the optical band. 
We work at
\begin{itemize}
\item $\fopt=400$~THz, i.e., the wave length is $\lambdaopt=749$~nm (red 
  light), 
\item $\fopt=500$~THz, i.e., the wave length is $\lambdaopt=600$~nm (orange
  light), 
\item $\fopt=600$~THz, i.e., the wave length is $\lambdaopt=500$~nm (green
  light), 
\item $\fopt=700$~THz, i.e., the wave length is $\lambdaopt=428$~nm (blue
  light). 
\end{itemize}
The relative electric permittivities $\epsr$ of silver and gold
corresponding to these frequencies can be found
in~Table~\ref{tab:MaterialParameters} (see \cite[p.~6]{JoChri72} for the
complete data set). 

We focus on elliptical cross-sections $\Drho'=\rho B'$, $\rho>0$,
where the lengths of the semi axes of the rescaled cross-section $B'$
are denoted by~$0<a\leq b<1$. 
As discussed in Remark~\ref{rem:spr}, a frequency~$\fres$ is called a
plasmonic resonance frequency of such a thin metallic nanowire, if the
aspect ratio $b/a$ of its elliptical cross-section satisfies
$b/a=-\real(\eps_r(\fres))$, and if~$\imag(\eps_r(\fres))>0$ is
sufficiently small. 
The total interaction cross-section of the nanowire (i.e., the
Hilbert-Schmidt norm of the associated far field operator) at a
plasmonic resonance frequency is much larger than away from this
frequency.
Accordingly, thin metallic nanowires are strongly scattering at
plasmonic resonance frequencies. 
Strongly scattering highly em-chiral nanowires would be very
interesting for the design of novel chiral metamaterials (see,
e.g.,~\cite{HenSchDuaGie17,HofEtAl19,ValEtAl13}). 
Thus, we choose in our first two examples the aspect ratios of the
elliptical cross-sections of the nanowires such that the
frequency~$\fopt$, where the shape optimization is carried out, is a
plasmonic resonance frequency, i.e., $\fopt=\fres$.
We show that strongly scattering thin metallic nanowires with
fairly large em-chirality measures can be obtained. 
In our third example we then design thin metallic nanowires with
even larger em-chirality measures, choosing the frequency~$\fopt$
to be around $100$ to $150$~THz below the plasmonic resonance
frequency~$\fres$ of the nanowire, i.e., $\fopt\not=\fres$. 
However, in this case the total interaction cross-section of the
optimized nanowire is smaller than in the previous examples. 

As already pointed out in Remark~\ref{rem:rho}, the scaling parameter
$\rho>0$ that determines the thickness of the nanowire $\Drho$
does not affect the outcome of the shape optimization. 
Accordingly, the results that we present in this section are valid for 
any $\rho>0$ that is small enough such that the leading order
term~$\TcalDrho$ in~\eqref{eq:FFExp} constitutes an acceptable 
approximation of the far field operator $\FcalDrho$. 
In \cite{CapGriKno21} we compared~$\TcalDrho$ for circular
cross-sections with radius $\rho$, a real-valued electric
permittivity $\eps_r>0$, and a whole range of values for $\rho$ with
numerical approximations of $\FcalDrho$ that have been computed using
the C++ boundary element library Bempp~\cite{SmiBetArrPhi15}.  
This study suggests that $\TcalDrho$ is an accurate approximation of
$\FcalDrho$ within a relative error of less than $5\%$ when the radius
of the thin tube $\Drho$ is less than $1.5\%$ of the  wave length of
the incident field, i.e., when~$k\rho\lesssim 0.1$.
For instance, choosing $\rho=0.1/\kopt$ means that the radius $\rho$
of the nanowire with circular cross-section is between~$6.8$~nm at
$\fopt=700$~THz and $11.9$~nm at $\fopt=400$~THz. 
Here, $\kopt$ denotes the wave number corresponding to the
frequency~$\fopt$.
In our visualizations of the optimized nanowires with elliptical
cross-sections, and for the plots of the total interaction
cross-section of these optimized nanowires in the examples below, we 
choose $\rho$ such that~$\kopt\rho \sqrt{a b} = 0.05$.

\begin{example}[Optimizing the twist rate of the cross-section along a
  straight nanowire] 
  \label{exa:1}
  In our first numerical example we discuss thin straight silver and
  gold nanowires with elliptical cross-sections.
  We consider four different frequencies $\fopt=400,500,600$, and
  $700$~THz, and for each of these frequencies we choose a different
  aspect ratio for the elliptical cross-section of the nanowire such
  that $\fopt=\fres$ is a plasmonic resonance frequency of the
  nanowire, i.e.,~$b/a=-\real(\epsr(\fopt))$. 
  We fix the spine curve of the nanowire to be a straight line segment,
  and we optimize just the twist rate of the elliptical cross-section 
  along the spine curve of the nanowire, i.e., the twist function
  $\theta$ in \eqref{eq:DefRtheta}--\eqref{eq:DefDrho}.
  We use the shape optimization scheme from
  Section~\ref{sec:ShapeOptimization}. 
  For the regularization parameters in \eqref{eq:DefPhi} we choose
  $\alpha_1 = \alpha_2 = 0$ and~$\alpha_3 = 5\times 10^{-5}$.

  To discuss the influence of the length of the nanowire on the
  optimized shape of the nanowire, we consider four different values
  $L=j\lambdaopt/4$ with $j=1,2,4,8$ for the length constraint in
  \eqref{eq:DefPsi1}.
  As before, $\lambdaopt$ denotes the wave length at the frequency
  $\fopt$. 
  Accordingly, we choose the maximal degree $N$ of circularly polarized
  vector spherical harmonics that is used in the discretization of the
  operator $\bfTrho(\bfp,\Vptheta)$ and of its Fr\'echet derivative
  $\bfTrho'[\bfp,\Vptheta](\bfh,\phi)$ (see
  Remark~\ref{rem:DiscretisationT}) to be~$N=2,4,6,8$ for 
  $L=j\lambdaopt/4$ with $j=1,2,4,8$, respectively. 

  We use cubic not-a-knot splines with $n = 10$ knots to describe the
  (fixed) spine curve $\Gamma$ and the twist function $\theta$ and $11$
  quadrature points for the composite Simpson rule on each spline
  segment to approximate integrals over $\Gamma$.
  Since the em-chirality measure $\chiHS$, and thus also the objective
  functional $\Phi$, are not differentiable at an em-achiral configuration,
  we choose an em-chiral initial guess for the shape optimization
  algorithm. 
  To this end, we start with a rotation minimizing frame along the
  straight spine curve and add a small random twist.
  The same random twist is used for all frequencies and length
  constraints. 

  \begin{figure}[thp]
    \centering
    \includegraphics[height=6cm]{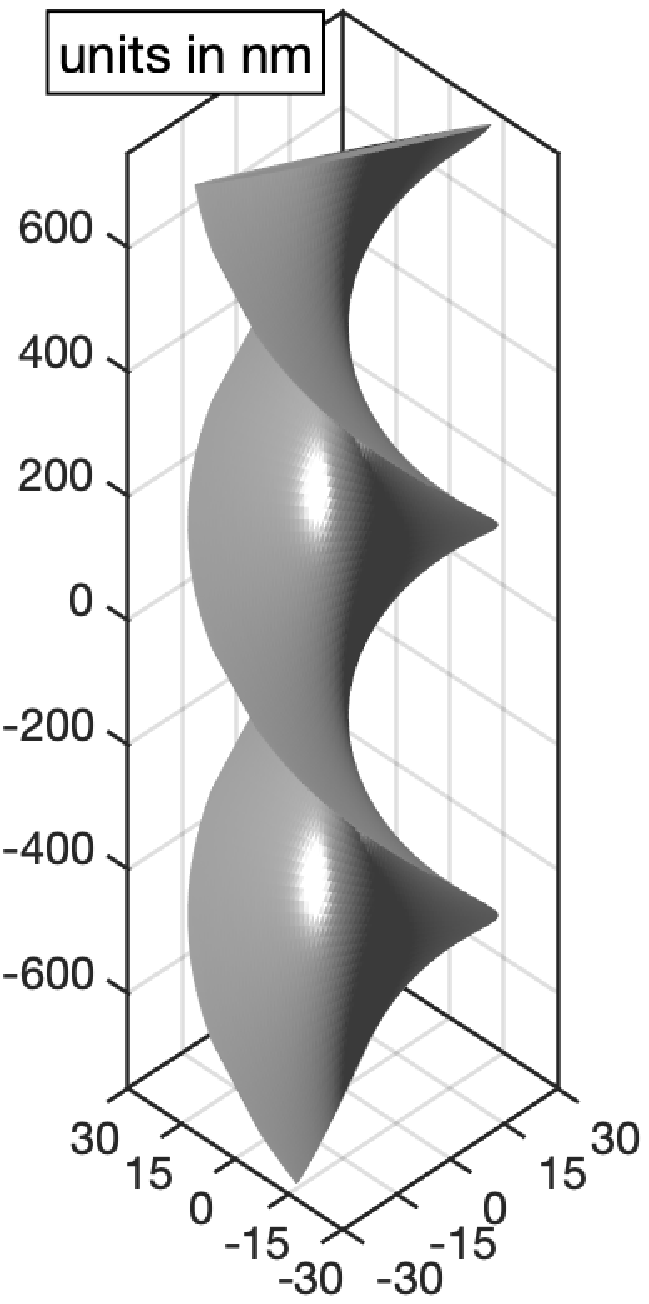} \qquad
    \includegraphics[height=6cm]{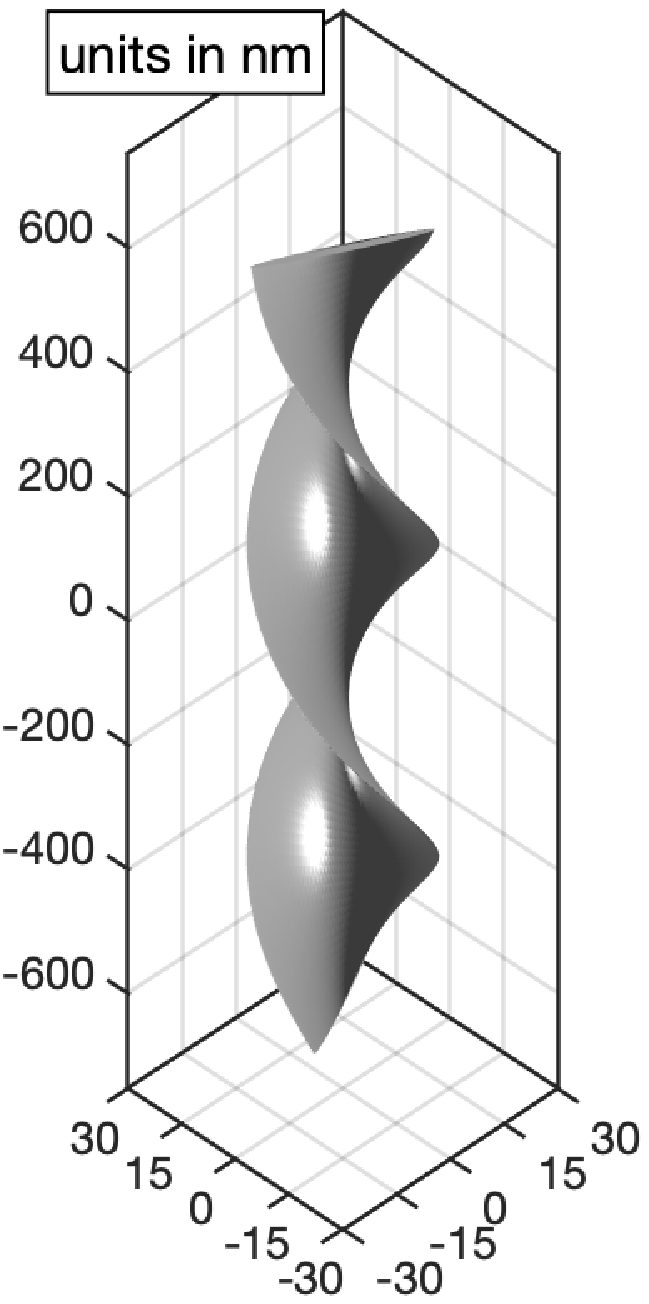} \qquad
    \includegraphics[height=6cm]{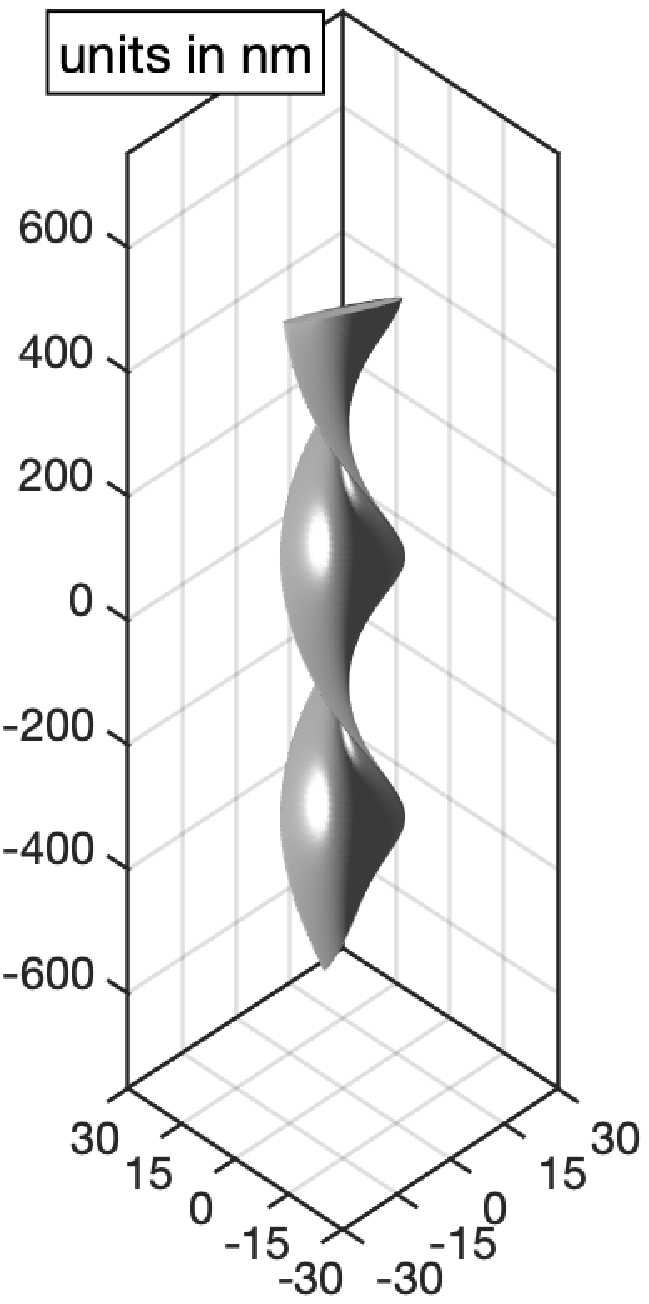} \qquad
    \includegraphics[height=6cm]{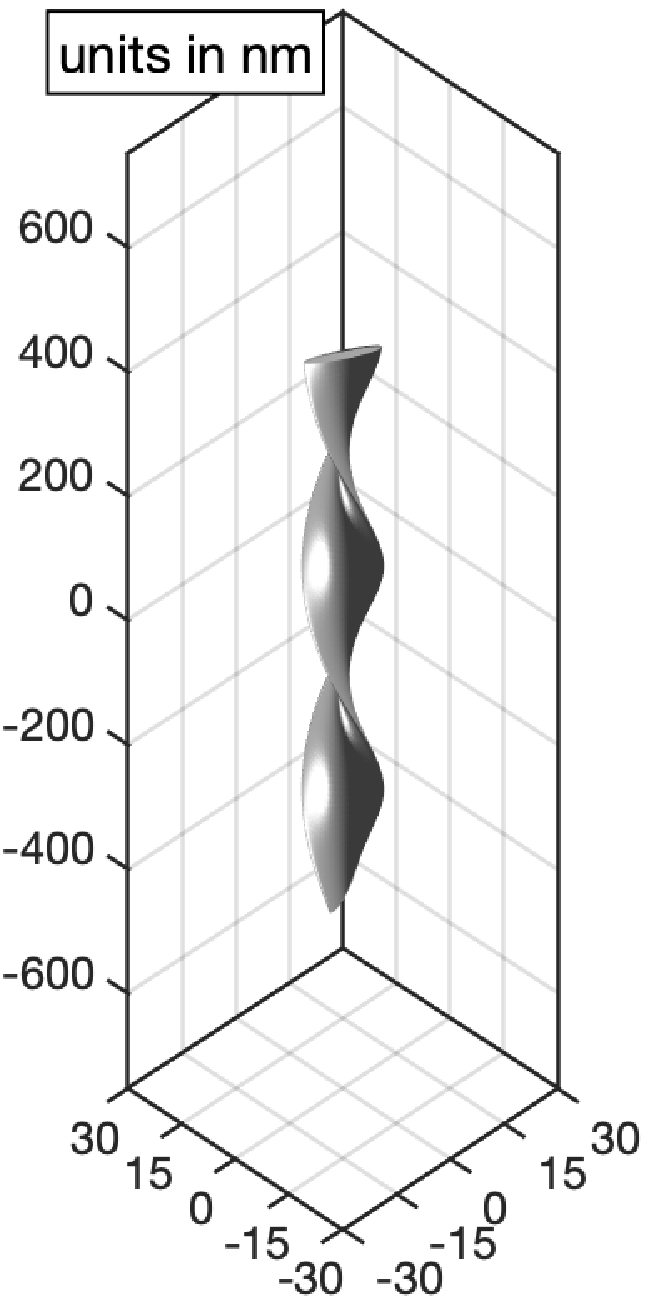}
    \caption{Optimized silver nanowires of length
      $L=2\lambdaopt$ from Example~\ref{exa:1} for
      $\fopt=400,500,600,700$~THz (left to right).
    }
    \label{fig:Exa1-1}
  \end{figure}
  In Figure~\ref{fig:Exa1-1} we show the optimized twisted silver
  nanowires obtained by the shape optimization for $L=2\lambdaopt$ and 
  $\fopt=400,500,600$, and~$700$~THz (left to right). 
  The direction of the twist of the optimized structure depends
  on the initial guess. 
  The aspect ratios $b/a$ of the elliptical cross-sections vary
  between $26.94$ at $\fopt=400$~THz and $5.94$ at $\fopt=700$~THz. 
  The optimized twist rate per wave length of the cross-sections of
  the four optimized twisted silver nanowires around the straight
  spine curve is almost constant and virtually the same for all
  frequencies. 

  \begin{table}
    \begin{subfigure}{0.5\textwidth}%
      \centering
      \begin{tabular}{llllll} 
        \hline
        \multicolumn{6}{c}{\textbf{Silver}} \\ \hline 
        \multicolumn{2}{r}{$\fopt$~[THz]} & $400$ & $500$ & $600$& $700$ \\ 
        \hline \hline 
        {\multirow{2}{*}{$L=\frac{\lambdaopt}{4}$}} & $\Jtwo$&0.26&0.26&0.26&0.26\\ 
        & ﻿$\JHS$ &0.12&0.12&0.12&0.12\\ \hline 
        {\multirow{2}{*}{$L=\frac{\lambdaopt}{2}$}} & $\Jtwo$&0.39&0.39&0.39&0.39\\ 
        & ﻿$\JHS$ &0.17&0.17&0.17&0.17\\ \hline 
        {\multirow{2}{*}{$L=\lambdaopt$}} & $\Jtwo$&0.37&0.37&0.37&0.37\\ 
        & ﻿$\JHS$ &0.20&0.20&0.20&0.20\\ \hline 
        {\multirow{2}{*}{$L=2\lambdaopt$}} & $\Jtwo$&0.32&0.32&0.32&0.32\\ 
        & ﻿$\JHS$ &0.19&0.19&0.19&0.19\\ \hline 
      \end{tabular}
    \end{subfigure}
    \begin{subfigure}{0.5\textwidth}%
      \centering
      \begin{tabular}{llllll} 
        \hline
        \multicolumn{6}{c}{\textbf{Gold}} \\ \hline 
        \multicolumn{2}{r}{$\fopt$~[THz]} & $400$ & $500$ & $600$& $700$ \\ 
        \hline \hline 
        {\multirow{2}{*}{$L=\frac{\lambdaopt}{4}$}} & $\Jtwo$&0.26&0.23&0.09&0.03\\ 
        & ﻿$\JHS$ &0.12&0.12&0.02&0.003\\ \hline 
        {\multirow{2}{*}{$L=\frac{\lambdaopt}{2}$}} & $\Jtwo$&0.39&0.37&0.17&0.06\\ 
        & ﻿$\JHS$ &0.17&0.17&0.03&0.004\\ \hline 
        {\multirow{2}{*}{$L=\lambdaopt$}} & $\Jtwo$&0.36&0.35&0.13&0.03\\ 
        & ﻿$\JHS$ &0.20&0.19&0.04&0.003\\ \hline 
        {\multirow{2}{*}{$L=2\lambdaopt$}} & $\Jtwo$&0.32&0.30&0.08&0.01\\ 
        & ﻿$\JHS$ &0.19&0.19&0.03&0.0008\\ \hline 
      \end{tabular}
    \end{subfigure}
    \caption{Normalized em-chirality measures $\Jtwo$ and $J_\HS$ of
      optimized silver (left) and gold nanowires (right) from
      Example~\ref{exa:1}.} 
    \label{tab:Exa1-2}
  \end{table}
  In Table~\ref{tab:Exa1-2} we collect the values of the normalized
  em-chirality measures $\Jtwo$ and $J_\HS$ from \eqref{eq:J2}
  and~\eqref{eq:JHS} of the optimized straight twisted silver and gold 
  nanowires for the four different frequencies and the four different
  length constraints. 
  Each pair of entries in these tables corresponds to a different
  optimized twisted silver or gold nanowire. 
  For the silver nanowires we observe that the values of $\Jtwo$ and
  $J_\HS$ that are reached for the different optimized structures are
  independent of the frequency. 
  On the other hand, for the optimized gold nanowires these values
  change significantly with frequency.  
  While at $\fopt=400$ and $500$~THz the normalized em-chirality
  measures of the optimized twisted gold nanowires are comparable to
  those of the optimized twisted silver nanowires, the normalized
  em-chirality measures of the optimized twisted gold nanowires
  quickly decrease at higher frequencies. 
  This is a consequence of the increasing imaginary part of the
  relative electric permittivity of gold at higher frequencies (see 
  Table~\ref{tab:MaterialParameters}).
  For the gold nanowires the aspect ratio~$b/a$ of the elliptical
  cross-section varies between~$20.11$ at $\fopt=400$~THz and $1.69$
  at $\fopt=700$~THz, i.e., the cross-section is somewhat rounder than
  for the corresponding silver nanowires.

  \begin{figure}[thp]
    \includegraphics[height=3.78cm]{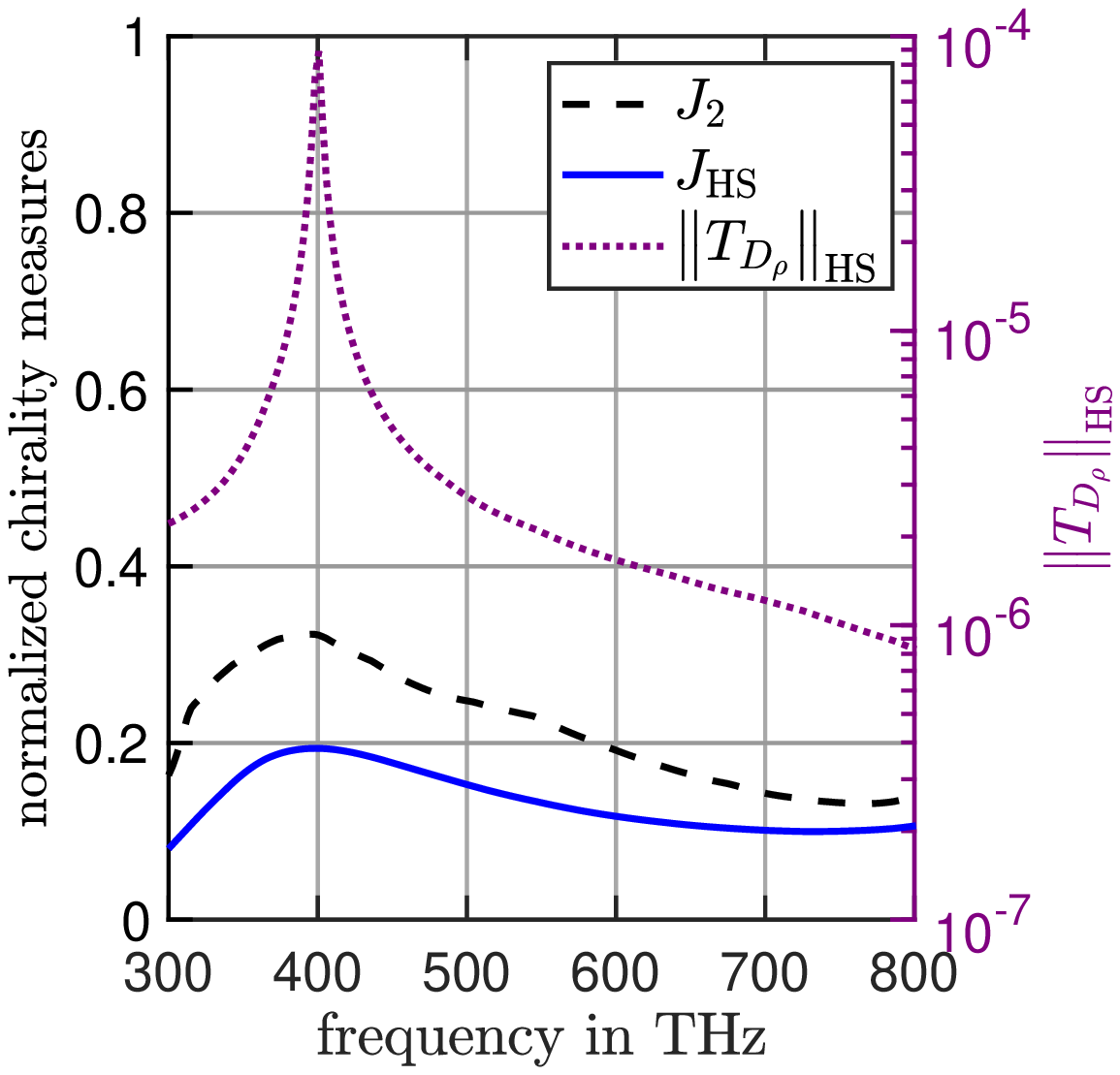} \,
    \includegraphics[height=3.78cm]{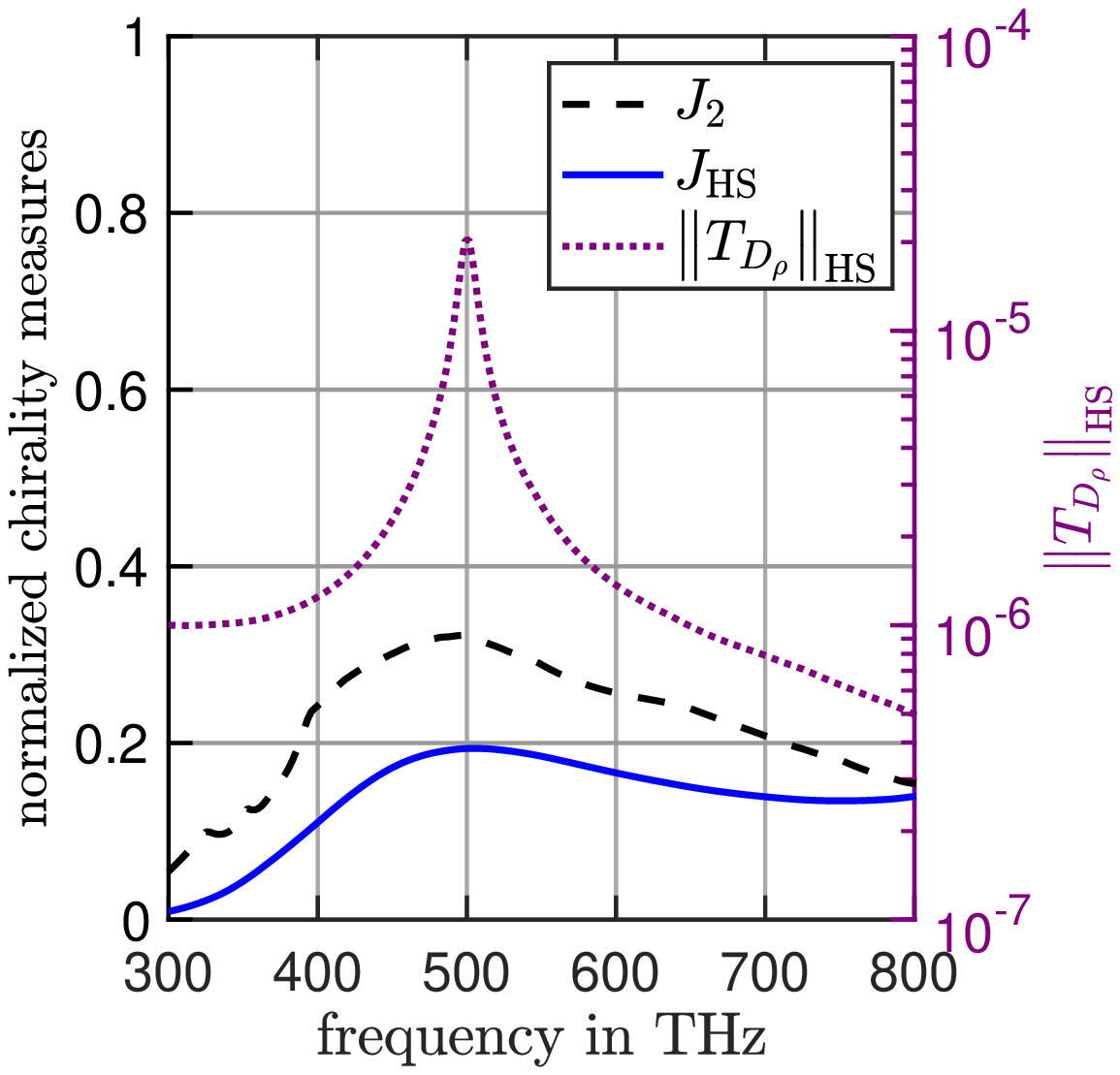} \,
    \includegraphics[height=3.78cm]{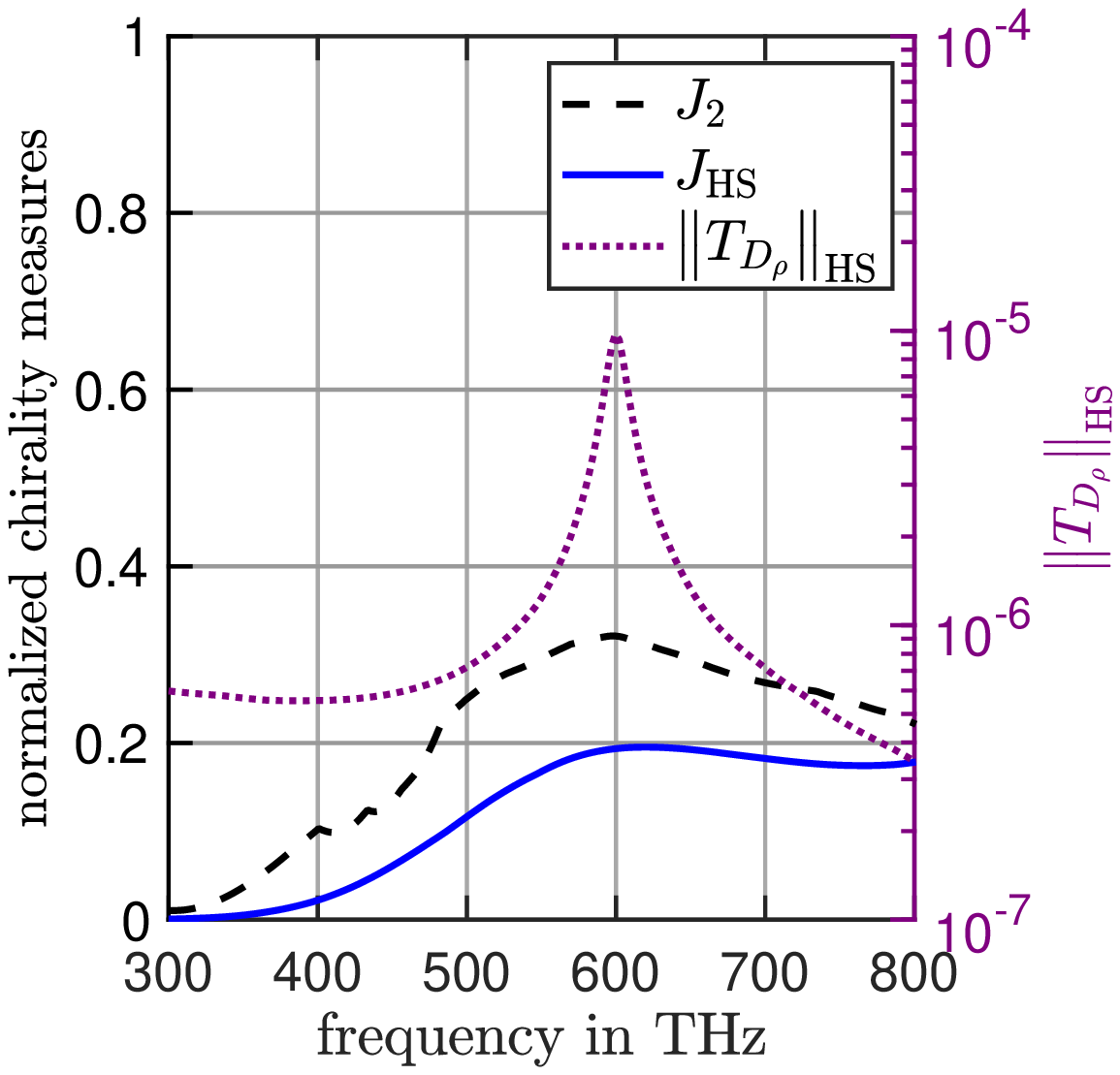} \,
    \includegraphics[height=3.78cm]{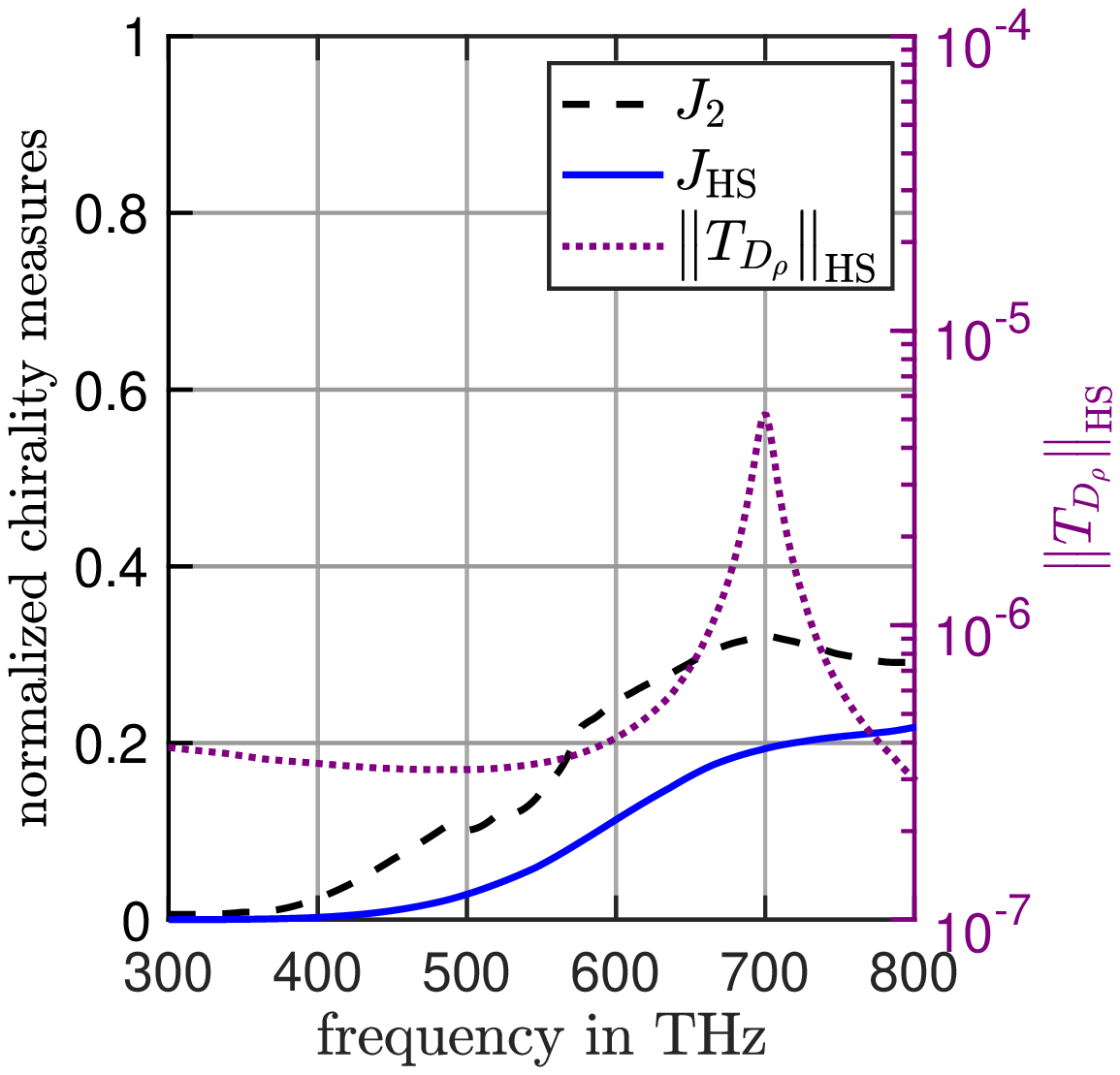}
    \caption{ Frequency scans for optimized silver nanowires of length
      $L=2\lambdaopt$ from Example~\ref{exa:1} at
      $\fopt=400,500,600,700$~THz (left to right).} 
    \label{fig:Exa1-3}
  \end{figure}
  \begin{figure}[thp]
    \includegraphics[height=3.78cm]{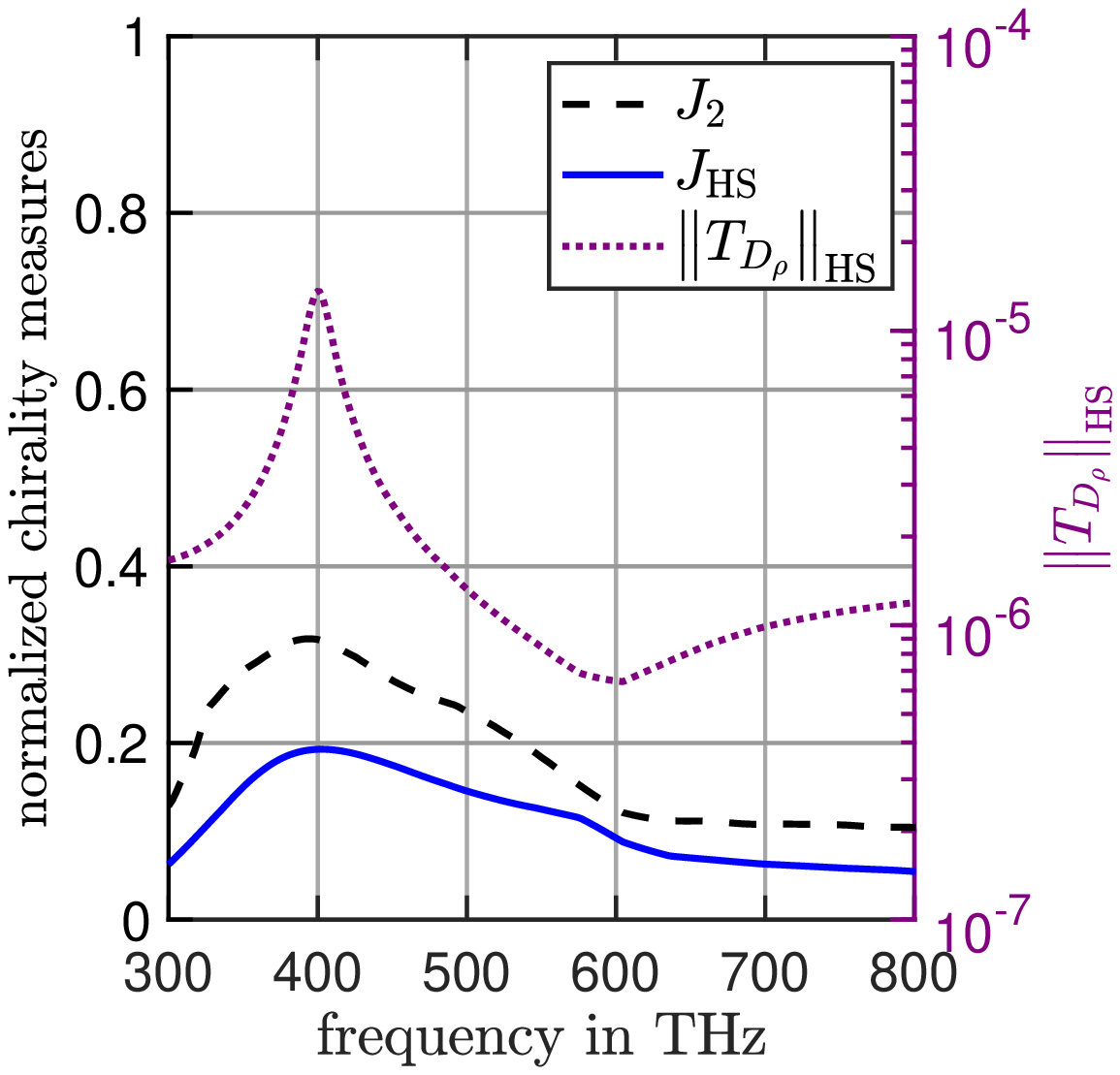} \,
    \includegraphics[height=3.78cm]{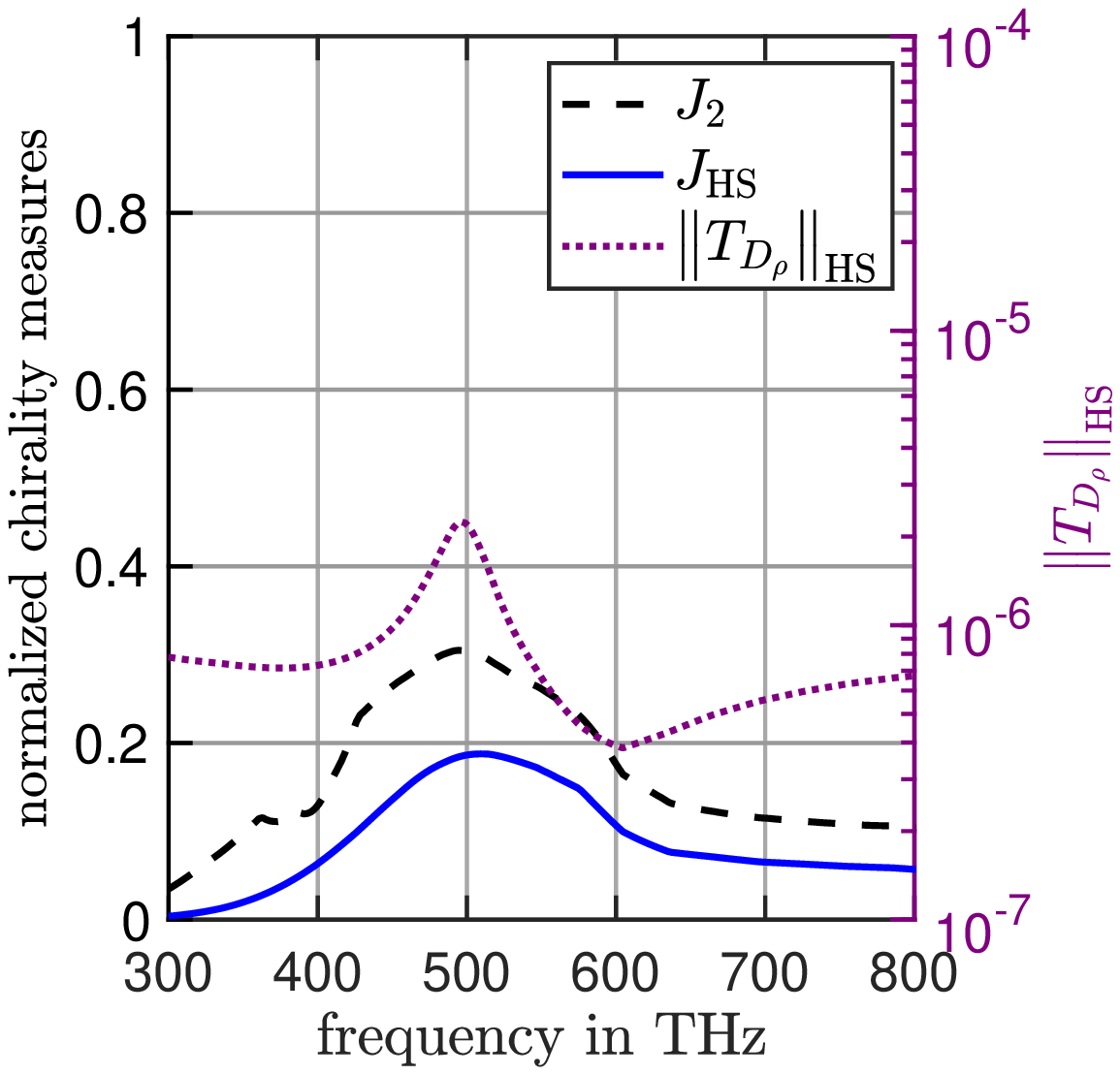} \,
    \includegraphics[height=3.78cm]{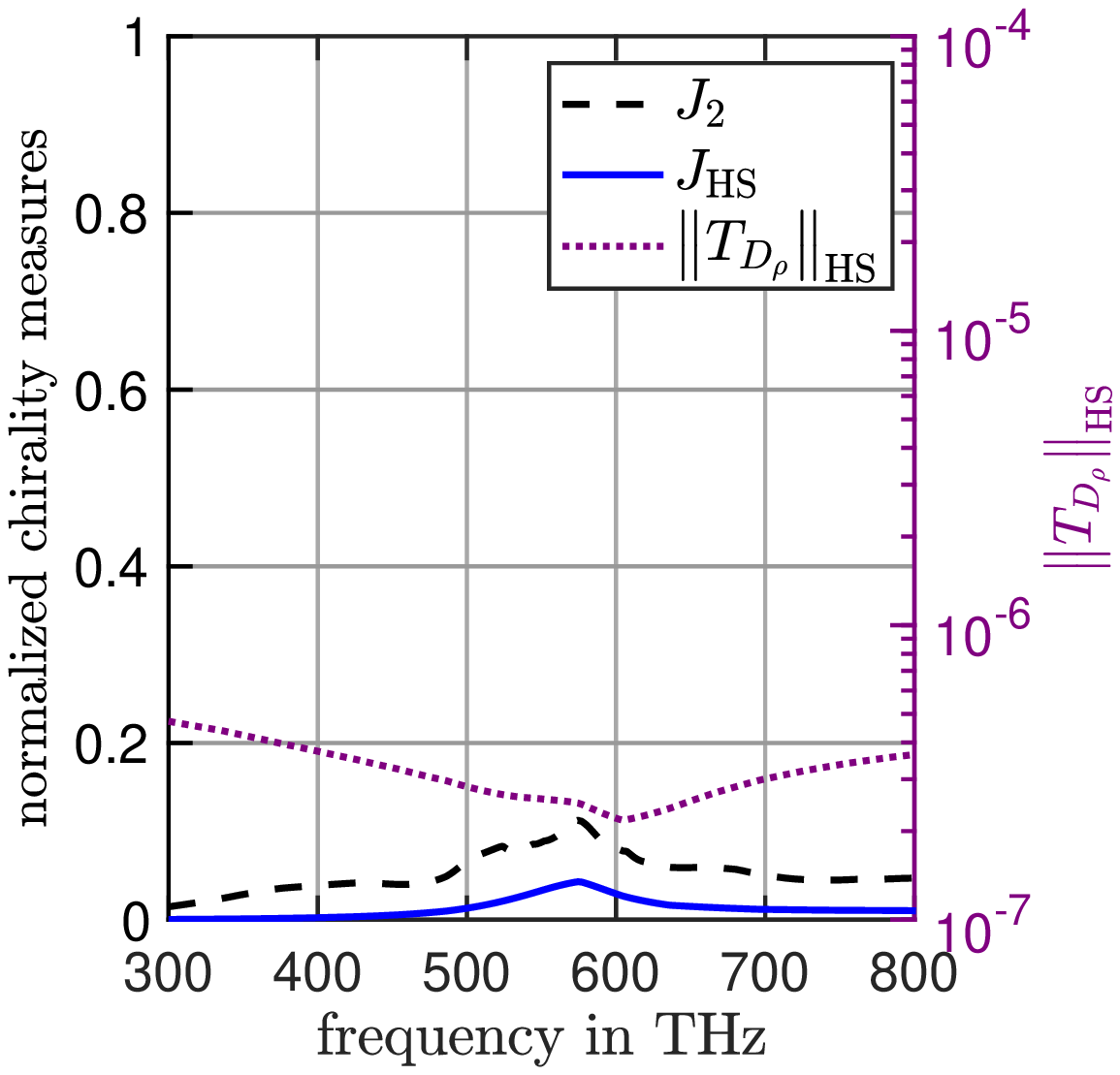} \,
    \includegraphics[height=3.78cm]{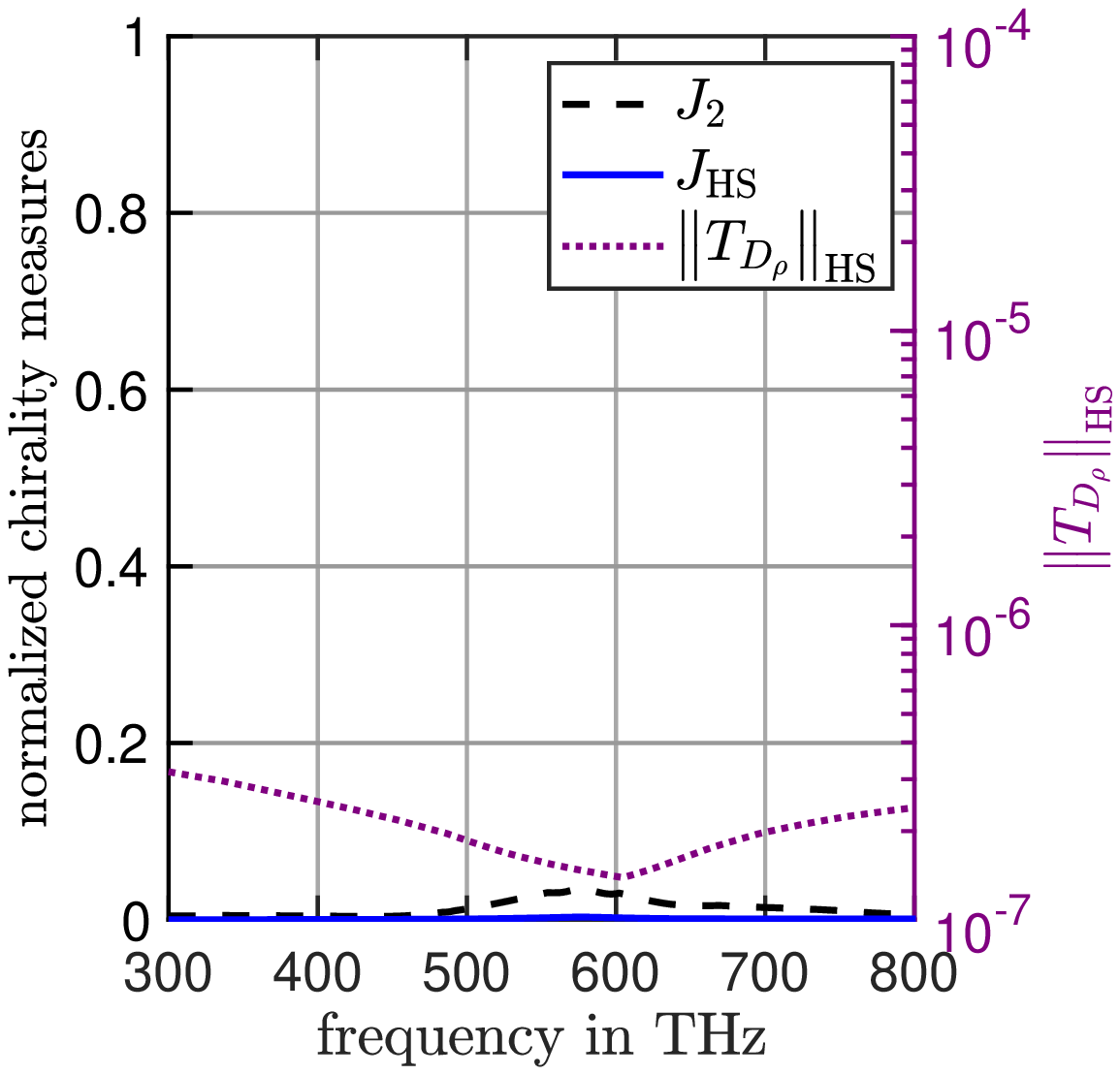}
    \caption{ Frequency scans for optimized gold nanowires of length
      $L=2\lambdaopt$ from Example~\ref{exa:1} at
      $\fopt=400,500,600,700$~THz (left to right).} 
    \label{fig:Exa1-4}
  \end{figure}
  Finally, we study the frequency dependence of the normalized
  em-chirality measures for the optimized twisted silver and gold
  nanowires of length $L=2\lambdaopt$ that have been optimized at
  $\fopt=400,500,600$, and~$700$~THz. 
  In each of the plots in Figure~\ref{fig:Exa1-3} and \ref{fig:Exa1-4}
  the optimized nanowire is fixed.
  However, it is illuminated with incident waves of different
  frequencies, and thus its frequency-dependent relative electric 
  permittivity~$\epsr$ varies.
  In Figure~\ref{fig:Exa1-3} we plot the normalized em-chirality
  measures $\Jtwo$ (dashed) and $\JHS$ (solid) and an approximation of 
  the total interaction cross-section (dotted) of the optimized thin
  twisted silver nanowires shown in Figure~\ref{fig:Exa1-1} over a
  frequency range between $300$ and $800$~THz.
  The approximation of the total interaction cross-section is obtained
  by evaluating the Hilbert-Schmidt norm of the operator~$\TcalDrho$
  from \eqref{eq:DefTcal} with $\rho = 0.05 / (\kopt \sqrt{a b})$. 
  The sharp peak in the total interaction cross-section is exactly at
  the plasmonic resonance frequency $\fres$ of the corresponding thin
  silver nanowire. 
  It is important to note that, in contrast to the normalized
  em-chirality measures, the total interaction cross-section is
  plotted in a logarithmic scale. 
  We find that $\Jtwo$ and $\JHS$ have a peak at the plasmonic
  resonance frequency $\fres$ as well.
  This is the frequency that has been used in the shape optimization,
  i.e., $\fopt=\fres$. 

  In Figure~\ref{fig:Exa1-4} we show the corresponding frequency scans
  for the optimized thin twisted gold nanowires. 
  For the gold nanowires that have been optimized at $\fopt=400$ and
  $500$~THz, the results are similar as for the silver nanowires that
  have been optimized at the same frequencies in
  Figure~\ref{fig:Exa1-3}. 
  On the other hand, for the gold nanowires optimized at $\fopt=600$ and
  $700$~THz, the plasmonic resonance is no longer visible in the plots
  of the total interaction cross-section.
  This is a consequence of the larger imaginary part of the electric
  permittivity of gold at $\fopt=600,700$~THz (see
  Table~\ref{tab:MaterialParameters}). 
  For these two higher frequencies, the values of the normalized
  em-chirality measures $\Jtwo$ and $\JHS$ are small across the entire
  frequency band. 
  The plasmonic resonance seems to be required to obtain thin metallic
  nanowires that exhibit large normalized em-chirality
  measures.~\hfill$\lozenge$ 
\end{example}

\begin{example}[Optimizing the shape of the spine curve of the nanowire] 
  \label{exa:2}
  In our second example we consider a free-form shape optimization for
  the spine curve of thin silver and gold nanowires with elliptical
  cross sections, but we do not optimize the twist rate of the
  cross-section of the nanowire along the spine curve.
  As in the first example, we consider four different frequencies
  $\fopt=400, 500, 600$, and $700$~THz, and for each of these frequencies
  we again choose the aspect ratio of the elliptical cross-section of
  the nanowire such that $\fopt=\fres$ is a plasmonic resonance
  frequency for the nanowire. 
  We use the optimization scheme from
  Section~\ref{sec:ShapeOptimization}. 
  For the regularization parameters in \eqref{eq:DefPhi} we choose
  $\alpha_1 = 5$, $\alpha_2= 8\times 10^{-3}$, and $\alpha_3 = 0$.
  
  The length constraint for the nanowire is set to be
  $L=3\lambdaopt/2$, and, accordingly, we choose the maximal degree
  $N$ of circularly polarized vector spherical harmonics that is used
  in the discretization of the operator~$\bfTrho(\bfp,\Vptheta)$ and
  of its Fr\'echet derivative $\bfTrho'[\bfp,\Vptheta](\bfh,\phi)$ 
  (see Remark~\ref{rem:DiscretisationT}) to be~$N=5$. 
  We use cubic not-a-knot splines with $n=20$ knots to parametrize the
  spine curve $\Gamma$ and the (fixed) twist function~$\theta$
  and~$11$ quadrature points on each spline segment to discretize line 
  integrals over $\Gamma$. 
  For the initial guess for the spine curve we use a straight line of
  length $L=3\lambdaopt/2$, and we add a small random perturbation to
  obtain an em-chiral configuration. 
  The same random perturbation is used for all frequencies.
  The initial geometry adapted frame is chosen to be rotation
  minimizing. 
  
  \begin{figure}[thp]
    \centering
    \includegraphics[height=4.2cm]{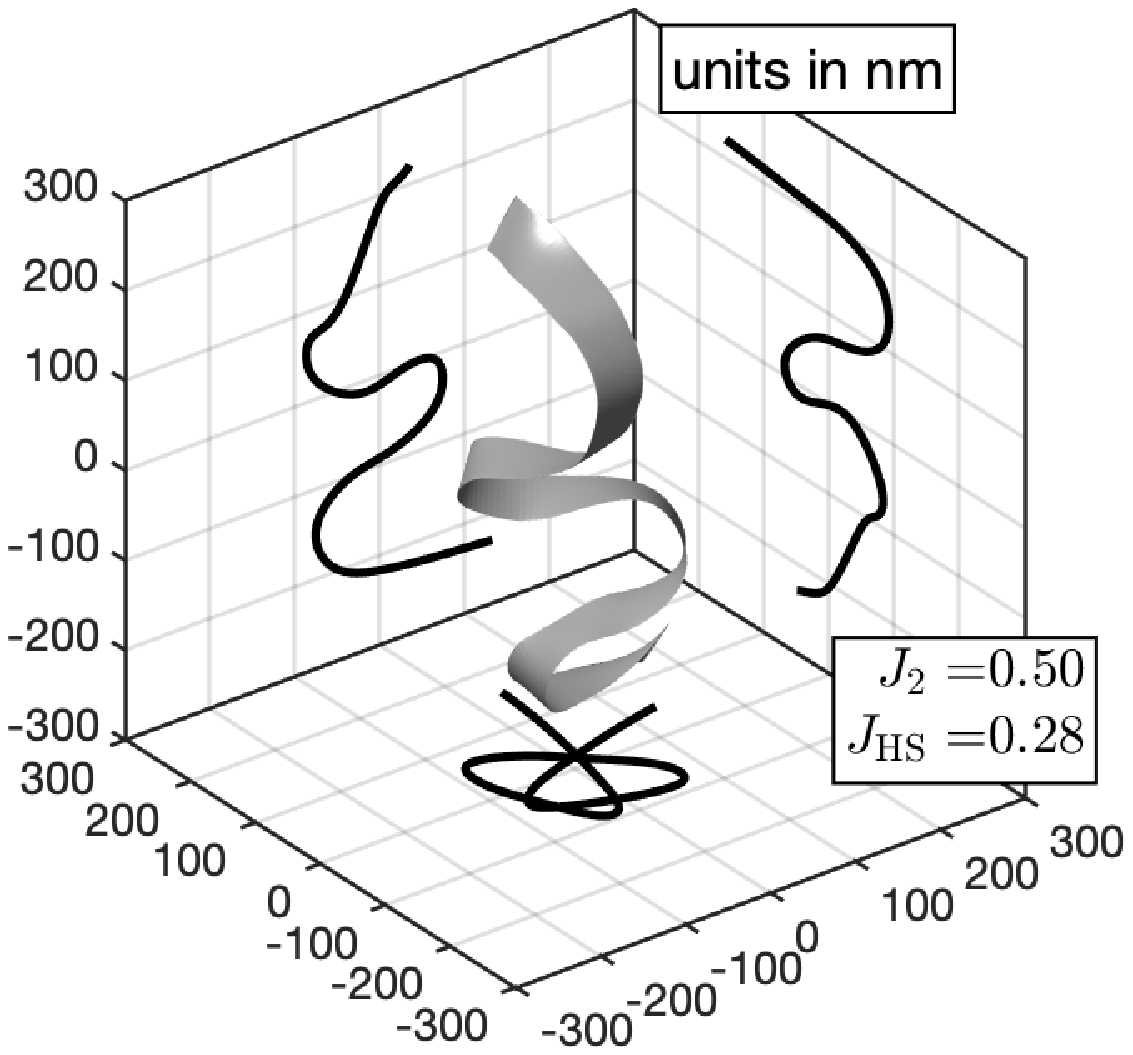}
    \includegraphics[height=3.78cm]{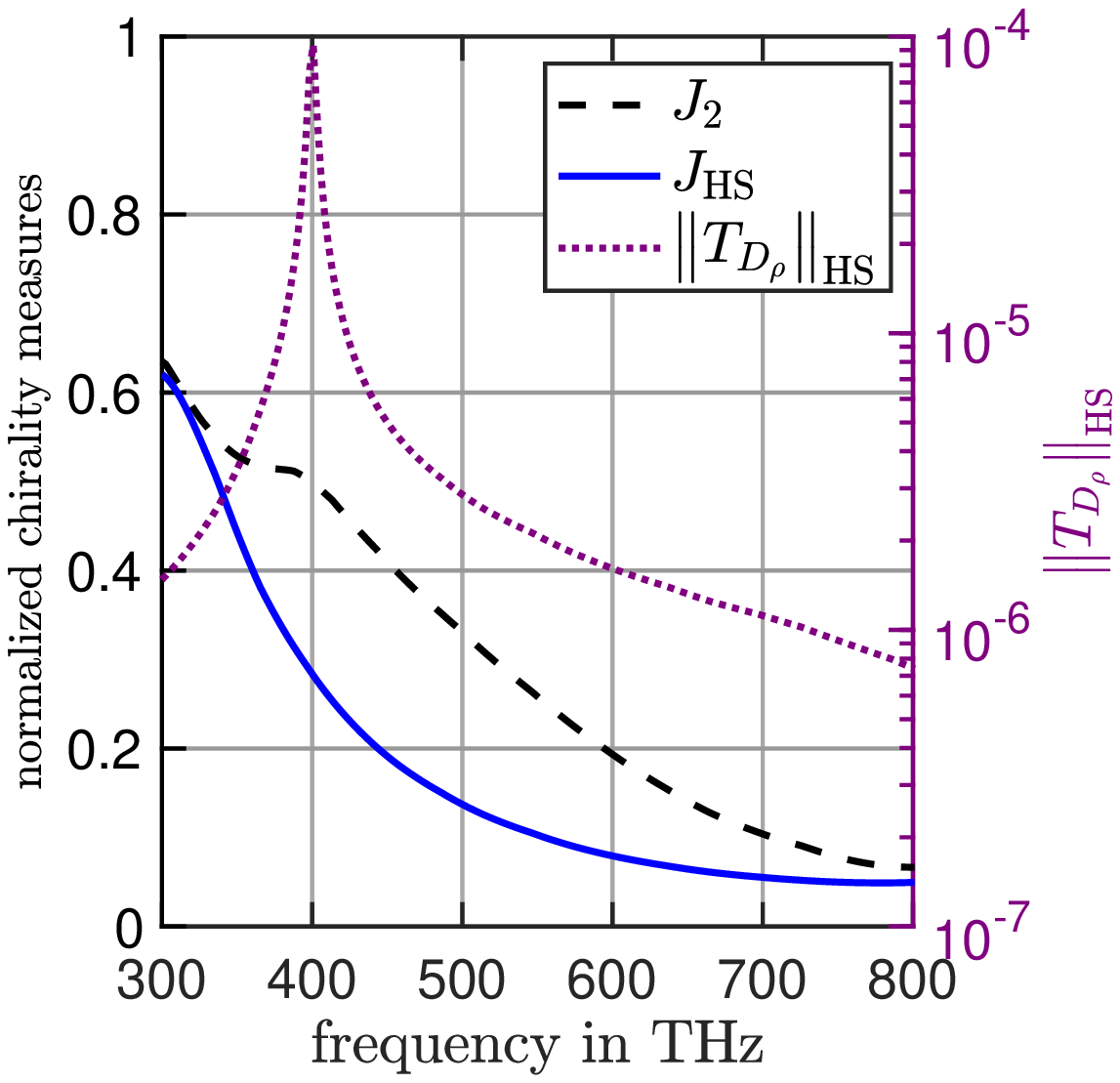} 
    \includegraphics[height=4.2cm]{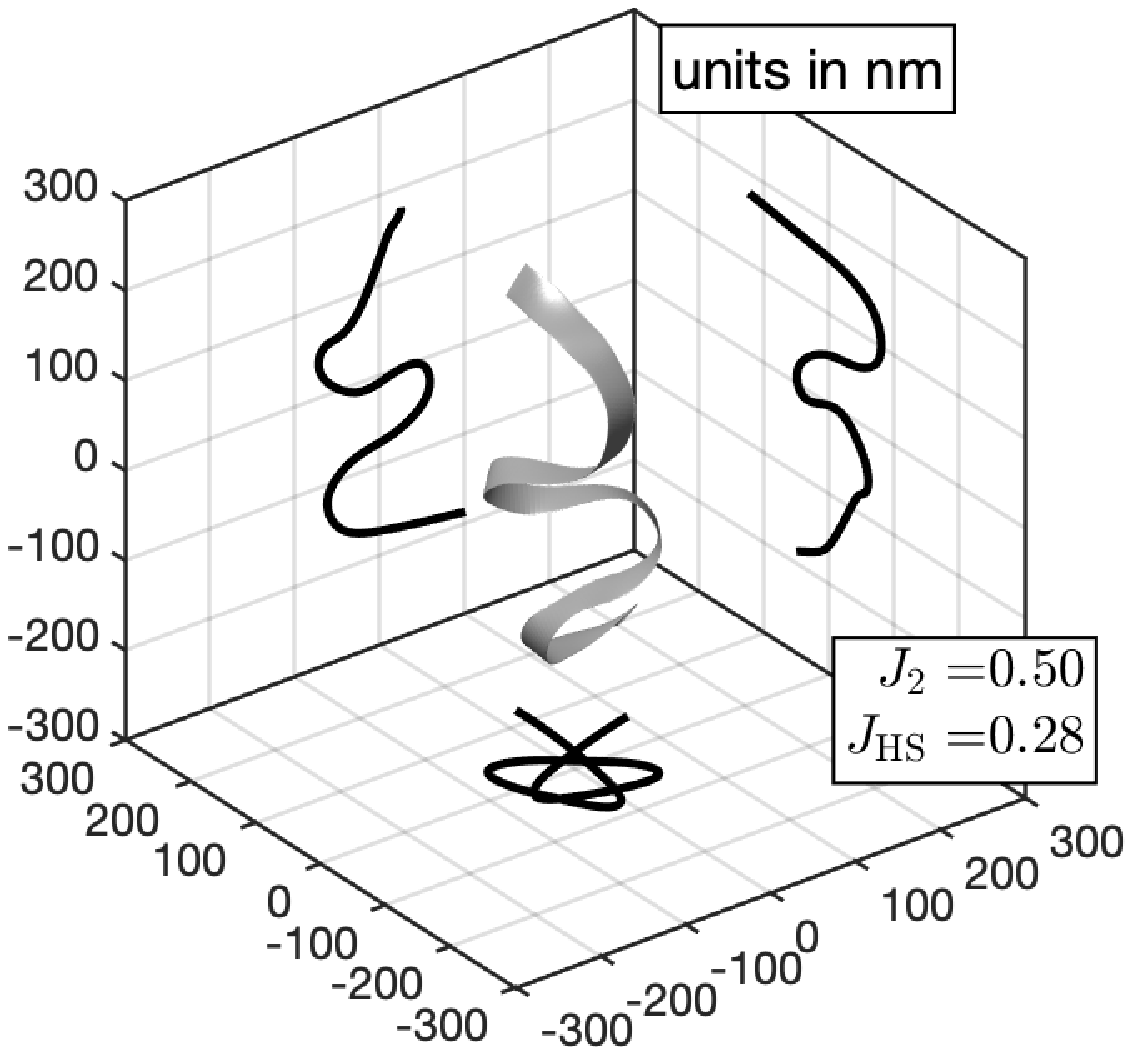}
    \includegraphics[height=3.78cm]{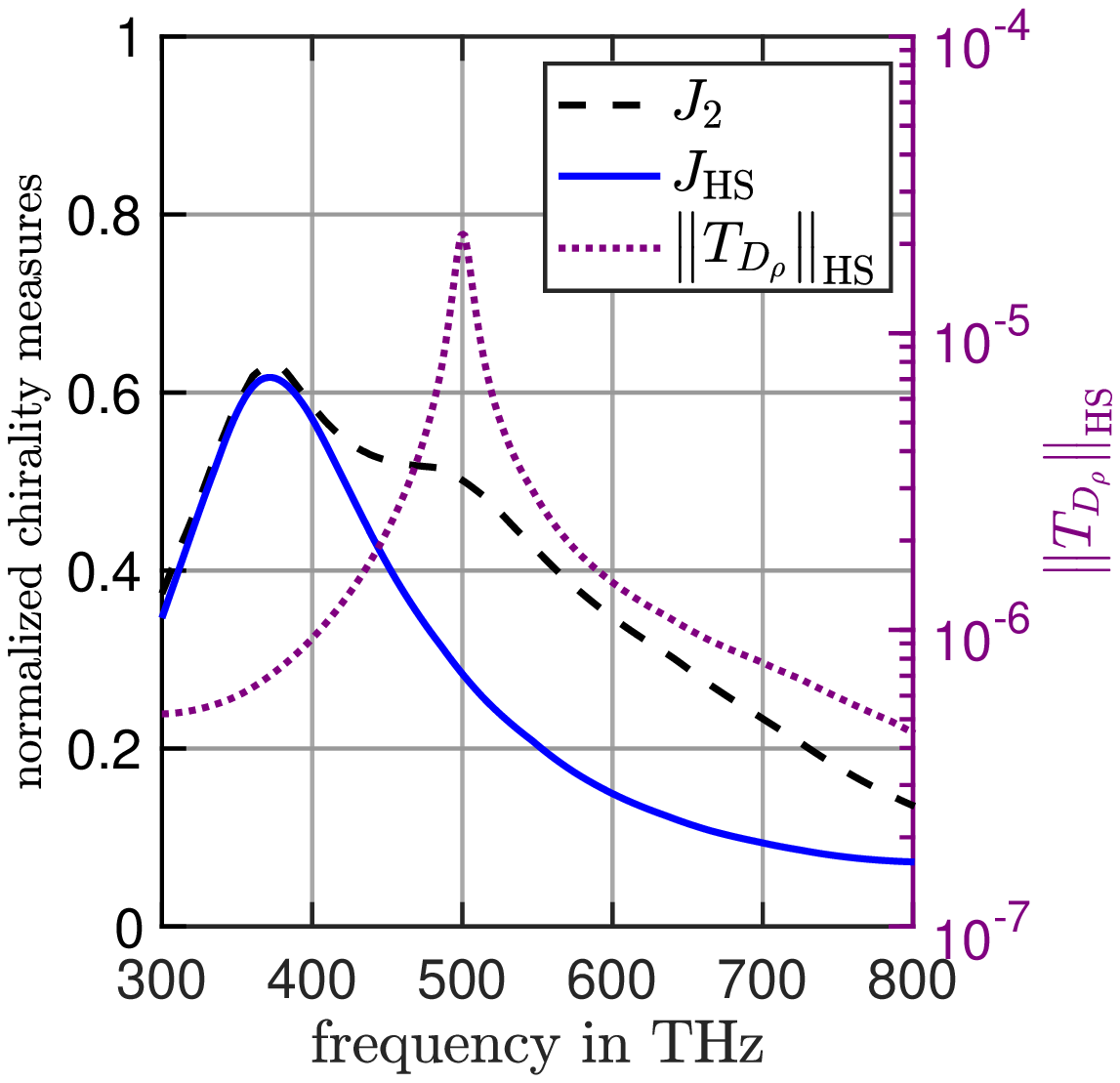}
    \includegraphics[height=4.2cm]{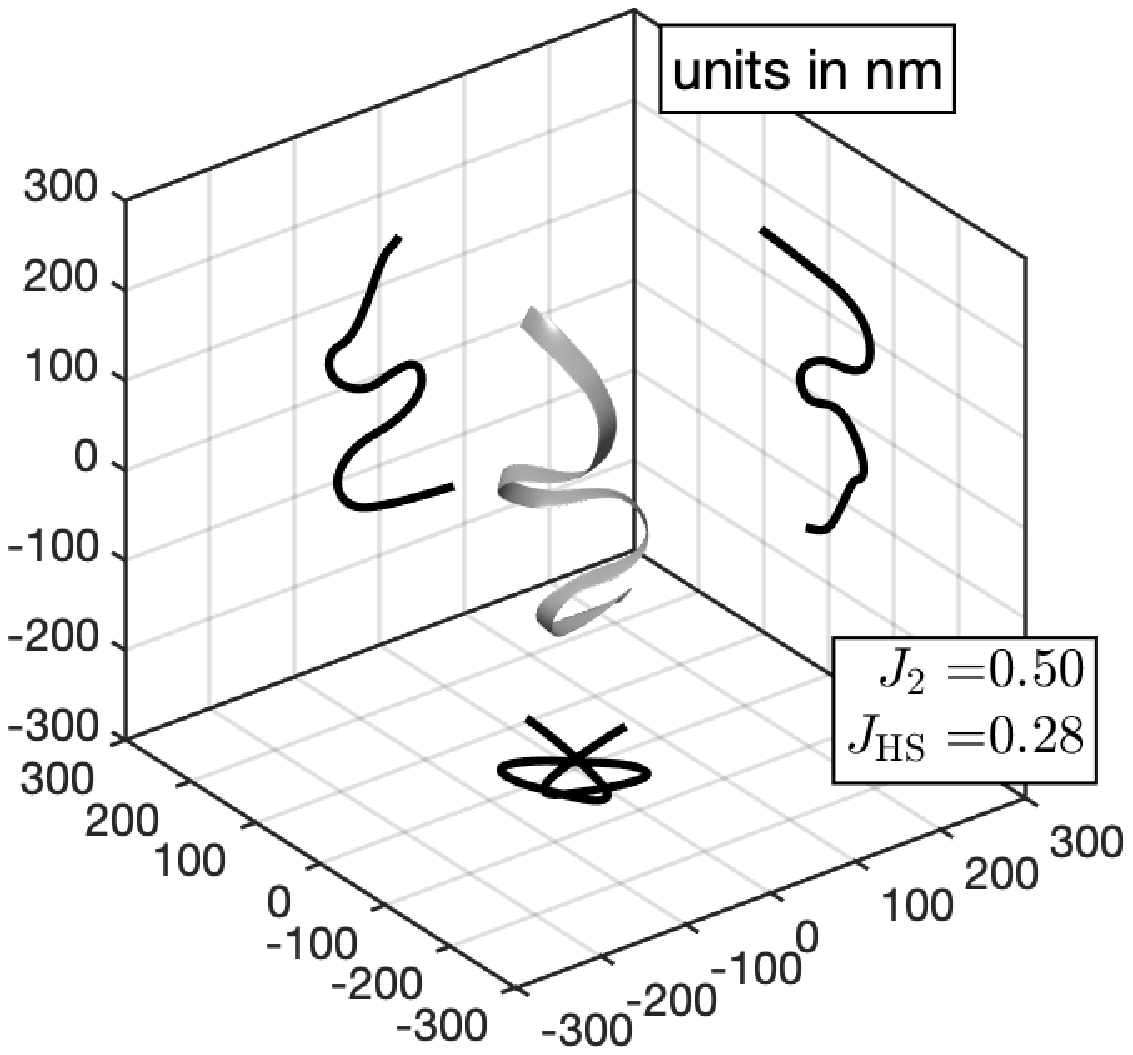}
    \includegraphics[height=3.78cm]{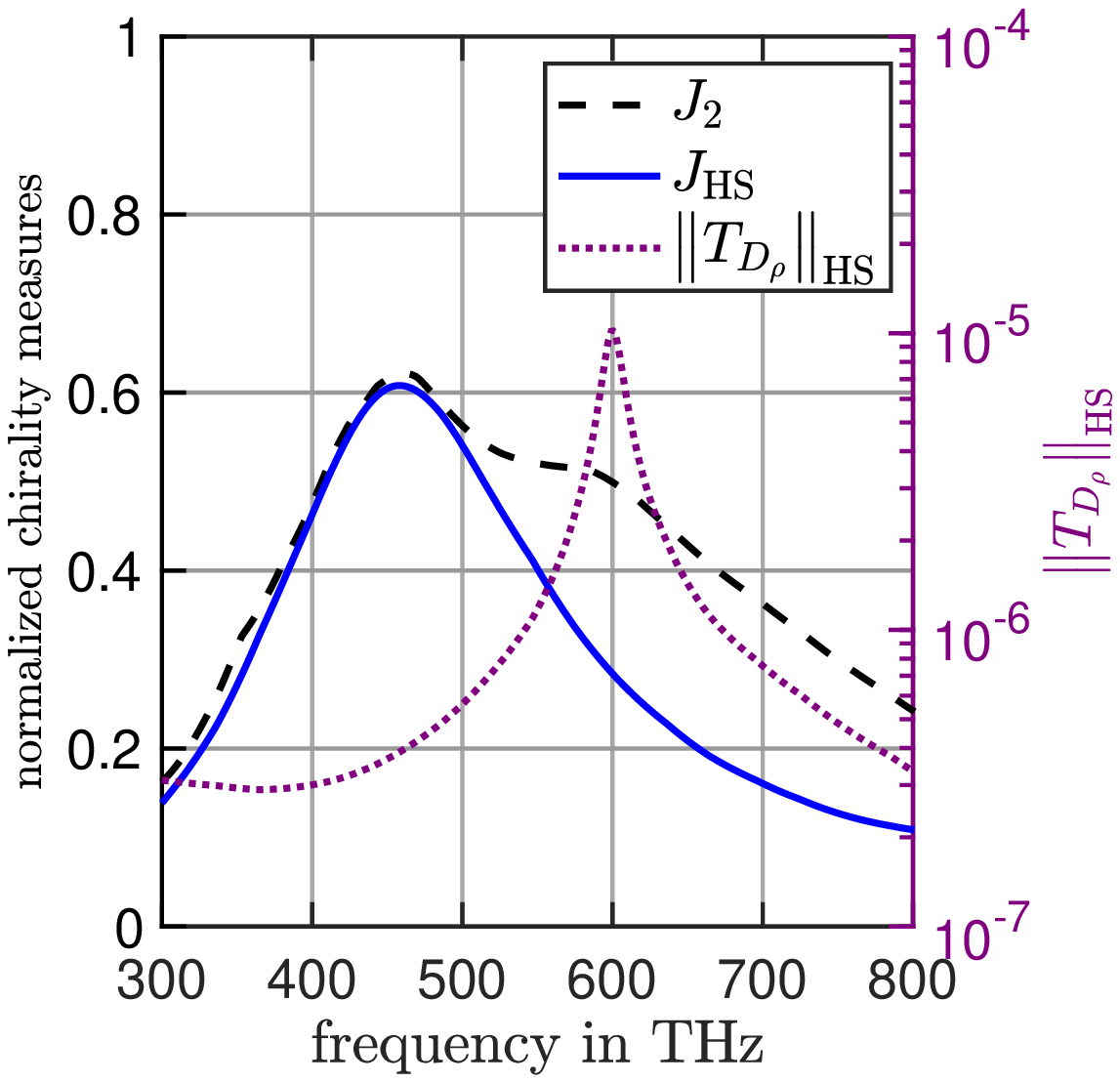}
    \includegraphics[height=4.2cm]{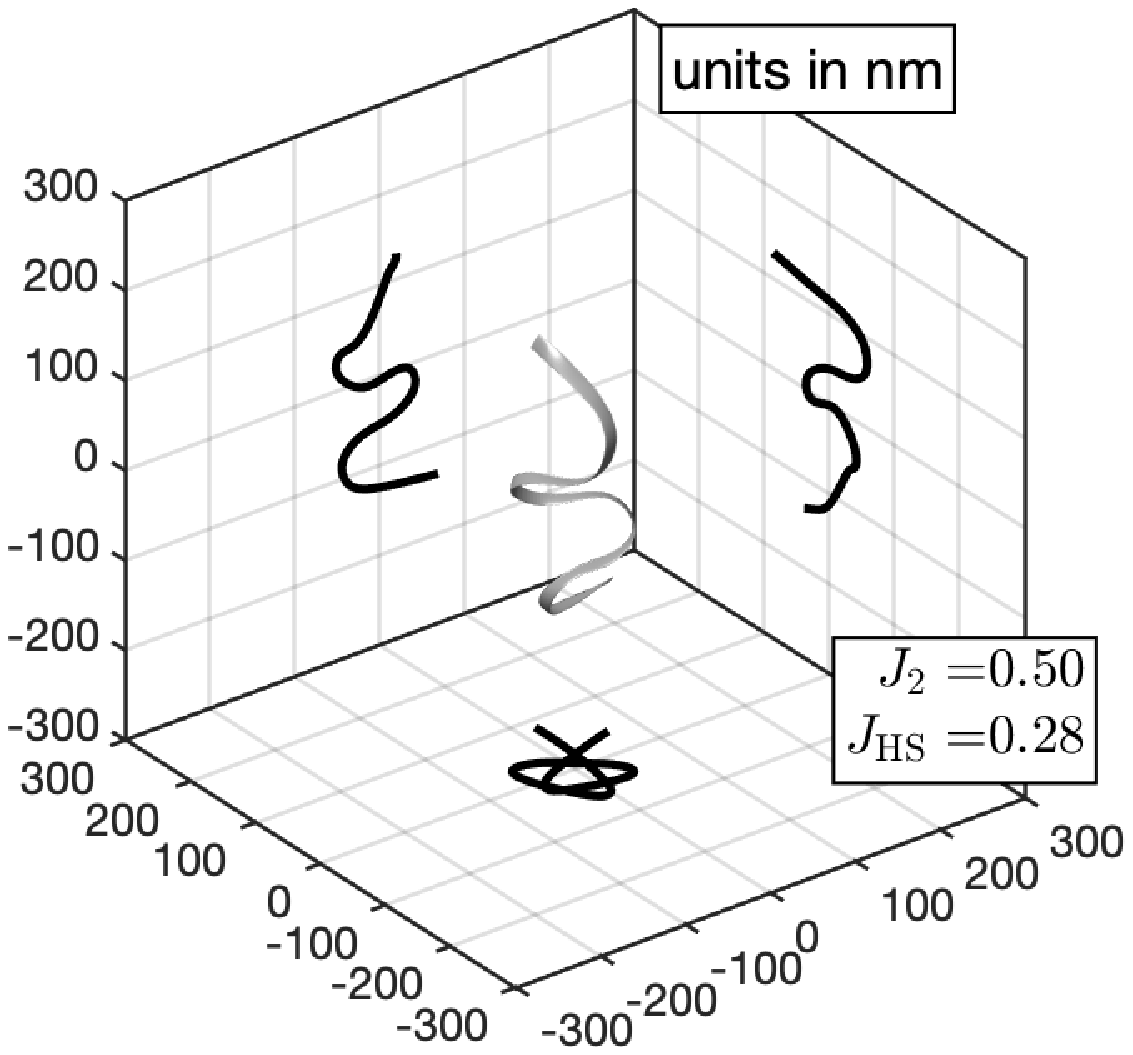}    
    \includegraphics[height=3.78cm]{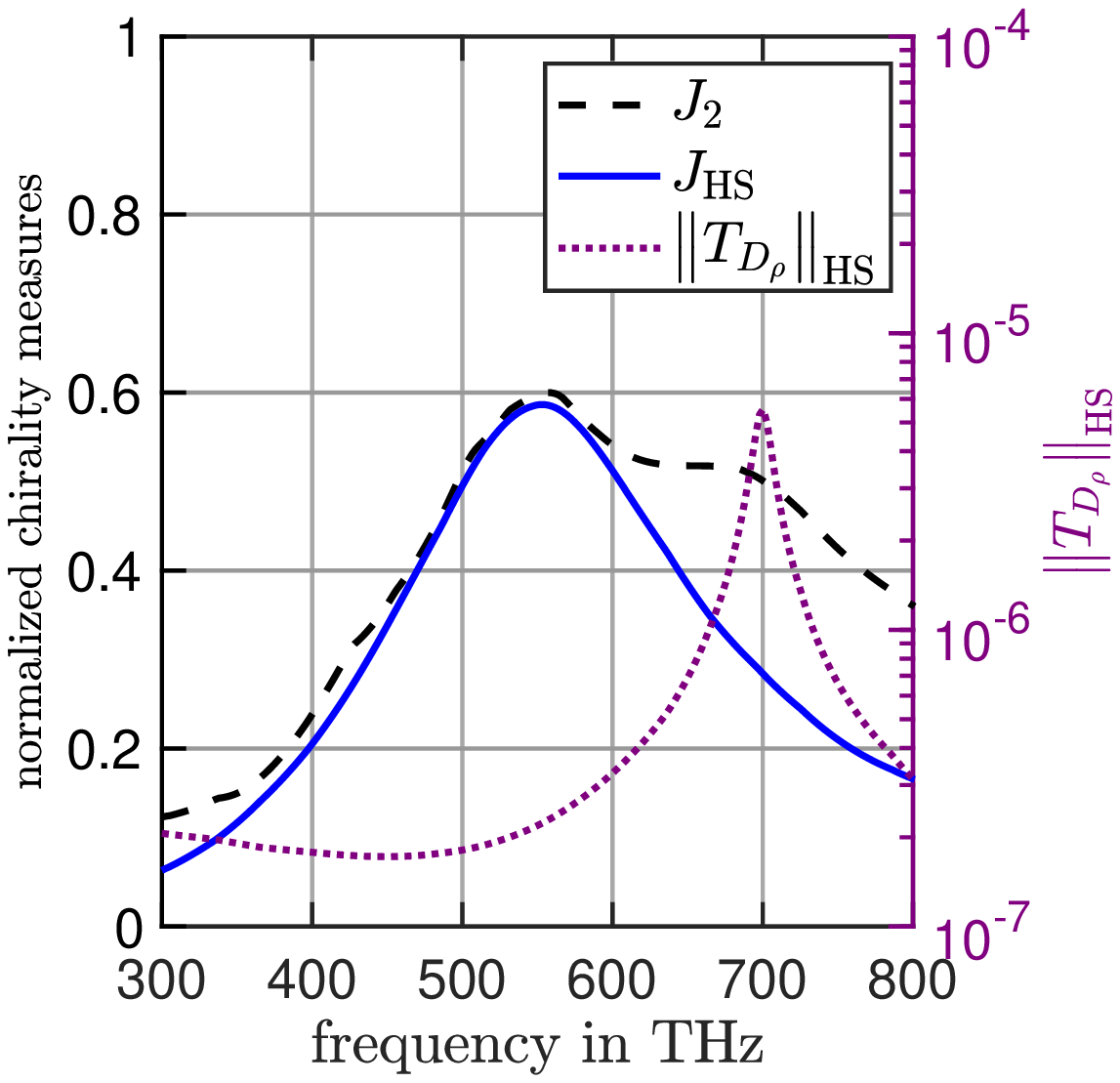} 
    \caption{Optimized silver nanowires and corresponding frequency
      scans from Example~\ref{exa:2} for
      $\fopt=400,500$~THz (top row) and
      $\fopt=600,700$~THz (bottom row).}
    \label{fig:Exa2-1}
  \end{figure}  

  In Figure~\ref{fig:Exa2-1} we show the optimized silver nanowires
  that have been obtained by the shape optimization for
  $\fopt=400,500,600$, and~$700$~THz. 
  The aspect ratios $b/a$ of the elliptical cross-sections are the
  same as in Example~\ref{exa:1} and vary between $26.94$ at
  $\fopt=400$~THz and $5.94$ at $\fopt=700$~THz.
  For a better three-dimensional impression, we also included the
  projections of the spine curves of the optimized nanowires on the
  three coordinate planes in these plots. 
  During the optimization the straight initial guess bends into
  a rather irregular shape that is difficult to interpret.
  However the optimized silver nanowires obtained at the four
  different frequencies have very similar shapes, which seem to be
  rescaled versions of each other with respect to the wave length.

  Figure~\ref{fig:Exa2-1} also contains plots illustrating the
  frequency dependence of the normalized em-chirality measures $\Jtwo$
  (dashed) and $\JHS$ (solid) as well as of the total interaction
  cross-section (dotted) of the optimized silver nanowires. 
  The maxima of the normalized em-chirality measures and the plasmonic
  resonances visible in the plots of the total interaction
  cross-section are rather localized.
  It is interesting to observe that, although the shape optimization
  has been carried out at the plasmonic resonance frequency, i.e.,
  $\fopt=\fres$, the maximum of the normalized em-chirality measures
  $\Jtwo$ and $\JHS$ is attained around $100$ to $150$~THz 
  below the plasmonic resonance frequency in all four examples.
  This is a common feature that we have observed in several other
  examples, and we will utilize this phenomenon in Example~\ref{exa:3}
  below to design thin silver and gold nanowires with even higher
  em-chirality measures.
  
  \begin{figure}
    \centering
    \includegraphics[height=4.2cm]{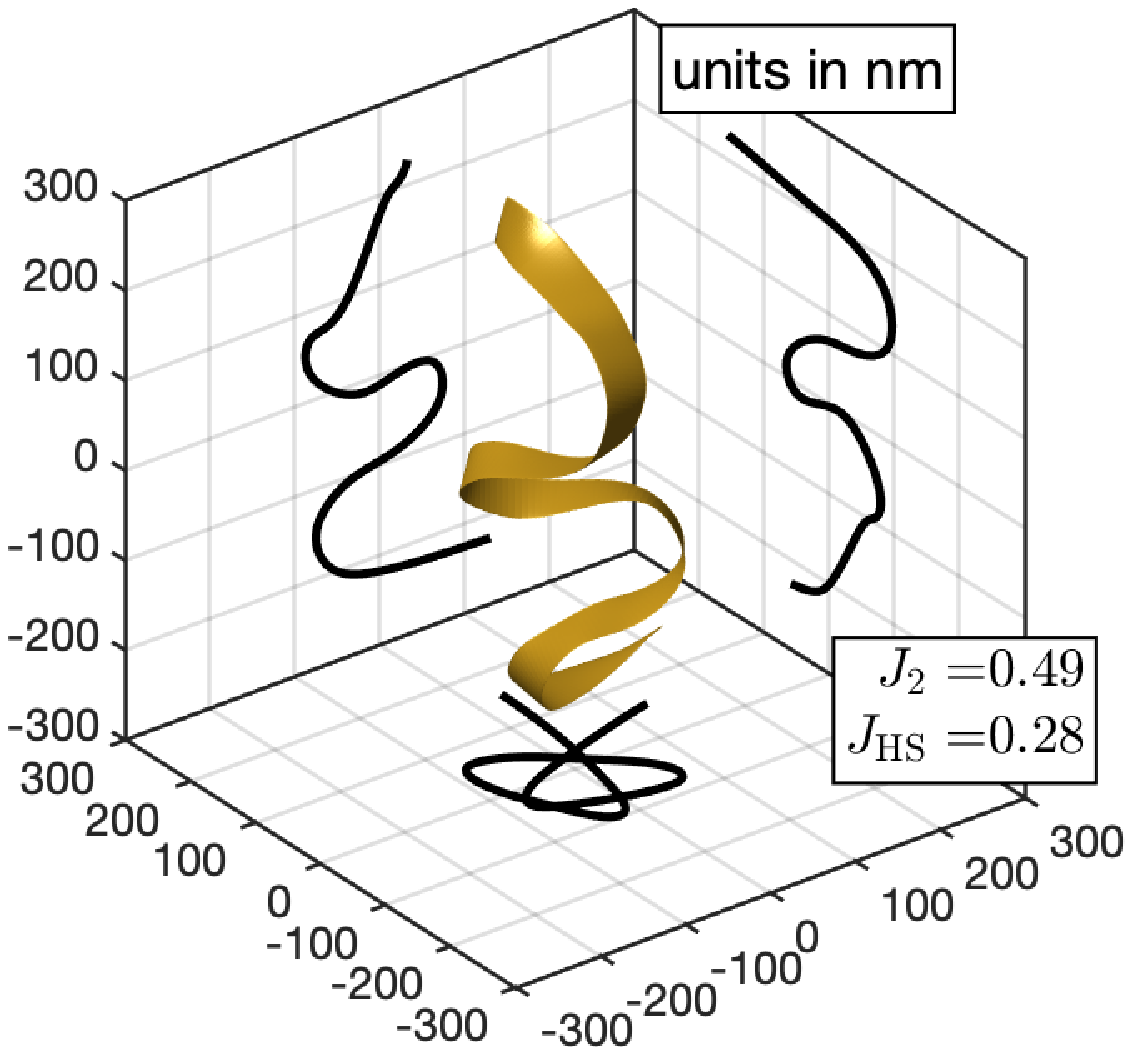}
    \includegraphics[height=3.78cm]{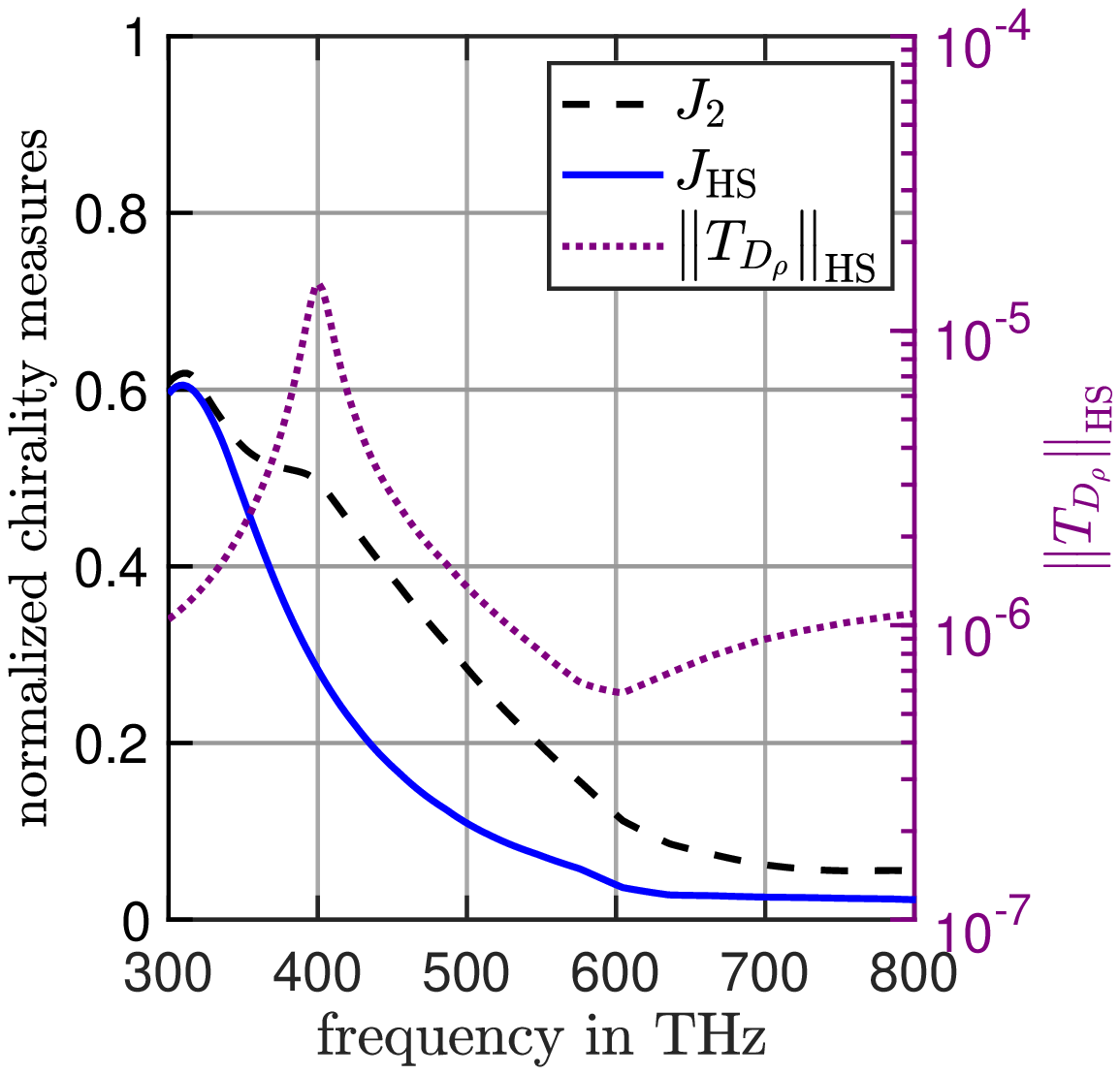}
    \includegraphics[height=4.2cm]{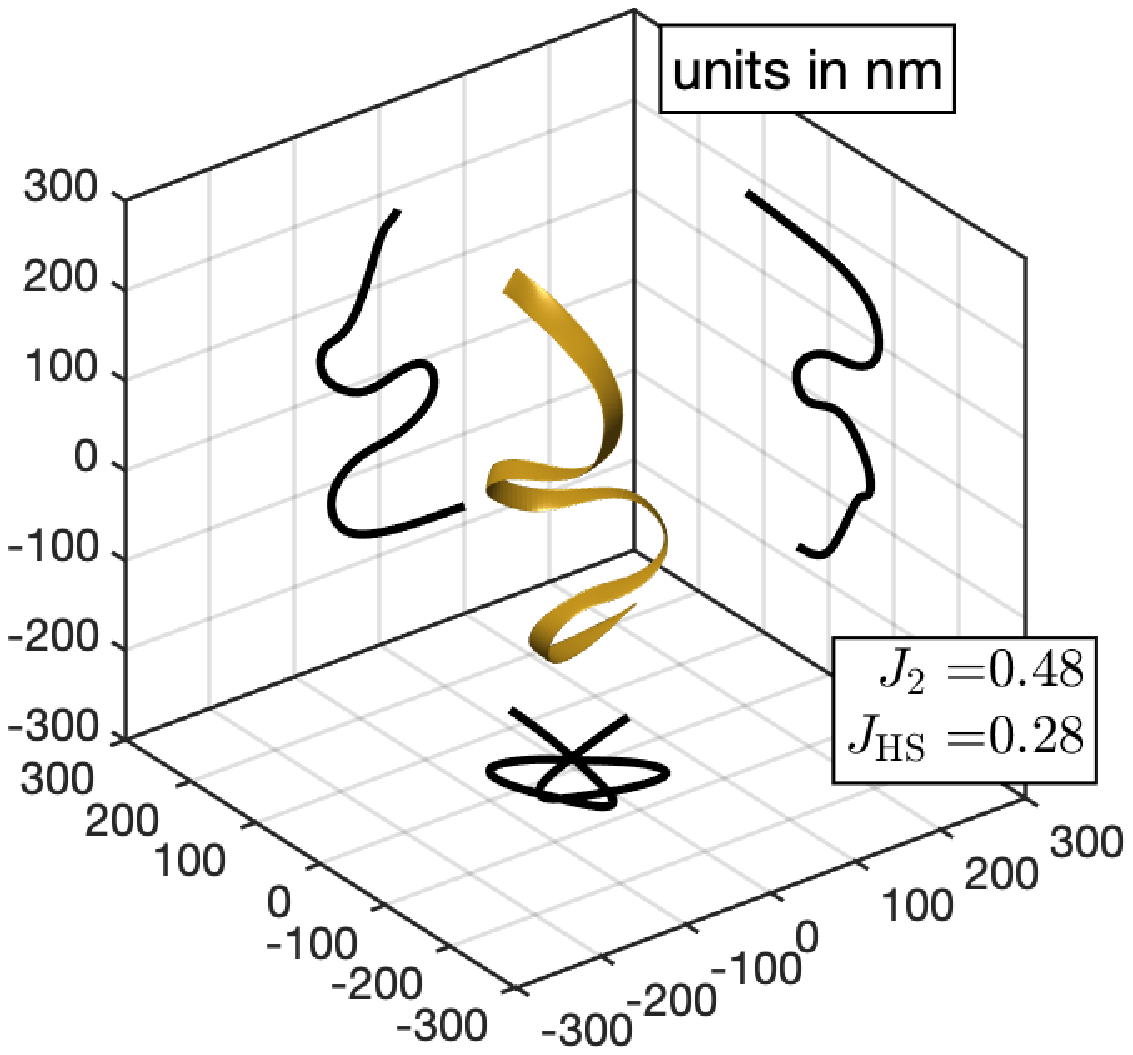}
    \includegraphics[height=3.78cm]{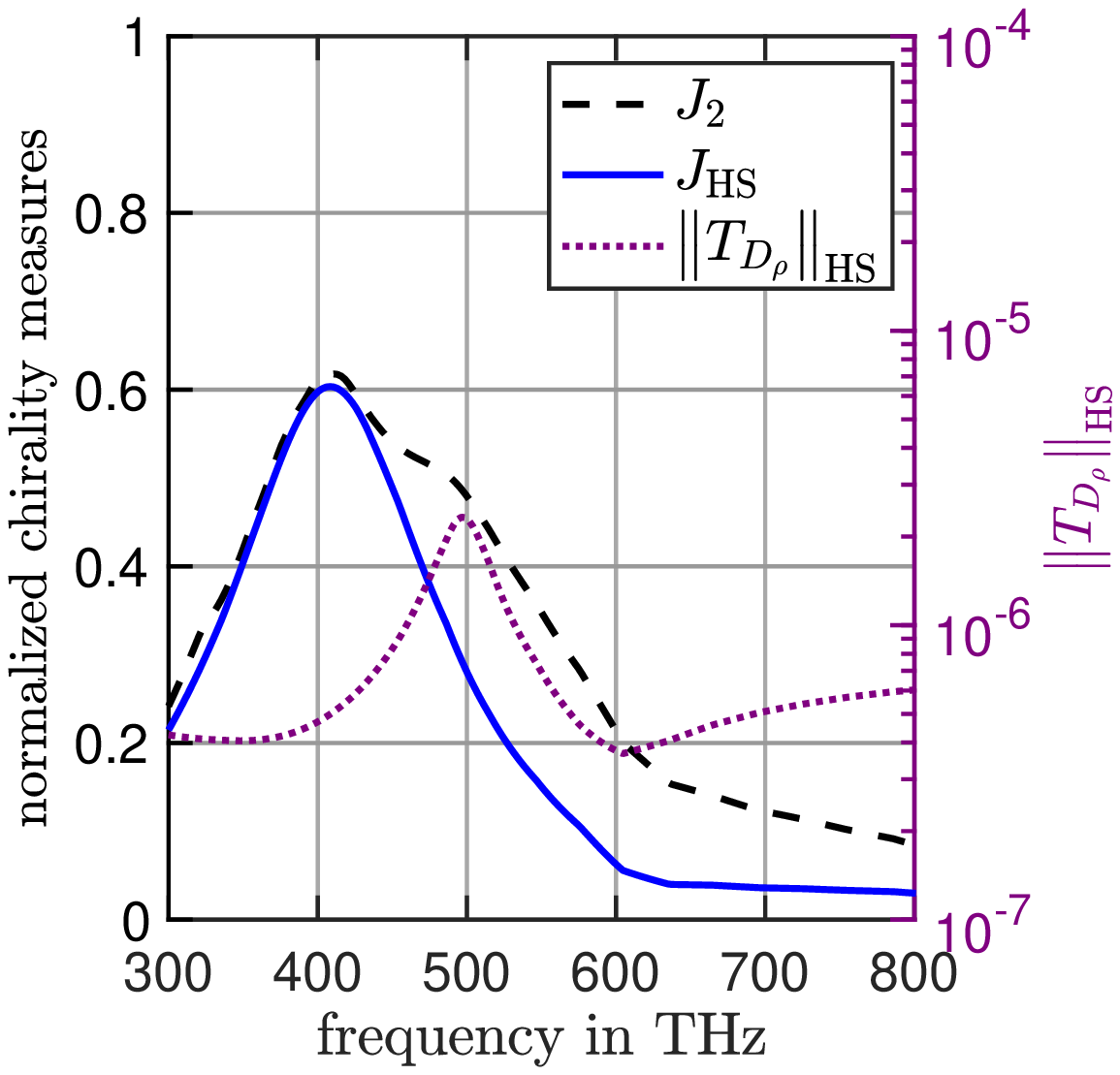}
    \includegraphics[height=4.2cm]{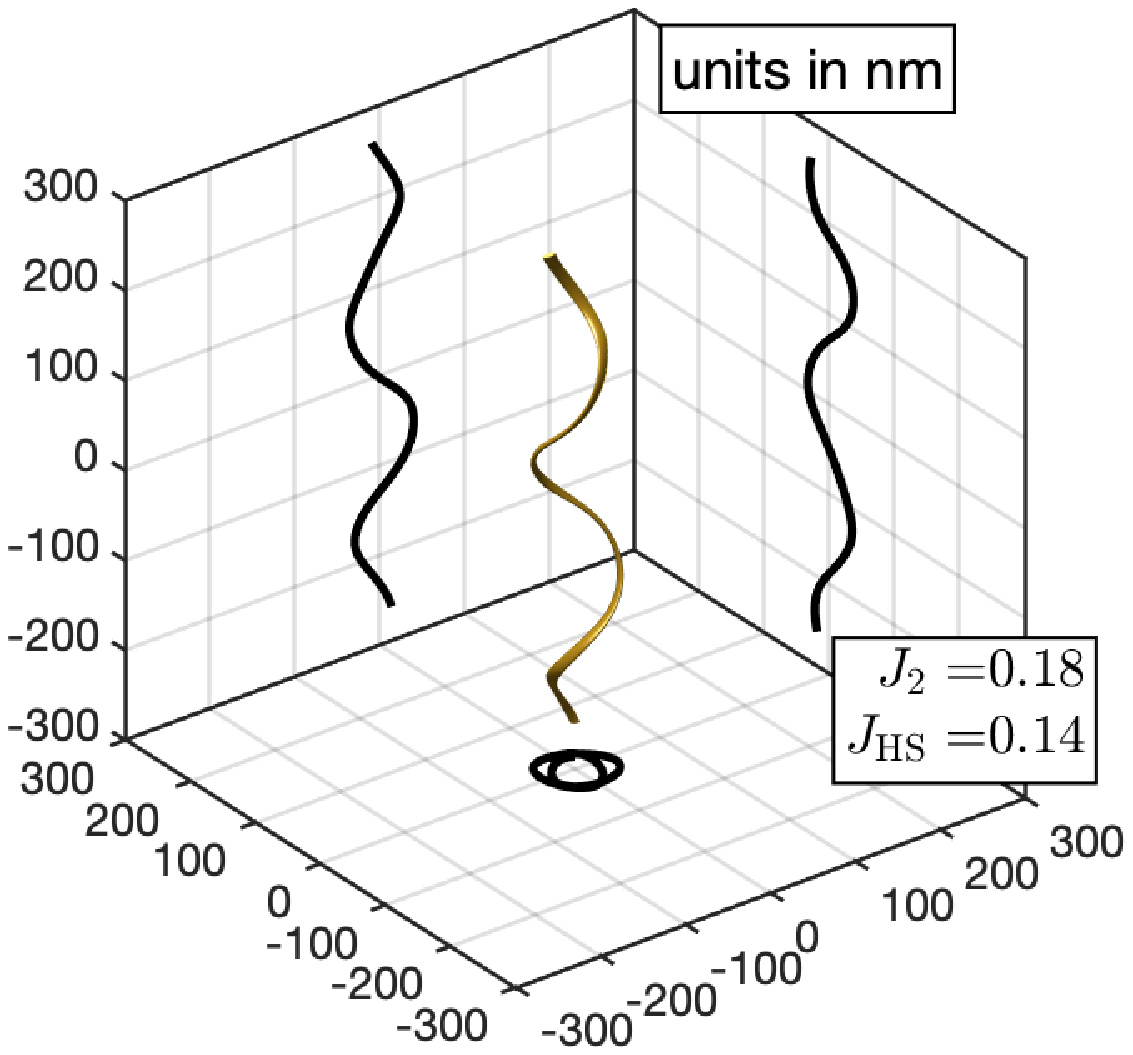}
    \includegraphics[height=3.78cm]{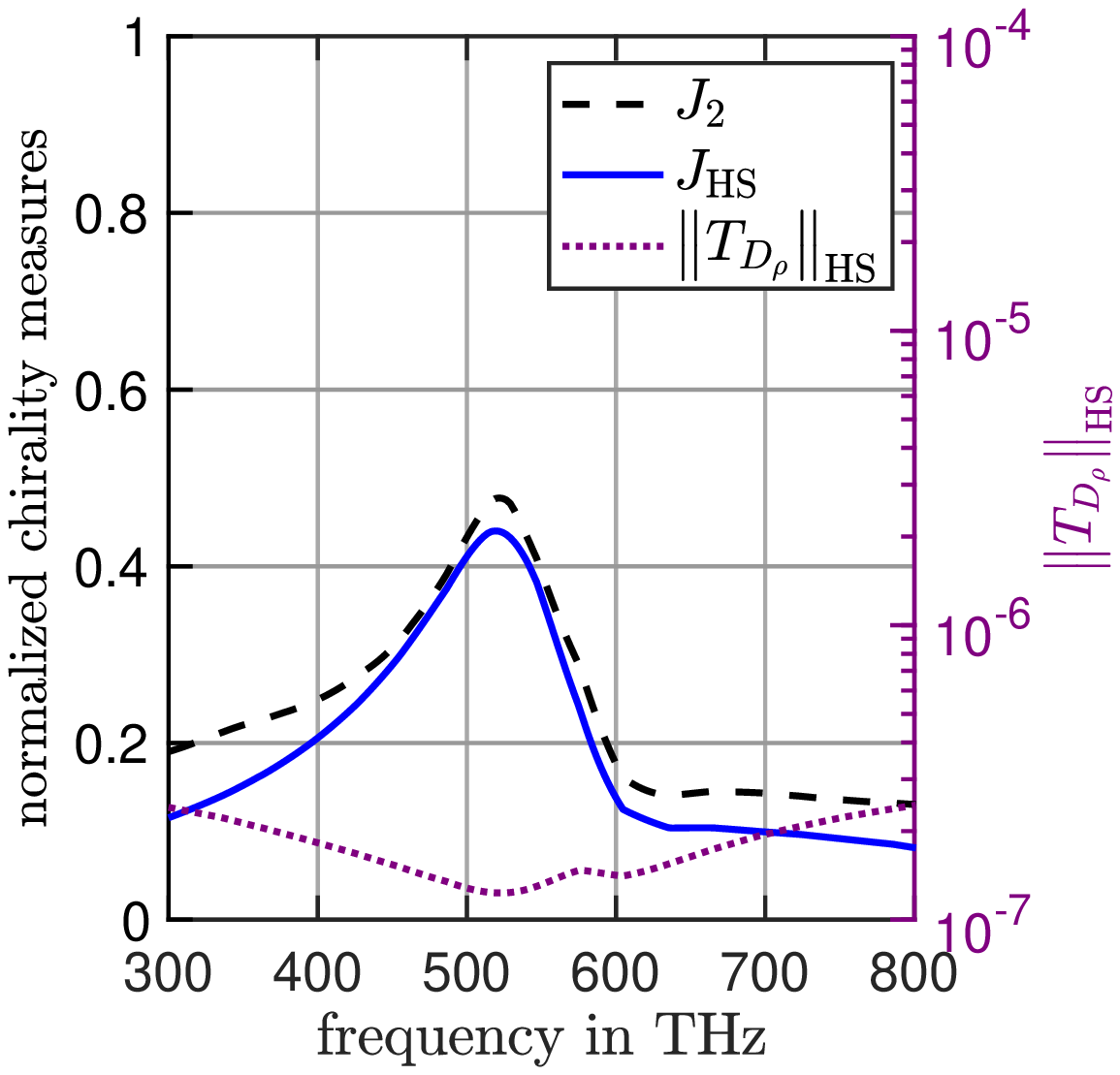}
    \includegraphics[height=4.2cm]{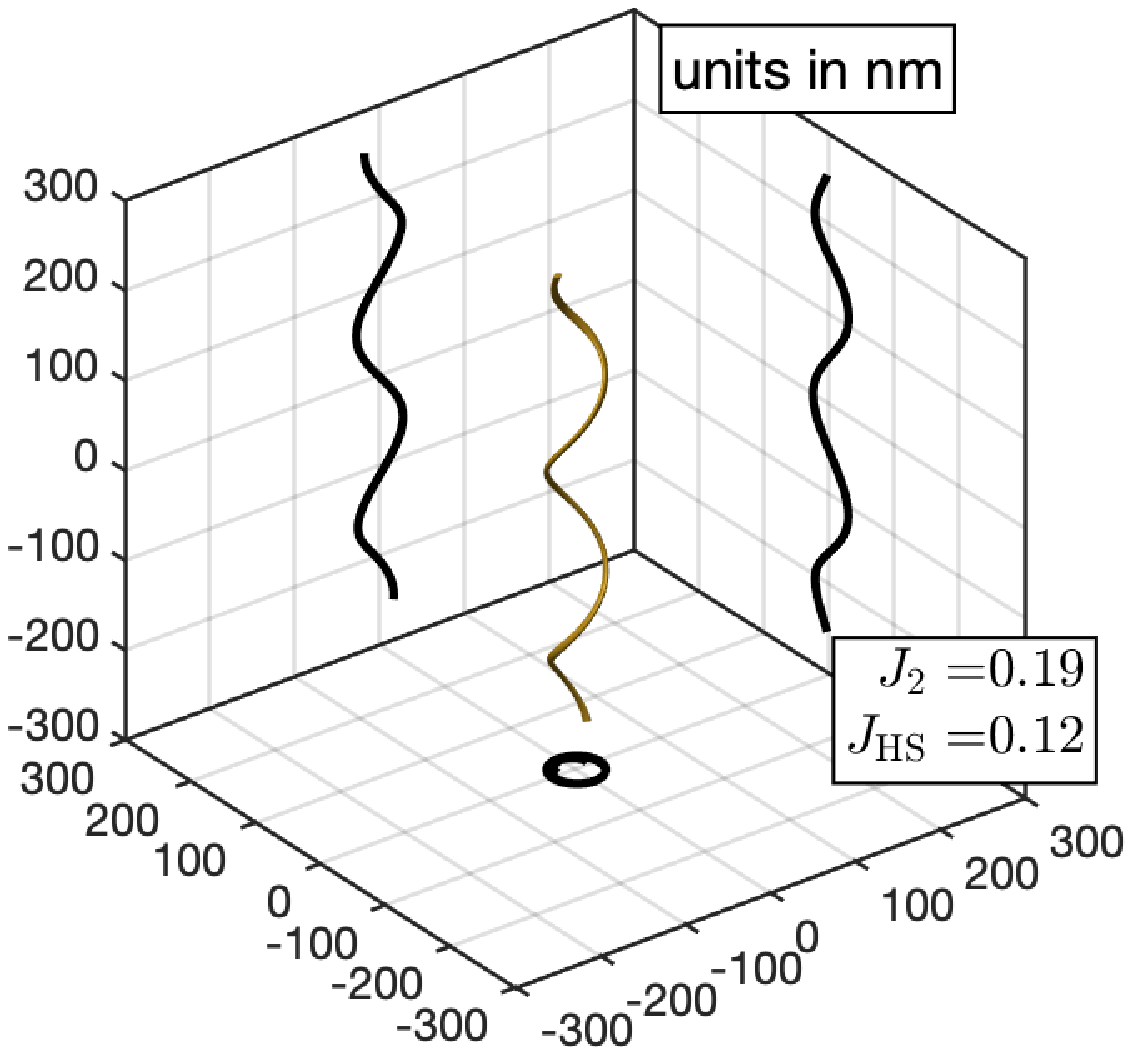}
    \includegraphics[height=3.78cm]{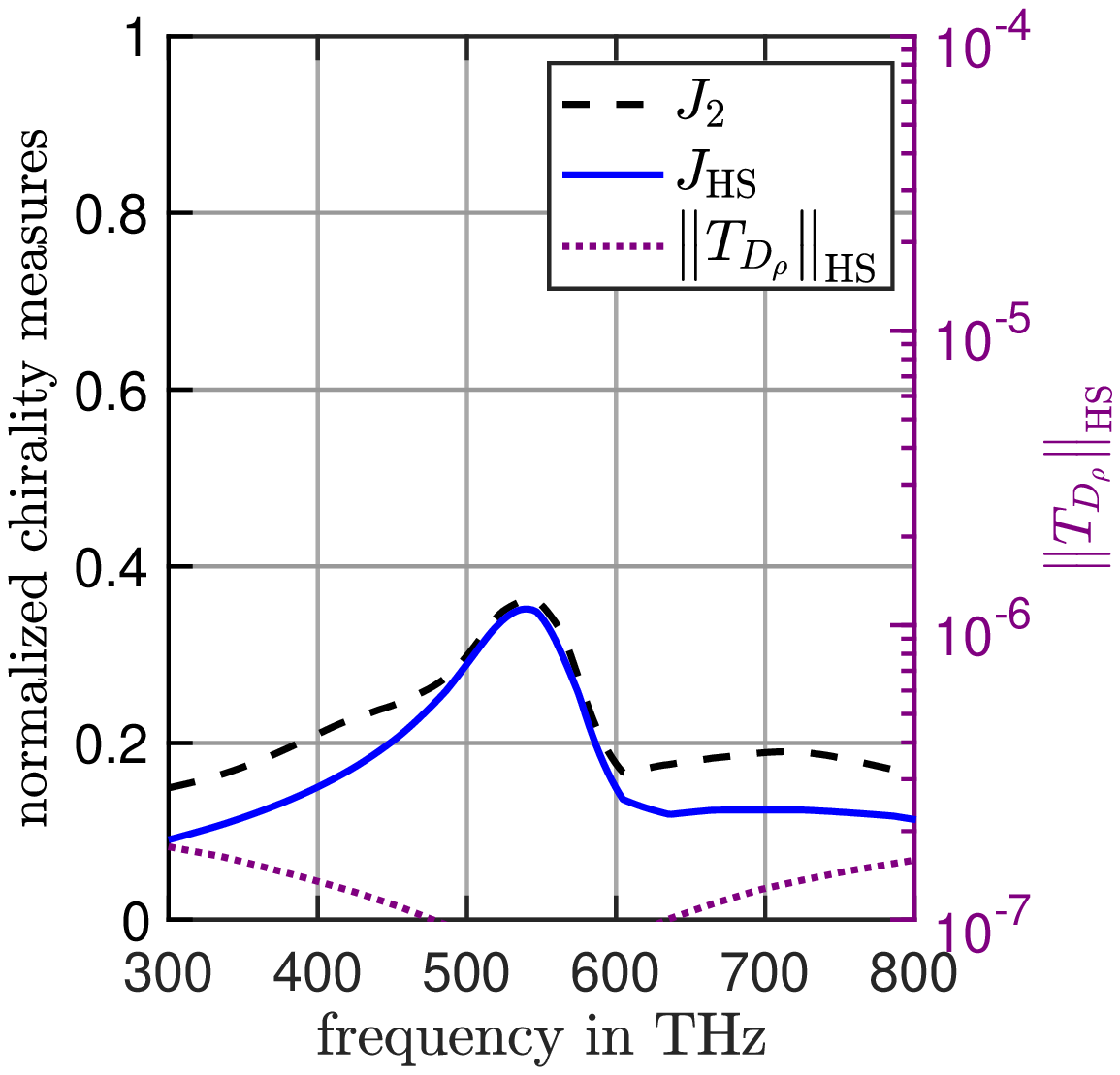}    
    \caption{Optimized gold nanowires and corresponding frequency
      scans from Example~\ref{exa:2} for 
      $\fopt=400,500$~THz (top row) and
      $\fopt=600,700$~THz (bottom row).}
    \label{fig:Exa2-2}
  \end{figure}

  In Figure~\ref{fig:Exa2-2} we show the corresponding results of the
  shape optimization of thin gold nanowires for~$\fopt=400,500,600$,
  and~$700$~THz. 
  The shapes of the gold nanowires that have been optimized
  at~$\fopt=400,500$~THz are similar to those of the optimized thin 
  silver nanowires in Figure~\ref{fig:Exa2-1}.
  Also the frequency scans in Figure~\ref{fig:Exa2-2} show a similar
  behavior, although the plasmonic resonance is not as pronounced as
  for the silver nanowires. 
  For~$\fopt=600,700$~THz the results are different.
  The shapes of the optimized gold nanowires look similar to those 
  obtained for dielectric nanowires in \cite{AreGriKno21}. 
  The optimized gold nanowires are helices, and no plasmonic
  resonances are visible in the plots of the total interaction
  cross-sections.
  This different behavior results from the larger imaginary part
  of the electric permittivity of gold at $\fopt=600,700$~THz (see
  Table~\ref{tab:MaterialParameters}).~\hfill$\lozenge$
\end{example}

\begin{example}[Optimizing the twist rate of the cross-section and the
  shape of the spine curve of the nanowire] 
  \label{exa:3}
  The goal of our final example is to design thin silver and gold
  nanowires with elliptical cross-sections that possess normalized
  em-chirality measures $\Jtwo$ and~$\JHS$ as close to $1$ as possible at
  optical frequencies.  
  We consider a free-form shape optimization for the spine curve of
  the nanowire, and we optimize the twisting of the elliptical
  cross-section of the nanowire along the spine curve.  
  As in the previous examples, we discuss four different frequencies
  $\fopt=400,500,600$, and $700$~THz.
  For each of these frequencies we choose the aspect ratio of the
  elliptical cross-section of the nanowire such that its plasmonic
  resonance frequency $\fres$ is around $100$ to $150$~Thz above the
  frequency $\fopt$ that is used in the shape optimization, i.e., we
  use $b/a=12.5$ at $\fopt=400$~THz, $b/a=7.14$ at $\fopt=500$~THz,
  $b/a=3.85$ at $\fopt=600$~THz, and~$b/a=1.92$ at $\fopt=700$~THz.
  In particular $\fopt\not=\fres$. 
  This is different from the previous two examples, where we optimized
  the shape of the nanowires directly at the plasmonic resonance
  frequency.
  It is motivated by our observations at the end of
  Example~\ref{exa:2}. 
  We apply the optimization scheme from
  Section~\ref{sec:ShapeOptimization}. 
  For the regularization parameters in \eqref{eq:DefPhi} we choose
  $\alpha_1=5$, $\alpha_2=8\times 10^{-3}$ and
  $\alpha_3=1\times 10^{-6}$.

  The outcome of the shape optimization strongly depends on the
  initial guess for the spine curve.
  Thus, we consider in this example $100$ different initial spine
  curves for the optimization scheme at each frequency. 
  These are helices with four turns, where the total height and
  the radius of the helix are chosen randomly in~$[0, 2\lambdaopt/3]$ and
  in~$[0, \lambdaopt/2]$, respectively.
  As before, $\lambdaopt$ denotes the wave length at the frequency~$\fopt$. 
  We also add different random twists to these initial guesses.
  We use cubic not-a-knot splines with $n=40$ knots to parametrize the
  spine curve and the twist function, and $21$ quadrature points on
  each spline segment to discretize line integrals over~$\Gamma$.
  We choose the maximal degree $N$ of circularly polarized vector
  spherical harmonics that is used in the discretization of the
  operator~$\bfTrho(\bfp,\Vptheta)$ and of its Fr\'echet derivative
  $\bfTrho'[\bfp,\Vptheta](\bfh,\phi)$ (see
  Remark~\ref{rem:DiscretisationT}) to be~$N=5$. 
  This gives $100$ different optimized silver and gold nanowires for
  each frequency $\fopt$, and we finally select those (for each
  frequency $\fopt$) that attain the highest normalized em-chirality
  measures. 


  \begin{figure}[thp]
    \includegraphics[height=4.2cm]{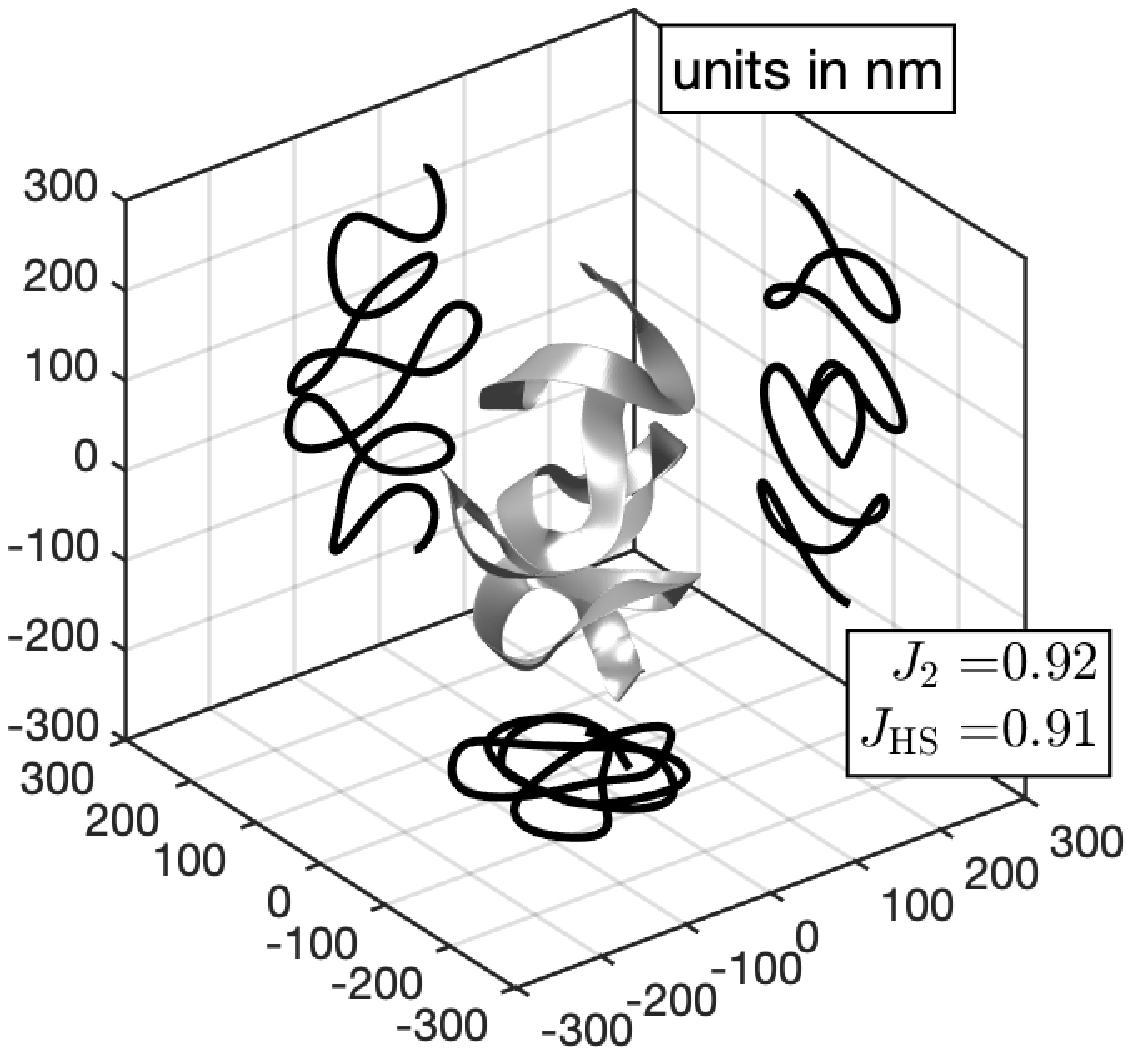}  
    \includegraphics[height=3.78cm]{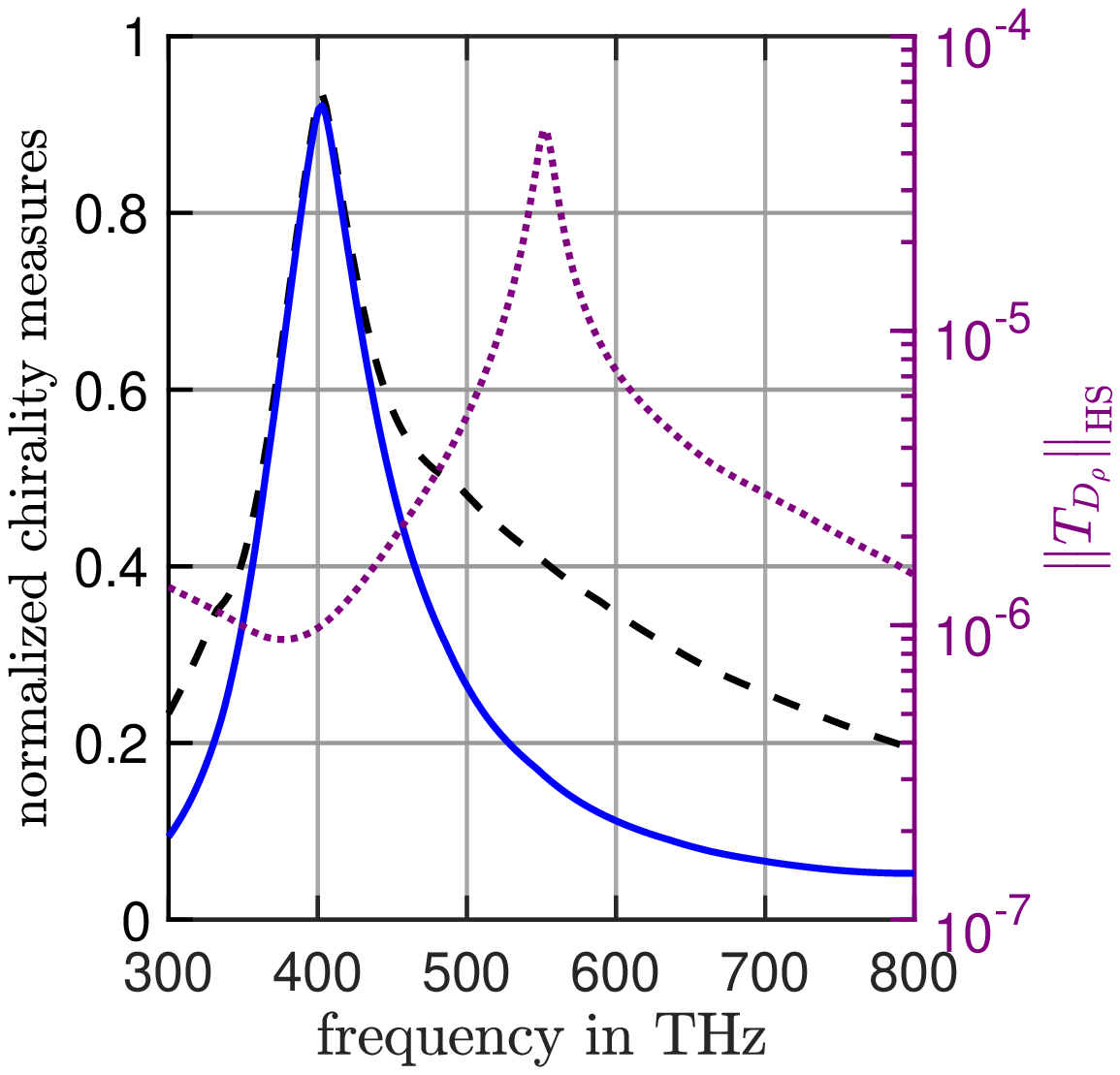} 
    \includegraphics[height=4.2cm]{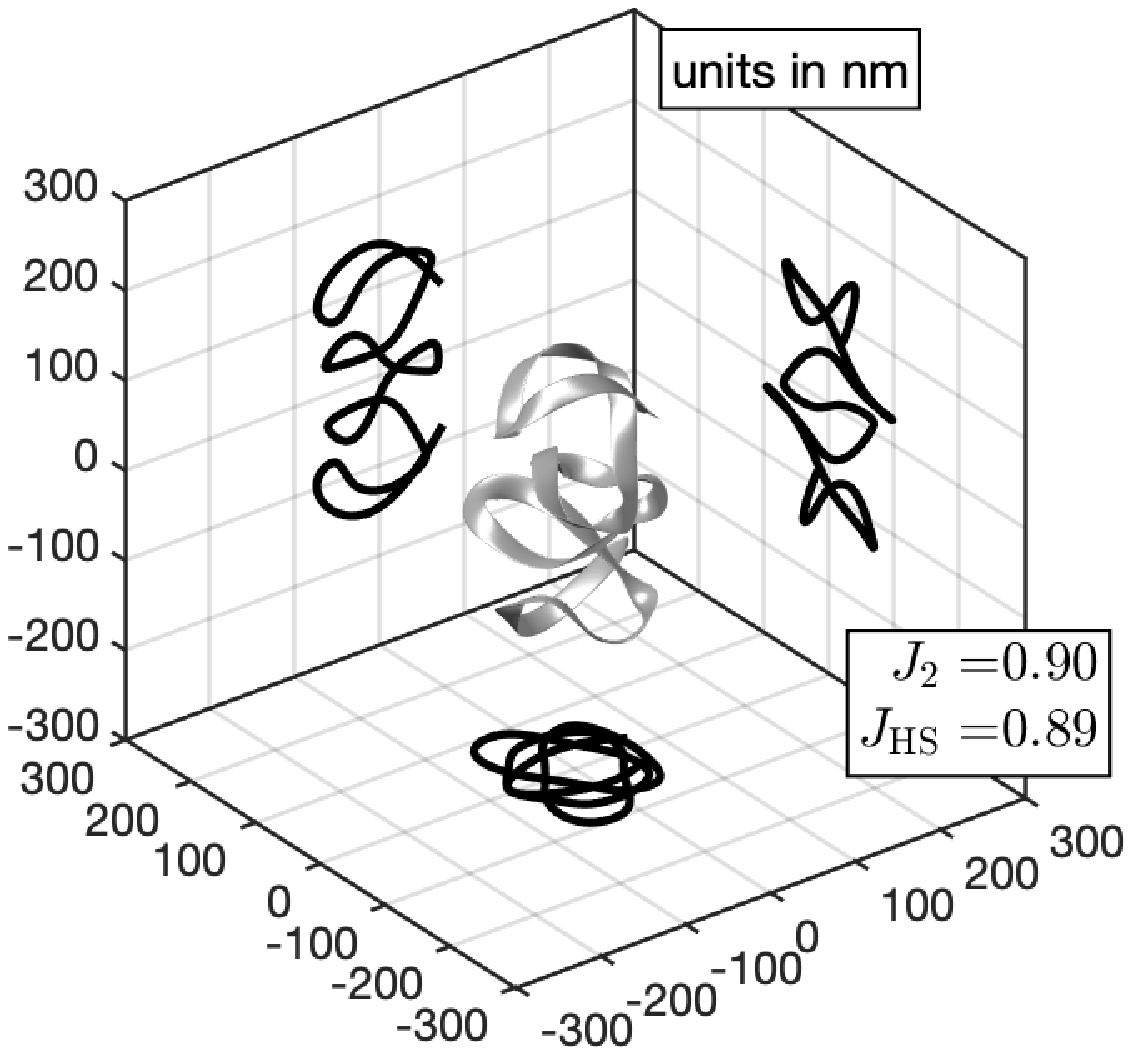} 
    \includegraphics[height=3.78cm]{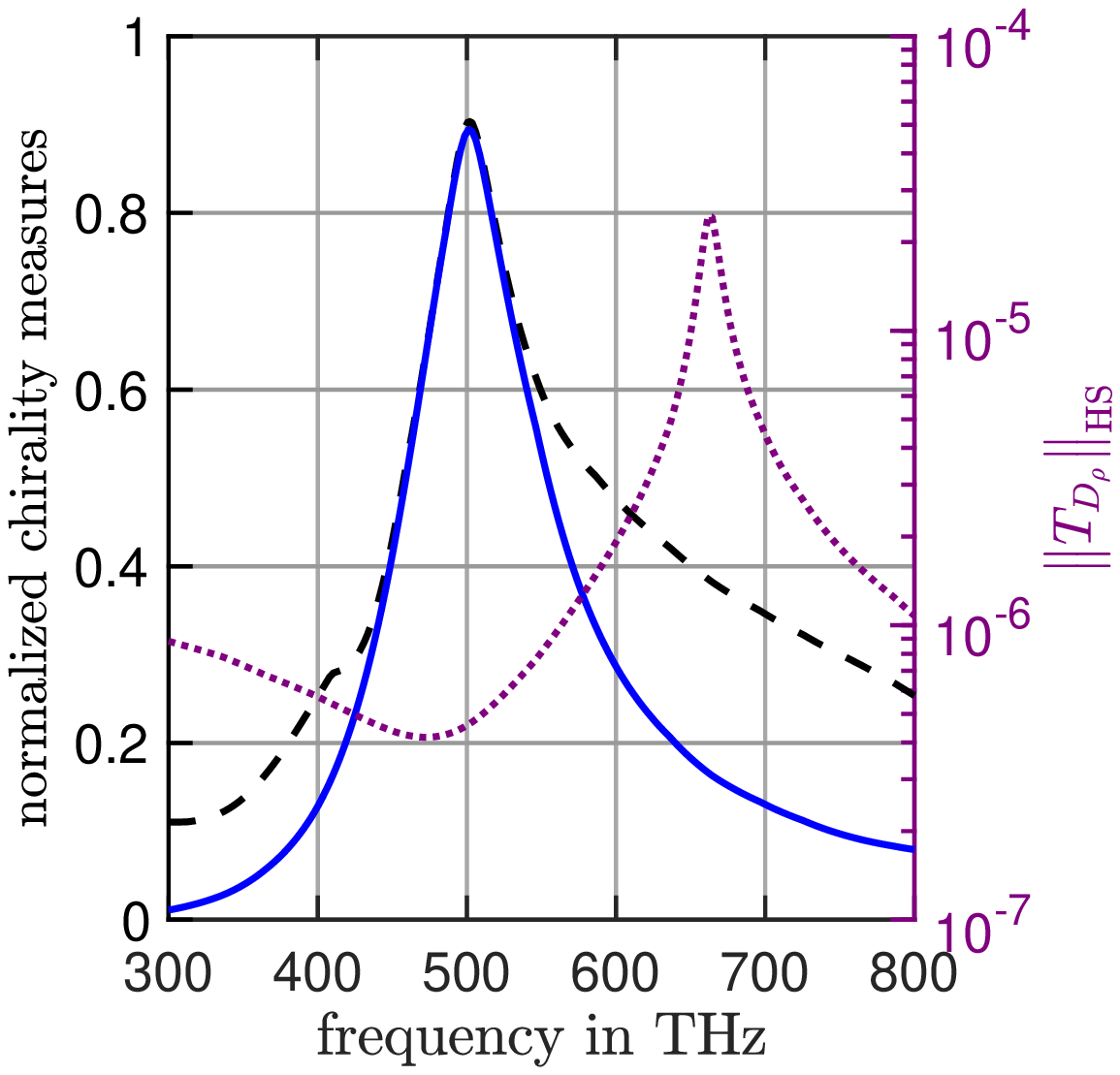} \\
    \includegraphics[height=4.2cm]{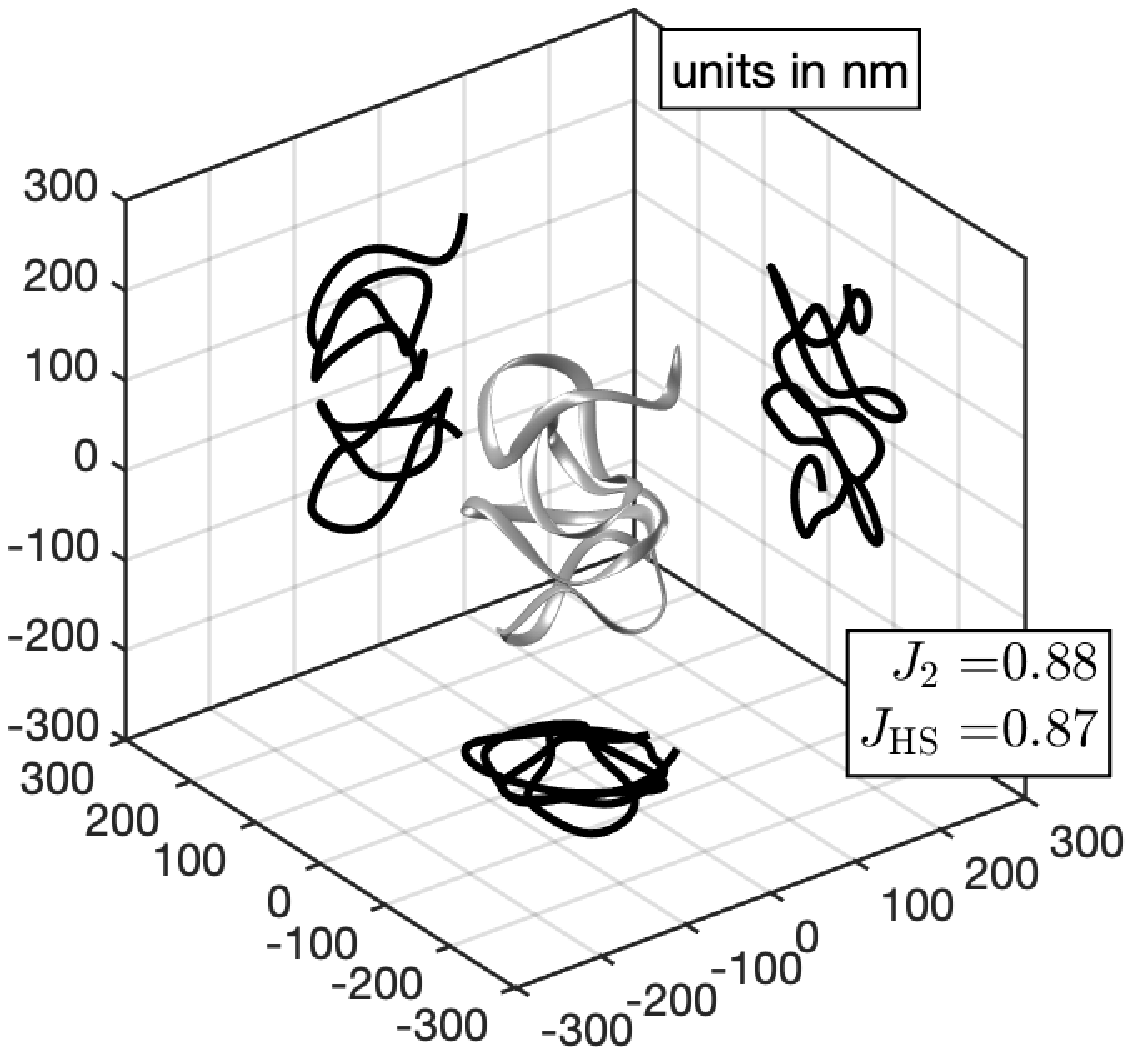} 
    \includegraphics[height=3.78cm]{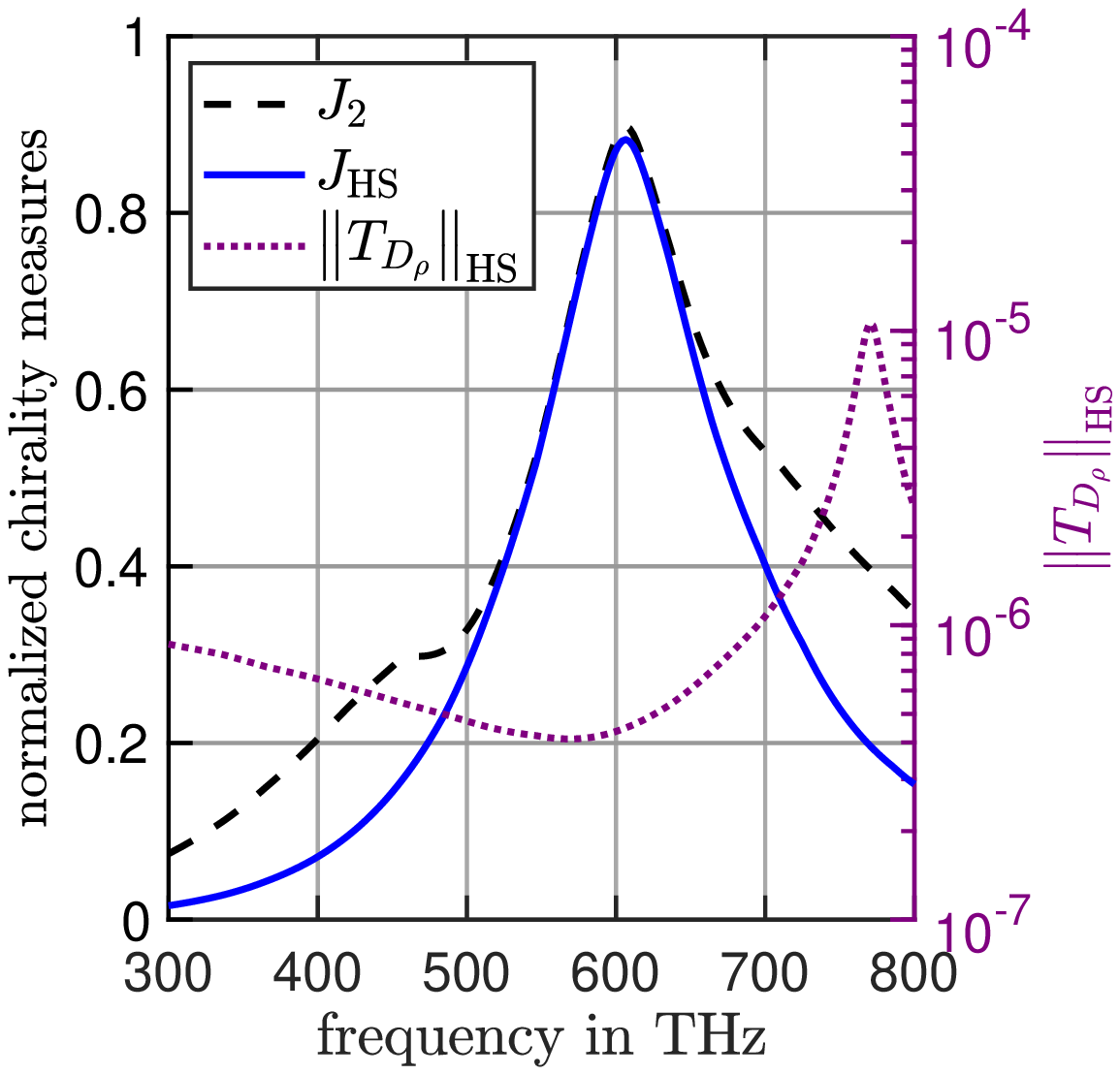} 
    \includegraphics[height=4.2cm]{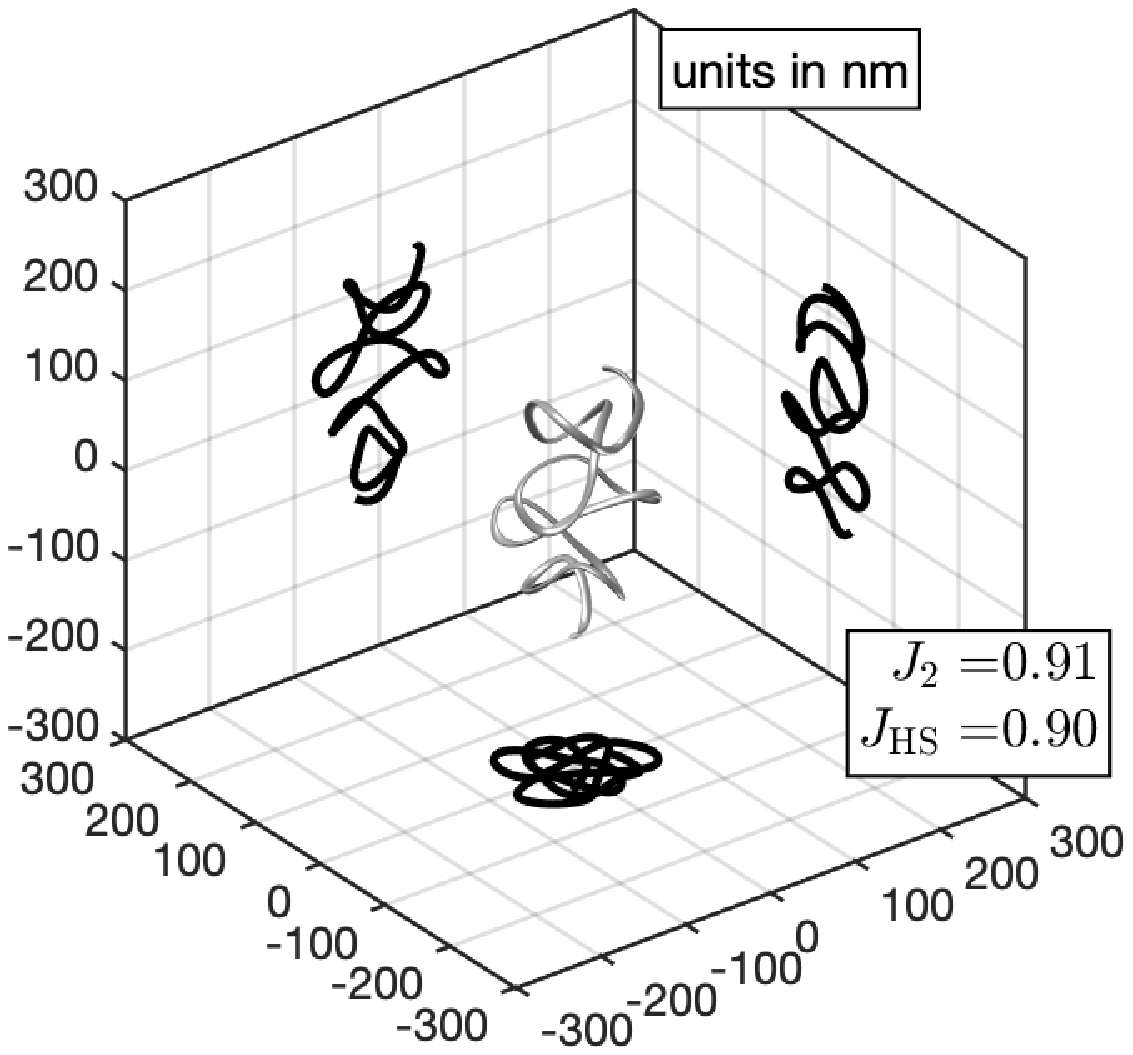} 
    \includegraphics[height=3.78cm]{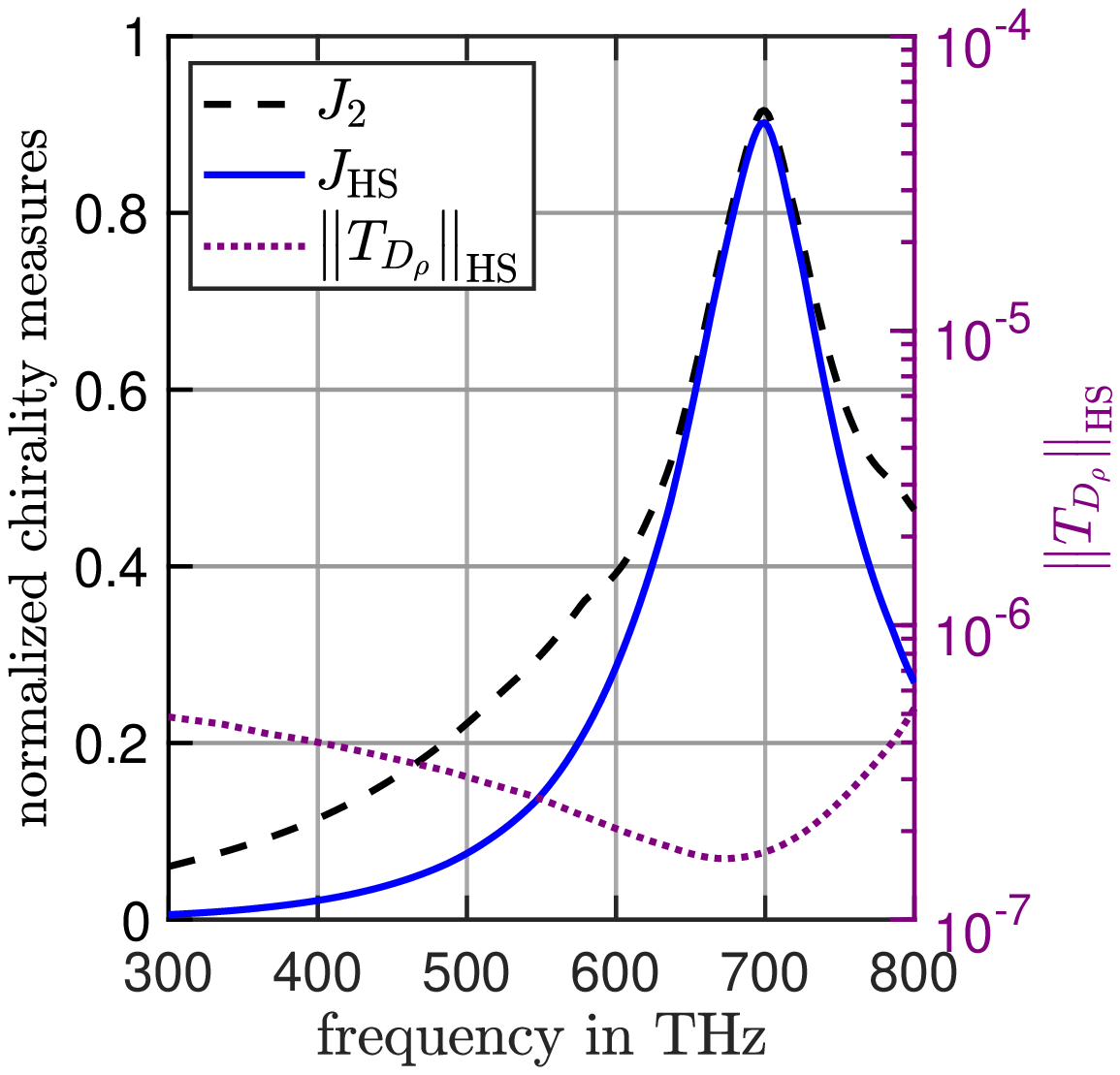} 
    \caption{Optimized silver nanowires and corresponding frequency
      scans from Example~\ref{exa:3} for $\fopt=400,500$~THz (top row) and
      for $\fopt=600,700$~THz (bottom row).}
    \label{fig:Exa3-1}
  \end{figure}

  In Figure~\ref{fig:Exa3-1} we show the optimized silver nanowires
  that have been obtained for $\fopt=400,500$~THz (top row) and for
  $\fopt=600,700$~THz (bottom row). 
  The shapes of the optimized nanowires look complicated but they show 
  similarities and seem to be scaled according to the wavelength,
  where the optimization has been carried out.
  Very high normalized em-chirality measures are being attained by the 
  optimized structures at all four frequencies $\fopt=400,500, 600$,
  and $700$~THz, respectively. 
  Figure~\ref{fig:Exa3-1} also shows plots of the normalized
  em-chirality measures $\Jtwo$ (dashed) and $\JHS$ (solid) as well as
  of the total interaction cross-section (dotted) of the optimized
  thin silver nanowires as a function of the frequency of the incident
  waves. 
  The maximal values of the normalized em-chirality measures appear at
  approximately the same frequency, where the total interaction
  cross-section of the nanowires has a local minimum.   
  Directly at the plasmonic resonance frequency the normalized
  em-chirality measures are smaller, but on the other hand the total
  interaction cross section of the nanowire is much larger.
  
  \begin{figure}[thp]
    \includegraphics[height=4.2cm]{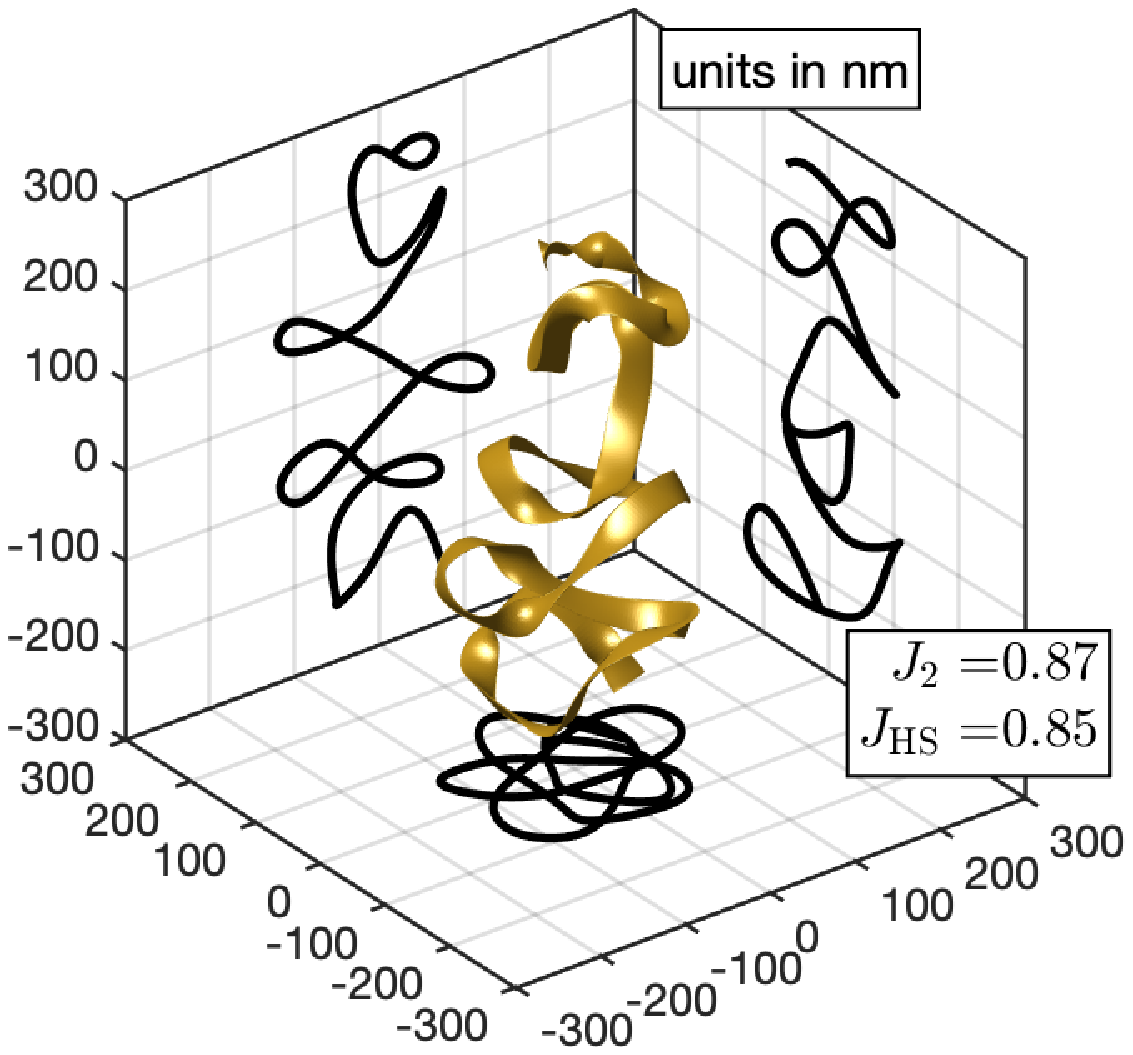}  
    \includegraphics[height=3.78cm]{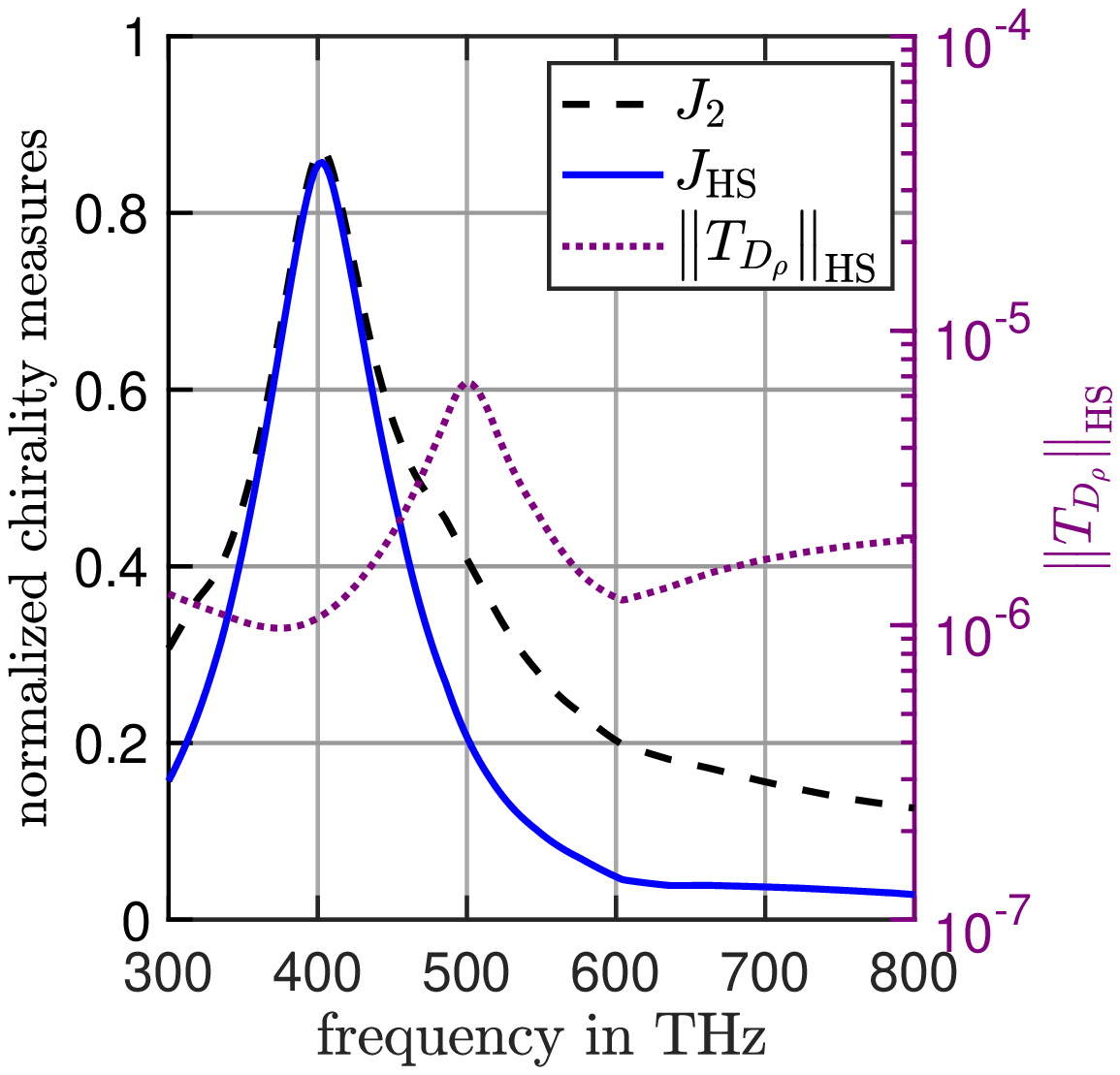}  
    \includegraphics[height=4.2cm]{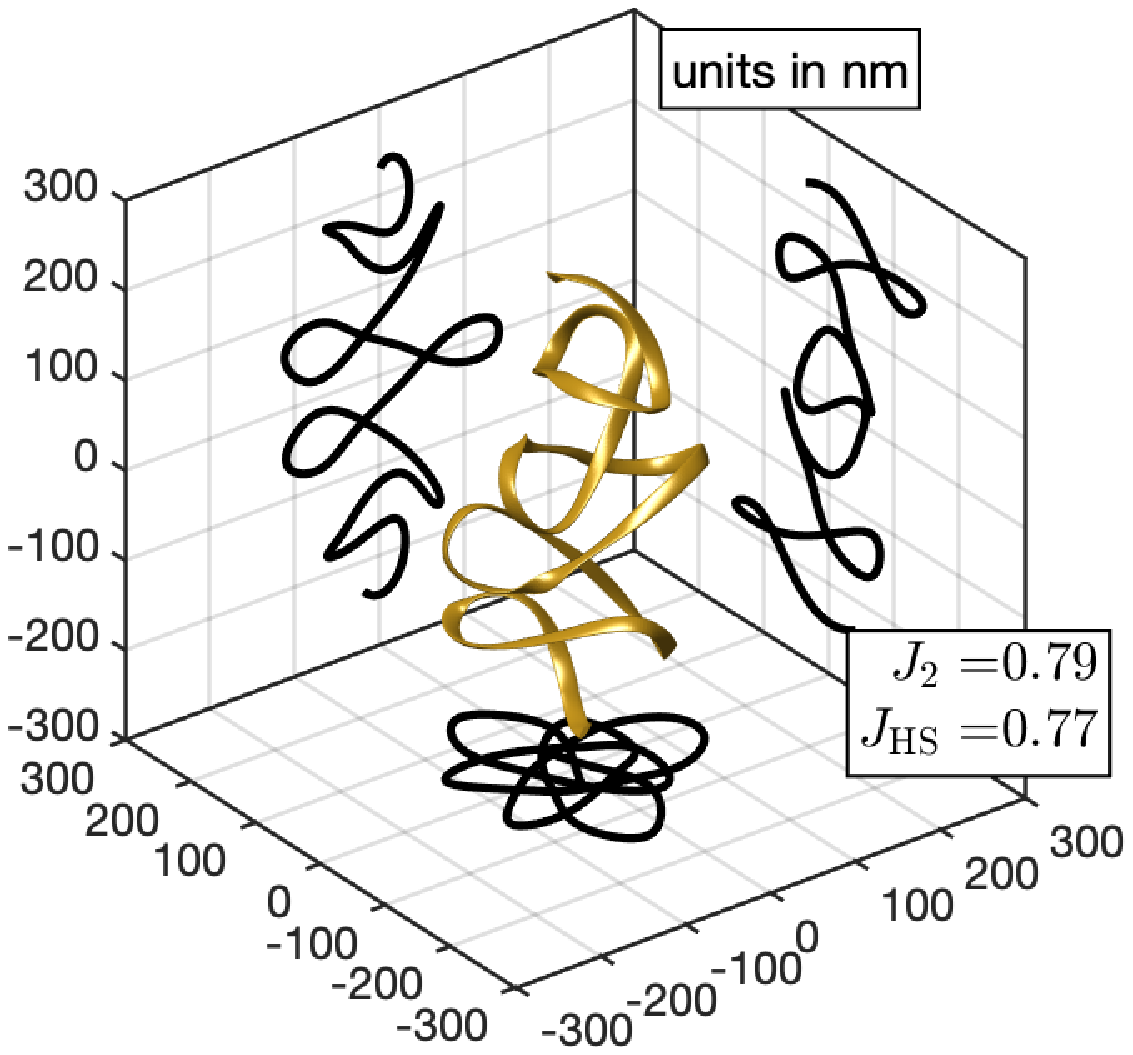}  
    \includegraphics[height=3.78cm]{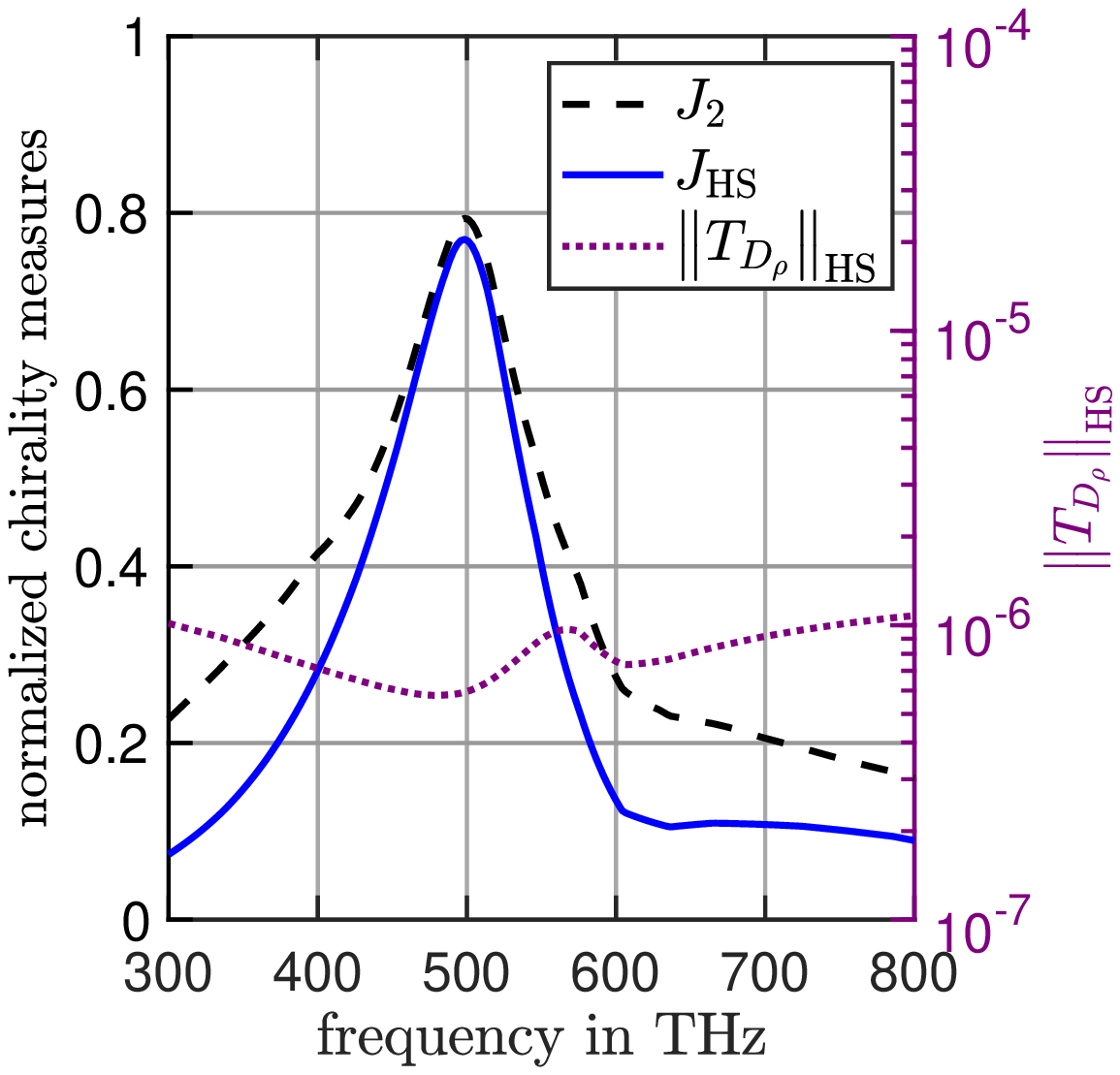} 
    \caption{Optimized gold nanowires and corresponding frequency
      scans from Example~\ref{exa:3} for $\fopt=400,500$~THz.}
    \label{fig:Exa3-2}
  \end{figure}

  In Figure~\ref{fig:Exa3-2} we show the corresponding results for
  gold nanowires at $\fopt=400$ and $500$~THz.
  The obtained normalized em-chirality measures are lower than for
  silver, which might be explained by the larger imaginary part of the
  relative electric permittivity of gold at these frequencies.
  Also the plasmonic resonances are not as pronounced as for the thin 
  silver nanowires. 
  As a consequence of the even larger imaginary part of the electric
  permittivity of gold at $\fopt=600$ and $700$~THz, the normalized
  em-chirality measures obtained for thin gold nanowires optimized at
  these frequencies are rather small.
  Therefore, we do not show the results for these
  frequencies.~\hfill$\lozenge$ 
\end{example}

\section*{Conclusions}
We have optimized the shape of thin free-form silver and gold
nanowires to maximize their electromagnetic chirality at particular
frequencies in the optical band. 
Our gradient based shape optimization scheme uses an asymptotic
representation formula for scattered electromagnetic fields due to
thin bended and twisted metallic nanowires with arbitrary
cross-sections to efficiently evaluate the objective functionals and
their shape derivatives in each iteration. 
We have extended the theoretical foundation of this asymptotic
representation formula from dielectric materials to noble metals with
complex-valued electric permittivities with negative real and positive
imaginary parts.

We have demonstrated that the optimized free-form silver and gold
nanowires yield significantly enhanced chiral responses, when compared
to traditional helical designs. 
Values larger than $90$\% of the maximum possible value of
electromagnetic chirality are obtained for optimized silver nanowires
with elliptical cross-section across the whole optical band. 
We have observed an interesting connection between plasmonic
resonances (and nearby local minima of the total interaction
cross-section) and extremal values of electromagnetic chirality for
thin metallic nanowires at optical frequencies.  
These results are relevant in practice, since the optimized free-form
silver and gold nanowires may serve as building blocks of novel
metamaterials with an increased chiral response.
However, due to the small thickness of our optimized nanowires,
the fabrication of these metamaterials might be challenging.

\section*{Acknowledgments}
Funded by the Deutsche Forschungsgemeinschaft (DFG, German Research
Foundation) -- Pro\-ject-ID 258734477 -- SFB 1173.

\appendix
\section{Polarization tensor bounds}
\renewcommand{\thesection}{\Alph{section}}
We show two pointwise bounds for the electric polarization
tensor~$\Meps$ that are similar to the bounds that have been
established in \cite[Thm.~1]{CapVog03} but valid 
for complex valued relative electric permittivities~$\epsr$ with
negative real and positive imaginary part. 
Since we will modify the arguments used in the proof of
\cite[Thm.~1]{CapVog03}, we work with the following definition of the 
electric polarization tensor $\Meps$, which is equivalent
to~\eqref{eq:CorrectorPotential}--\eqref{eq:DefPolTen} (see
\cite[Lmm.~1]{CapVog06}). 
For any $\bfxi\in\Sd$, we denote by
$V_\rho^{(\bfxi)}\in H^1_\diamond(\BR)$ the corrector potential
satisfying 
\begin{equation}
  \label{eq:CorrectorPotentialAlt}
  \div\bigl(\eps_\rho\nabla V_\rho^{(\bfxi)}\bigr) 
  \,=\, 0 \quad \text{in } \BR \,, \qquad
  \eps_\rho \frac{\di V_\rho^{(\bfxi)}}{\di\bfnu} 
  \,=\, \eps_\rho\, \bfnu\cdot\bfxi \quad \text{on } \di\BR \,,
\end{equation}
where $H^1_\diamond(\BR)$ denotes the space of $H^1$-functions on $\BR$
with vanishing integral mean on $\di\BR$. 
Then, the electric polarization tensor $\Meps\in L^2(\Gamma,\Cdd)$ is
uniquely determined by 
\begin{equation}
  \label{eq:DefPolTenAlt}
  \frac{1}{|\Gamma|} \int_\Gamma \bfxi \cdot \Meps \bfxi \, \psi \ds
  \,=\, \frac{1}{|D_\rho|} \int_{D_\rho} \bigl( \bfxi \cdot 
  \nabla V^{(\bfxi)}_\rho \bigr) \, \psi \dx 
  + o(1) \qquad \text{as } \rho\to0 
\end{equation}
for all $\psi \in C(\ol{\BR)})$ and any $\bfxi\in\Sd$.  

\begin{lemma}
  \label{lmm:PolTenBounds}
  Suppose that $\eps_r\in\C$ with $\real(\eps_r)<0$ and
  $\imag(\eps_r)>0$, and for any $0<\rho<\rGamma$ let $\Drho\tm\Rd$
  and $\eps_\rho$ be as in \eqref{eq:DefDrho} and \eqref{eq:epsrho},
  respectively.
  As before, we write $\eps_m := \eps_0\eps_r$ for the electric
  permittivity of the thin metallic nanowire. 
  \begin{itemize}
  \item[(a)] The electric polarization tensor
    $\Meps \in L^2(\Gamma,\Cdd)$ is symmetric, i.e., 
    \begin{equation}
      \label{eq:PoltenSymmetry}
      \Meps_{ij}(\bfx)
      \,=\, \Meps_{ji}(\bfx) \qquad  \text{for } 1\leq i,j\leq 3
      \text{ and a.e.\ } \bfx\in\Gamma \,.
    \end{equation}
  \item[(b)] We have that
    \begin{equation}
      \label{eq:InequImag}
      \frac{\bfxi \cdot  \imag  \bigl((\ol{\eps_m-\eps_0})
        \Meps(\bfx) \bigr) \bfxi}{\imag \left(\ol{\eps_m-\eps_0}\right)}
      \,\leq\, |\bfxi|^2 
      \qquad \text{for every } \bfxi\in\Sd
      \text{ and a.e.\ } \bfx\in\Gamma \,.
    \end{equation}
  \item[(c)] There exists $\gamma \in (0, \pi/2)$ such that
    $\real(\rme^{\rmi \gamma}\ol{\eps_0}) >0$,
    $\real(\rme^{\rmi \gamma}\ol{\eps_m}) >0 $, and
    $\real(\rme^{\rmi \gamma}(\ol{\eps_m-\eps_0}))<0$.
    For every~$\gamma$ with these properties, we have that  
    \begin{equation}
      \label{eq:InequReal}
      \frac{\bfxi \cdot \real\bigl( \rme^{\rmi \gamma}(\ol{\eps_m - \eps_0})
        \Meps(\bfx) \bigr) \bfxi}{\real \left( \rme^{\rmi \gamma}(
          \ol{\eps_m-\eps_0}) \right)}
      \,\geq\, |\bfxi|^2
      \qquad \text{for every } \bfxi\in\Sd
      \text{ and a.e.\ } \bfx\in\Gamma \,.
    \end{equation}
  \end{itemize}
\end{lemma}

\begin{proof}
  (a)\, Let $\bfxi\in\Sd$.
  We denote by $V_\rho^{(\bfxi)}\in H^1_\diamond(\BR)$ the
  corresponding solution to \eqref{eq:CorrectorPotentialAlt}. 
  Moreover, we define $V_0^{(\bfxi)}\in H^1_\diamond(\BR)$ by
  $V_0^{(\bfxi)}(\bfx) := \bfx\cdot\bfxi$ for all $\bfx\in\BR$. 
  Then, $V_0^{(\bfxi)}$ solves \eqref{eq:CorrectorPotentialAlt} with
  $\eps_\rho$ replaced by~$\eps_0$.
  
  It can be seen as on p.~169 of \cite{CapVog03} that, for any two
  $\bfxi,\bfeta\in\Sd$ and for all $\psi\in C(\ol{\BR})$, 
  \begin{equation*}
    \frac1{|\Drho|} \int_{\Drho}
    (\eps_0-\eps_m) \nabla V_0^{(\bfxi)}
    \cdot \nabla V_\rho^{(\bfeta)} \,\psi \dx
    \,=\, \frac1{|\Drho|} \int_{\Drho}
    (\eps_0-\eps_m) \nabla V_0^{(\bfeta)}
    \cdot \nabla V_\rho^{(\bfxi)} \,\psi \dx
    + o(1) 
  \end{equation*}
  as $\rho\to 0$.
  Substituting this into \eqref{eq:DefPolTenAlt} gives
  \begin{equation*}
    \frac{1}{|\Gamma|} \int_\Gamma (\eps_0-\eps_m)
    \bfxi \cdot \Meps \bfeta \, \psi \ds
    \,=\, \frac{1}{|\Gamma|} \int_\Gamma (\eps_0-\eps_m)
    \bfeta \cdot \Meps \bfxi \, \psi \ds
    + o(1) \,,
  \end{equation*}
  which implies \eqref{eq:PoltenSymmetry}.
  
  (b)\, The same calculation as on p.~170 of \cite{CapVog03}
  shows that, for any $\bfxi\in\Sd$, 
  \begin{multline}
    \label{eq:ProofBounds1}
    \frac{1}{|\Gamma|} \int_\Gamma (\ol{\eps_m-\eps_0})
    \bfxi \cdot \Meps \bfxi \,\psi \ds
    \,=\, \frac1{|\Drho|} \int_{\Drho} (\ol{\eps_m-\eps_0})
    \nabla V_\rho^{(\bfxi)}
    \cdot \nabla V_0^{(\bfxi)} \,\psi \dx\\
    \,=\, \frac1{|\Drho|} \int_{\Drho} (\ol{\eps_m-\eps_0})
    |\nabla V_0^{(\bfxi)}|^2 \,\psi \dx
    - \frac1{|\Drho|} \int_\BR \ol{\eps_\rho}
    |\nabla(V^{(\bfxi)}_0-V^{(\bfxi)}_\rho)|^2 \,\psi \dx
    + o(1) 
  \end{multline}
  as $\rho\to 0$.
  From now on we consider non-negative test functions $\psi\in
  C(\ol{\BR})$, i.e., $\psi \geq 0$. 
  Taking the imaginary part on both sides of \eqref{eq:ProofBounds1} and
  recalling that $\imag(\eps_\rho)\geq0$ gives
  \begin{equation}
    \label{eq:ProofBounds2}
    \frac{1}{|\Gamma|} \int_\Gamma \bfxi \cdot
    \imag \bigl( (\ol{\eps_m-\eps_0})
    \Meps \bigr) \bfxi \,\psi \ds
    \,\geq\, \frac1{|\Drho|} \int_{\Drho}
    \imag \bigl(\ol{\eps_m-\eps_0}\bigr)
    |\nabla V_0^{(\bfxi)}|^2 \,\psi \dx
    + o(1) 
  \end{equation}
  as $\rho\to 0$. 
  Passing to the limit on the right hand side of 
  \eqref{eq:ProofBounds2} then yields \eqref{eq:InequImag}.

  (c)\, Writing $\ol{\eps_m} = |\eps_m|\rme^{\rmi \alpha}$ and
  $\ol{\eps_m-\eps_0} = |\eps_m-\eps_0|\rme^{\rmi \beta}$ in polar form,
  our assumptions on $\eps_0$ and $\eps_m$ imply that
  $\alpha,\beta \in (\pi,3\pi/2)$ with $\alpha>\beta$.
  Accordingly, choosing $\gamma = 3\pi/2 - (\alpha+\beta)/2$, we obtain
  that $\gamma \in (0, \pi/2)$ with
  $\real(\rme^{\rmi \gamma}\ol{\eps_0}) >0$,
  $\real(\rme^{\rmi \gamma}\ol{\eps_m}) >0 $, and
  $\real(\rme^{\rmi \gamma}(\ol{\eps_m-\eps_0}))<0$.
  Multiplying both sides of \eqref{eq:ProofBounds1} with $\rme^{\rmi\gamma}$
  and taking the real part gives
  \begin{equation}
    \label{eq:ProofBounds3}
    \frac{1}{|\Gamma|} \int_\Gamma \bfxi \cdot
    \real\bigl( \rme^{\rmi\gamma} (\ol{\eps_m-\eps_0})
    \Meps \bigr) \bfxi \,\psi \ds
    \,\leq\, \frac1{|\Drho|} \int_{\Drho}
    \real\bigl( \rme^{\rmi\gamma}(\ol{\eps_m-\eps_0}) \bigr)
    |\nabla V_0^{(\bfxi)}|^2 \,\psi \dx
    + o(1) 
  \end{equation}
  as $\rho\to 0$. 
  Passing to the limit on the right hand side of
  \eqref{eq:ProofBounds3} then yields \eqref{eq:InequReal}.
\end{proof}

{\small
  \bibliographystyle{abbrv}
  \bibliography{metallic_chirality}
}

\end{document}